\documentclass[a4paper,12pt,twoside]{amsart}
\DeclareSymbolFont{matha}{OML}{txmi}{m}{it}
\DeclareMathSymbol{\varv}{\mathord}{matha}{118}
\usepackage[OT2,T1]{fontenc}
\usepackage{mathtools}
\usepackage[T1]{fontenc}
\usepackage{vmargin}
\usepackage{amssymb}
\usepackage{datetime}
\usepackage{amsmath}
\numberwithin{equation}{subsection}
\usepackage{amsthm}
\usepackage{mathrsfs}
\usepackage{vmargin}
\usepackage{color}
\usepackage{amscd}
\usepackage{stmaryrd} 
\newtheorem{theorem}{Theorem}[section]
\newtheorem{lemma}[theorem]{Lemma}
\newtheorem{proposition}[theorem]{Proposition}
\newtheorem{definition}[theorem]{Definition}{\rm}
\newtheorem{corollary}[theorem]{Corollary}
\newtheorem{remark}[theorem]{Remark}
\newtheorem{claim}[theorem]{Claim}

\newtheorem{conjecture}[theorem]{Conjecture}

\newtheorem*{nb}{\footnotesize {N.B}}

\def\tab{\leaders\hbox to 1.5mm{\hfil.\hfil}\hfill}

\def\p{\partial}
\def\no{\noindent}
\def\io{{\infty}}
\def\curl{\operatorname{curl}}

\def\range{\operatorname{ran}}

\def\re{\operatorname{Re}}

\def\im{\operatorname{Im}}
\def\Id{\operatorname{Id}}
\def\Lg{\operatorname{Log}}
\def\dive{\operatorname{div}}
\def\grad{\operatorname{grad}}
\def\moo{C^{\io}}
\def\mooc{C^{\io}_{\textit c}}

\def\N{\mathbb N}

\def\R{\mathbb R}
\def\C{\mathbb C}
\def\rot{\mathbf C}
\def\poscal#1#2{\langle#1,#2\rangle}
\def\proscald#1#2{<\!\!#1, #2\!\!>_{_{\hskip-4pt \R^{d}}}}
\def\proscal3#1#2{<\!\!#1, #2\!\!>_{_{{\hskip-4pt\R^{3}}}}}

\def\norm#1{\Vert#1\Vert}
\def\trinorm#1{\vert\hskip-1pt\Vert #1\Vert\hskip-1pt\vert}
\def\val#1{\vert#1\vert}
\def\Val#1{\left\vert#1\right\vert}
\def\valjp#1{\langle#1\rangle}

\def\l2{L^2(\R^{n})}
\def\L2{L^2(\R^{2n})}
\def\supp{\operatorname{supp}}
\def\spec{\operatorname{spect}}

\def\tr#1{{^t}#1}

\def\trace{\operatorname{trace}}

\def\vs{\vskip.3cm}

\let \dis=\displaystyle
\let\no=\noindent
\let \dis=\displaystyle

\let\no=\noindent
\def\mat22#1#2#3#4{\begin{pmatrix}#1&#2\\ #3&#4\end{pmatrix}}

\def\mattre#1#2#3{\begin{pmatrix}#1\\ #2\\#3\end{pmatrix}}

\def\beq{\begin{equation}}
\def\eeq{\end{equation}}

\def\tr#1{{^t}\!\!#1}

\def\dd{\texttt{d}}

\def\leray{\mathbb P}
\def\pleray{\widetilde{\mathbb P}}

\makeatletter
\@namedef{subjclassname@2020}{%
  \textup{2020} Mathematics Subject Classification}
\makeatother
\begin{document}
\baselineskip=1.1\normalbaselineskip
\title[Wiener Algebras Methods for Liouville theorems]{Wiener Algebras Methods for Liouville theorems\\ on the stationary Navier-Stokes system}
\author[Nicolas Lerner]{Nicolas Lerner}
\address{\noindent \textsc{N. Lerner, Institut de Math\'ematiques de Jussieu,
Sorbonne Universit\'e ,
Campus Pierre et Marie Curie,
4 Place Jussieu,
75252 Paris cedex 05,
France}}
\email{nicolas.lerner@imj-prg.fr, nicolas.lerner@sorbonne-universite.fr}
\numberwithin{equation}{subsection}
\begin{abstract}
We prove some  Liouville theorems for the stationary Navier-Stokes system for incompressible fluids.
We provide some sufficient conditions on the low frequency part of the solution,
using some properties of classical singular integrals with respect to Wiener algebras.
\end{abstract}
\vfill
\maketitle
\centerline{\color{magenta}\boxed{\text{\bf VERSION 1, \today, \currenttime
}}}
{\color{red}\footnotesize\tableofcontents}
\vfill\eject
\section{\color{magenta}Introduction}
\subsection{The classical Liouville theorem}
\begin{theorem}[Joseph Liouville, 1844]\label{thm.001}
Let $f$ be a bounded harmonic function on $\R^{d}$. Then $f$ is a constant function.
 \end{theorem}
 \begin{proof}
 Since $f$ is bounded, it is a tempered distribution and then we have,
 since polynomials are multipliers of $\mathscr S'(\R^{d})$,  
 $$
 \val\xi^{2}\widehat f(\xi)=0,
 $$
 where $\widehat f$ stands for the Fourier transform of $f$.
This  implies that $\supp{\widehat f}\subset\{0\}$ and thus that $f$ is a polynomial.
 Since $f$ is assumed to be bounded, this implies that $f$ is a constant, which is the sought result.
 \end{proof}
 \begin{remark}\label{rem.000}\rm
We have proven in particular that
 a tempered distribution  which is harmonic is a polynomial\footnote{
 Of course, 
 some harmonic functions are not in $\mathscr S'(\R^{d})$ (when $d\ge 2$) such as 
 $f(x_{1},x_{2})=e^{x_{1}}\sin x_{2}$, which is a smooth function, thus a Schwartz distribution, but not a tempered distribution.
 }.
 \end{remark}
 \begin{remark}\label{rem.001}\rm
 Note that if $f$ is assumed to be harmonic and  such that there exists $m\in \R_{+}$ such that
 $$
 \val{f(x)}(1+\val x)^{-m}\in L^{\io}(\R^{d}),
 $$
 the same proof as above  gives that 
 $f$ is a polynomial of degree $\le m$.
 Also, if $f$ is a harmonic function whose gradient belongs to some $L^{p}(\R^{d})$ space with $p\in[1,+\io)$,
 or to $L^{\io}(\R^{d})$ with limit $0$ at infinity,
 then the previous proof gives that $\nabla f$ is a polynomial
 which must be $0$ (no non-zero polynomial could belong to $L^{p}$ with $p$ finite,
 only constant polynomials belong to $L^{\io}$; in the latter case the requirement that the limit of the gradient is zero at infinity forces $\nabla f=0$). This proves that $f$ is a constant function.
\end{remark}
Of course, it looks  quite natural  to expect some similar results for functions in the kernel of linear  elliptic operators, but even for constant coefficients operators, these Liouville-type results \emph{do depend} on the lower order terms and ellipticity alone does not suffice to secure such results.
Let us for instance check,
\begin{equation}\label{}
P=-\Delta+i\p_{x_{1}}, \quad\text{the Fourier multiplier}\ \val \xi^{2}-\xi_{1}.
\end{equation}
We note that any function on $\R^{d}$ with $d\ge 2$ given by 
$$
f(x_{1}, x_{2}, \dots, x_{d})= e^{ix_{1}} g(x_{2}, \dots, x_{d}), \quad\text{with $g$ harmonic on $\R^{d-1}$,}
$$
belongs to the kernel of $P$. In particular  the choice $g=1$ gives $Pf=0$ with $f$ bounded non-constant.
We may also notice that for the operator $-\Delta+\p_{x_{1}}$,
 a Liouville-type theorem holds true, i.e. bounded solutions in its kernel are constant:
 indeed, repeating the proof of Theorem \ref{thm.001},
 we find that a bounded solution $f$ of 
 $-\Delta f+\p_{x_{1}}f$
 should satisfy
 $$
 \bigl(\val\xi^{2}+i\xi_{1}\bigr) \widehat f(\xi)=0,
 $$
 so that $\supp \widehat f\subset\{0\}$ and thus $f$ is a polynomial and a bounded function, thus a constant. The quite naive examples above show that some reasonable conditions should be set-up for lower order terms and we give some details below on this matter.
 \par
 A good question that could be raised for linear elliptic operators would be the following.
 Let 
 \begin{equation}\label{linope}
P=-\Delta+X,
\end{equation}
 be a linear operator on $\mathbb R^{d}$ where $X$ is a real-valued vector field with null divergence and coefficients in $L^{p}(\R^{d})$ for some $p\in (1,+\io)$. 
 We may raise the following question: {\sl Is there an explicit condition on $p$  ensuring that  bounded solutions $f$ vanishing at infinity of 
 \begin{equation}\label{equ001}
 -\Delta f+X f=0,
\end{equation}
are actually 0?
} For instance we could assume that for some $p\in [1,+\io)$, 
 \begin{align}
& X \text{ is a real vector field with null divergence and coefficients in $L^{p}(\R^{d})$,}\label{lio001}\\
& \nabla f \in L^{2}(\R^{d}),\label{lio002}\\
& \text{$f$ is bounded with limit 0 at infinity.}\label{lio003}
\end{align}
Then for $d\ge 3$, we get easily that $f$ belongs to $L^{q}(\R^{d})$,
with 
\begin{equation}\label{ind987}
q=\frac{2d}{d-2},\quad \text{i.e.,}\quad \frac1q=\frac12-\frac1d.
\end{equation}
Considering a function 
\begin{equation}\label{fun001}
\chi\in \mooc(\R^{d}, [0,1]),\quad \text{supported in $2\mathbb B^{d}$ and equal to 1 on
$\mathbb B^{d}$,}
\end{equation}
 we define for $\lambda>1$, 
\begin{equation}\label{fun002}
\chi_{\lambda}(x)=\chi\Bigl(\frac x\lambda\Bigr),\text{ noting that $\nabla \chi_{\lambda}$ is bounded in $L^{d}(\R^{d})$,}
\end{equation}
since $\norm{\nabla\chi_{\lambda}}_{L^{d}(\R^{d})}=\norm{\nabla\chi}_{L^{d}(\R^{d})}$.
Let us take for granted that ellipticity of the Laplace operator  implies that the solutions of \eqref{equ001} are smooth functions,
so that we can multiply the equation by the smooth and compactly supported $\chi_{\lambda}f$ and integrate by parts. 
We find (assuming as we may that $f$ is real-valued)\footnote{We use  in \eqref{123456} that
with $\proscald{a}{b}$ standing for the canonical dot product in $\R^{d}$,
\begin{align*}
\int_{\R^{\mathtt d}}&\proscald{X(x)}{(\nabla f)(x)}\chi_{\lambda}(x) f(x) dx
\\
&=- \int_{\R^{\mathtt  d}}
f(x)    \proscald  {X(x)}    {      \bigl(\nabla(\chi_{\lambda}f)\bigr)(x)    }         dx
- \int_{\R^{\mathtt  d}} f(x)(\dive X)(x)(\chi_{\lambda}f)(x) dx
\\
\text{\tiny (using $\dive X=0$)}&=-\int_{\R^{\mathtt d}}\proscald{X(x)}{(\nabla f)(x)}\chi_{\lambda}(x) f(x) dx
-\int_{\R^{\mathtt d}}f(x)^{2}\proscald{X(x)}{(\nabla \chi_{\lambda})(x)} dx.
\end{align*}
\label{foot02}}
\begin{align}
0&=\poscal{-\Delta f+Xf}{\chi_{\lambda}f}\notag\\
&=\int \chi_{\lambda}(x)\val{(\nabla f)(x)}^{2}dx
+\poscal{\underbrace{\nabla f}_{L^{2}}}{\underbrace{(\nabla \chi_{\lambda})}_{L^{d}}\underbrace {f}_{L^{q}}}
-\frac12\poscal{\underbrace{X\cdot (\nabla\chi_{\lambda})}_{L^{p}\cdot L^{d}} }
{\underbrace{f^{2}}_{L^{q/2}}}.
\label{123456}
\end{align}
Note that we have used for the calculation of the last term above  that $f$ is real-valued; for a linear equation, it is a harmless assumption since we may separate real and imaginary parts,
but since we are contemplating non-linear cases where $Xf$ may be quasi-linear,
we already specified \textit{ab initio} that $f$ is real-valued. We  have 
$
1=\frac12+\frac1d+\frac1q,
$
where $q$ is defined by \eqref{ind987},
and  we need to check
$$
1=\frac1p+\frac1d+\frac2q,\text{ which is equivalent to }p=d.
$$
We have the estimate
\begin{multline}\label{hgf541}
\val{\varepsilon_{1}(\lambda)=\poscal{\underbrace{\nabla f}_{L^{2}}}{\underbrace{(\nabla \chi_{\lambda})}_{L^{d}}\underbrace {f}_{L^{q}}}}
\\\le \norm{\nabla f}_{L^{2}(\{\lambda\le \val x\le 2\lambda\})}
\norm{\nabla \chi_{1}}_{L^{d}}\norm{f}_{L^{q}(\{\lambda\le \val x\le 2\lambda\})},
\end{multline}
so that 
$\lim_{\lambda\rightarrow+\io}\varepsilon_{1}(\lambda)=0$ and 
\begin{equation}\label{hgf54+}
\val{\varepsilon_{2}(\lambda)=\poscal{\underbrace{X\cdot (\nabla\chi_{\lambda})}_{L^{p}\cdot L^{d}} }
{\underbrace{f^{2}}_{L^{q/2}}}}
\le \norm{X}_{L^{p}(\{\lambda\le \val x\le 2\lambda\})}
\norm{\nabla \chi_{1}}_{L^{d}}\norm{f}_{L^{q}(\{\lambda\le \val x\le 2\lambda\})}^{2},
\end{equation}
which implies that 
$\lim_{\lambda\rightarrow+\io}\varepsilon_{2}(\lambda)=0$.
As a consequence we find that 
\eqref{equ001}, \eqref{lio001}, \eqref{lio002}, \eqref{lio003} with $p=d$
imply $f=0$.
We have thus proved the following result.
\begin{proposition}\label{pro.dq98}
 Let $X$ be a real divergence-free\footnote{We may weaken that hypothesis (and keep the conclusion) by assuming that $(\dive X)_{+}$ belongs to $L^{\frac d2}(\R^{d})$ such that 
 $$
 \norm{(\dive X)_{+}}_{L^{\frac d 2}(\R^{d})}<2\sigma(d),
 $$
 where $\sigma(d)$ is a dimensional constant. It is possible to identify the constant $\sigma(d)$ as the largest constant such that
 $$
 \sigma(d)\norm{f}_{L^{\frac{2d}{d-2}}(\R^{d})}^{2}\le \norm{\nabla f}_{L^{2}(\R^{d})}^{2}.
 $$
 Note that this assumption holds true in particular if we assume that $\dive X\le 0$.
 A proof of that improvement is given in our Appendix with  Proposition \ref{app.pro.dq99}.} vector field in $\R^{d}$ ($d\ge 3$) with $\moo$ coefficients in $L^{d}(\R^{d})$
  and let $f$ be a solution of \eqref{equ001}
 such that  \eqref{lio002}, \eqref{lio003} hold true. Then $f$ is identically 0.
\end{proposition}
This is a result \emph{at the scaling} and we may wonder if condition \eqref{lio001} could be modified.
The book \cite{MR4319031} by C. Le Bris \& P.-L. Lions
studies the existence and uniqueness of solutions to parabolic-type equations with irregular coefficients and can be used to derive some Liouville-type theorems analogous to Proposition \ref{pro.dq98}, with much less constraints on the regularity of the vector field $X$.
We may also consider \emph{non-linear} elliptic equations and raise roughly the same questions as above: assuming that we have a bounded solution vanishing at infinity,
which additional conditions on the solution are sufficient to ensure that it vanishes identically?
In this paper, we want to focus our attention on the Stationary Navier-Stokes System for Incompressible Fluids.
\subsection{Liouville theorems for the stationary Navier-Stokes system}
We consider a vector field $v\in L^{2}_{\text{loc}}(\R^{3},\R^{3})$ satisfying the Stationary Navier-Stokes System for Incompressible Fluids,
\begin{equation}\label{SNSI}
-\nu \Delta{v}+\sum_{1\le j\le 3}\p_{j}(v_{j}v)+\nabla p=0,\quad  \dive v=0,
\end{equation}
where the parameter $\nu>0$ stands for the kinematic viscosity, the \emph{pressure} $p$ is a scalar distribution on $\R^{3}$ and the $v_{j}$ are the components of the vector field\footnote{Note also that the physical unit of $\nu$ is the{ \sc Stokes}: we write
$[\nu]=\mathtt{L^{2}T^{-1}}$; we check easily that 
\begin{equation*}
[\nu \Delta v]=\mathtt{L^{2}T^{-1}L^{-2}L T^{-1}}=\mathtt{LT^{-2}},
\quad
[\p_{j}(v_{j}v)]=\mathtt{L^{-1}(LT^{-1})^{2}}=\mathtt{LT^{-2}}.
\leqno{(\sharp)}\end{equation*}
Also since $p$ is the pressure divided by the density, we have
$$
[p]=\frac{\mathtt{ MLT^{-2}}}{\mathtt{L^{2}}}\frac{1}{\mathtt{ML^{-3}}}=\mathtt{ L^{2}T^{-2}}, \quad [\nabla p]=\mathtt{LT^{-2}},
\text{\quad which is consistent with \eqref{SNSI} and $(\sharp)$.}
\leqno{(\natural)}$$
\label{foot01}} $v$.
We notice that the products $v_{j}v$ belong to $L^{1}_{\text{loc}}(\R^{3},\R^{3})$, so that their (distribution) derivatives make sense.
We shall always assume that
\begin{align}
&v\in L^{1}_{\text{loc}}(\R^{3}, \R^{3})\cap\mathscr S'(\R^{3}, \R^{3}) \text{ and }
\forall s>0, \ \Val{\{x\in \R^{3}, \val{v(x)}>s\}}<+\io,
\label{hyp001}
\\
&\curl v\in L^{2}(\R^{3}, \R^{3}).
\label{hyp002}
\end{align}
Here \eqref{hyp001} stands for a quite weak version of  \textit{``$v$ tends to zero at infinity''}. We refer the reader to our Definition \ref{def.417ooo} in our Appendix and various properties linked to that 
definition.
A (much) stronger version of \eqref{hyp001} is the condition,
\begin{equation}\label{hyp001+++}
v\in L^{\io}(\R^{3}, \R^{3}), \quad \lim_{R\rightarrow+\io}\norm{v}_{L^{\io}(\{\val x\ge R\},\R^{3})}=0.
\end{equation}
In fact, introducing the condition
\begin{equation}\label{hyp001plu}
v\in L^{1}_{\text{loc}}, 
\exists R_{0}\ge 0 \text{ s.t. } v\in L^{\io}(\{\val x\ge R_{0}\}) \text{ and }
\lim_{R\rightarrow+\io}\norm{v}_{L^{\io}(\{\val x\ge R\})}=0,
\end{equation}
we obtain easily that
\begin{equation}\label{541hgx}
\eqref{hyp001+++}\Longrightarrow\eqref{hyp001plu}\Longrightarrow\eqref{hyp001}.
\end{equation}
Also in our Remark \ref{rem.54potr} below, we formulate a couple of comments on the relevance of our choice.
\vs
The main conjecture on the topic discussed in that article  is the following
\begin{conjecture}\label{coj001}
 Let  $v$ be a vector field on $\R^{3}$
 such that
 \eqref{SNSI}, \eqref{hyp001}, \eqref{hyp002} are satisfied.
 Then $v$ is identically equal to $0$. 
\end{conjecture}
When this paper is being written (\today), this conjecture is still an open problem.
Conjecture \ref{coj001} is saying
that a stationary solution of the Navier-Stokes system for incompressible fluids which vanishes at infinity and is such that $\curl v$ belongs to $L^{2}$ should be trivial
(i.e. must be identically $0$). Condition \eqref{hyp002} originates from J. Leray's foundational work 
\cite{MR1555394}.
We shall see that various additional conditions on the vector field $v$ can lead to the result $v\equiv 0$, and we wish to summarize in the sequel of this Introduction the state of the art
on these matters.
The consistency of units is of key importance and is addressed in Footnote \ref{foot01}.
Another invariance is related to the scaling of the equation, namely 
\begin{claim}
Let $v, p$ be solutions of \eqref{SNSI} and let $\lambda\in \R$. Defining
\begin{equation}\label{126126}
w(x)=\lambda v(\lambda x), \quad q(x)=\lambda^{2}p(\lambda x),
\end{equation}
we find that \eqref{SNSI} holds true with $v=w, p=q$.
\end{claim}
\begin{proof}[Proof of the Claim] We find easily that
\begin{align*}
&(\Delta w)(x)=\lambda^{3}(\Delta v)(\lambda x),\quad
\bigl(\p_{j}(w_{j}w)\bigr)(x)=\lambda^{3}
\bigl(\p_{j}(v_{j}v)\bigr)(\lambda x),
\\
&(\nabla q)(x)=\lambda^{3}(\nabla p)(\lambda x),
\end{align*}
proving the sought result.
\end{proof}
\begin{remark}\label{rem.54potr}\rm 
The reader may wonder what is the actual ground for choosing Condition \eqref{hyp001} instead of the simpler and somewhat natural \eqref{hyp001+++}: a first answer to that self-raised query is that 
\eqref{hyp001} is weaker, which is strengthening the conjecture and the results using as hypotheses only \eqref{hyp001} and not \eqref{hyp001+++}.
Also the reader could check in our Appendix with  Lemma \ref{lem.jhgf}
that a function which is $L^{p}$ ($1\le p<+\io$) on a neighborhood of infinity  has limit $0$ at infinity in the sense of Definition \ref{def.417ooo},
but does not satisfy  \eqref{hyp001+++}  in general.
Another reason for that choice seems to be more important:
looking at the conditions 
\eqref{SNSI}-\eqref{hyp001}-\eqref{hyp002}, we see that they contain only two parameters, the viscosity $\nu$ and 
\begin{equation}\label{}
\gamma_{0}=\norm{\curl v}_{L^{2}(\R^{3}, \R^{3})}, \quad[\gamma_{0}]={ \mathtt{T^{-1} L^{3/2}}}.
\end{equation}
We shall be able to \emph{deduce} \eqref{hyp001+++} from the assumptions of the conjecture and to prove in particular that $v$ belongs to the Wiener algebra $\mathcal F\bigl(L^{1}(\R^{3}, \R^{3})\bigr)$ and thus to $L^{\io}$,
so that  the $L^{\io}$ norm of $v$ will be bounded from above by a quantity depending only on $\nu$ and $\gamma_{0}$. This is indeed a much better situation than having to introduce the $L^{\infty}$-norm of $v$ as a parameter of our problem.
\end{remark}
\begin{remark}\label{rem.54kj}\rm
Let $\psi:\R^{3}\longrightarrow \R$ be a harmonic function (thus a smooth function)  and let us consider the vector field $v=\nabla \psi$. Then we have $\dive v=\Delta\psi=0$ and 
\begin{align*}
&\sum_{1\le j\le 3}\p_{j}(v_{j}v)=
\sum_{1\le j\le 3}(\p_{j}^{2}\psi) \nabla \psi+\sum_{1\le j\le 3}(\p_{j}\psi)\nabla({\p_{j}\psi})
=\frac12\nabla({\val{\nabla \psi}^{2}}),
\\
&\Delta v=\Delta(\nabla \psi)=\sum_{1\le j\le 3}\p_{j}^{2}(\nabla \psi)=\nabla(\Delta \psi)=0,
\end{align*}
so that for $p=-\frac12{\val{\nabla \psi}^{2}}$, we have 
$
-\nu \Delta{v}+\sum_{1\le j\le 3}\p_{j}(v_{j}v)+\nabla p=0,
$
so that $v=\nabla \psi$ solves \eqref{SNSI}
and verifies \eqref{hyp002} since $\curl v=\curl \nabla \psi=0$.
Now $\nabla \psi$ does not necessarily vanish identically
(take for instance $\psi$ equal to any harmonic polynomial of the three variables $(x_{1}, x_{2}, x_{3})$ with degree $\ge 1$).
So some hypothesis at infinity is needed and in this particular case, if we know that $v=\nabla \psi$ is bounded and satisfies \eqref{hyp001},
we get, thanks to Remark \ref{rem.001}, that $\psi$ should be constant and thus $v=0$.\footnote{If $\psi(x)=\frac12\poscal{Q x}{x}+\poscal{\eta}{x}+y$, where $Q$ is a  symmetric $3\times 3$ matrix with null trace, we find also that $v=Qx+\eta$
solves \eqref{SNSI} with $p=-\frac12\val{\nabla \psi}^{2}$, satisfies  \eqref{hyp002}, belongs to $L^{1}_{\text{loc}}\cap \mathscr S'$ but satisfies \eqref{hyp001} only if $v=0$.
More generally if $\psi$ is a harmonic function, we have seen that $v=\nabla \psi$
solves \eqref{SNSI} with $p=-\frac12\val{\nabla \psi}^{2}$, satisfies  \eqref{hyp002}, belongs to $L^{1}_{\text{loc}}$;
thanks to Remark \ref{rem.000} if $\psi$ belongs to $\mathscr S'$, it is a polynomial and $v$ is also a polynomial so that it could satisfy \eqref{hyp001} only if $v=0$ (cf. Lemma \ref{cla.548uyt} in our Appendix).
Note that if $\psi$ is a harmonic function which is not a temperate distribution (e.g $e^{x_{1}}\sin x_{2}$),
$v=\nabla \psi$
solves \eqref{SNSI} with $p=-\frac12\val{\nabla \psi}^{2}$, satisfies  \eqref{hyp002}, belongs to $L^{1}_{\text{loc}}$
but does not satisfy \eqref{hyp001}.
}
\end{remark}
\subsection{Known results}
In the first place, an important regularity result is due to G.P.~Galdi in the reference book  \cite{MR2808162}.
\begin{theorem}[G.P.~Galdi]\label{thm.known1}
 Let $v$ be a vector field satisfying the assumptions of Conjecture \ref{coj001}.
 Then $v$ is $C^{\io}$,
 with all its derivatives bounded with limit zero at infinity.
 \end{theorem}
 Let us make a short list of known additional conditions ensuring the identical vanishing of $v$.
\emph{Let $v$ be a vector field satisfying the assumptions of Conjecture \ref{coj001}.
}
\begin{enumerate}\label{thm.known2}
\item\label{111kno} Assuming that $v$ belongs to $L^{\frac92}(\R^{3}, \R^{3})$ implies $v\equiv 0$,
a result due to G.P.~Galdi in  \cite{MR2808162}.
\item\label{222kno} Assuming that  $\int_{\R^{3}}{\val{v(x)}^{\frac92}}\bigl[\ln(2+\val{v(x)}^{-1})\bigr]^{-1} dx<+\io$, implies that $v\equiv 0$, 
a logarithmic improvement of the previous result,
due to D.~Chae \& J.~Wolf in 
\cite{MR3548261}.
\item\label{333kno} Assuming that $\Delta v$ belongs to $L^{6/5}(\R^{3}, \R^{3})$ implies $v\equiv 0$,  a result due to D.~Chae in \cite{MR3162482}. For this result, it is not necessary to assume \eqref{hyp002}; moreover, the reader may note that the homogeneity of this hypothesis is the same as \eqref{hyp002}.
In particular, we have with $v_{\lambda}=w$ in \eqref{126126},
$$
\frac{\norm{\curl v_{\lambda}}_{L^{2}}}{\norm{\curl v}_{L^{2}}}=\frac{\norm{\curl^{2} v_{\lambda}}_{L^{6/5}}}{\norm{\curl^{2} v}_{L^{6/5}}}=\lambda^{1/2}.
$$
\item\label{444kno} Assuming that $v=\curl w$, with $w\in \text{BMO}(\R^{3},\R^{3})$ implies $v\equiv 0$,  a result due to 
G. ~Seregin in \cite{MR3538409}.
\item\label{445kno} Assuming that 
$
\norm{v}_{L^{\frac92, \io}}\le \delta\bigl(\nu\norm{\curl v}^{2}_{L^{2}}\bigr)^{1/3}
$
with a small enough $\delta$,
H. Kozono, Y. Terasawa,  and Y. Wakasugi in \cite{MR3571910}
were able to prove that $v\equiv 0$. That result was generalized in the G. Seregin \& W. Wang' article \cite{MR3937507}. 
\item[(6)] Assuming that for some $q,\beta$, such that $q\in(\frac32, 3), \beta>\frac{6q-3}{8q-6}$, we have,
$$
\sup_{R>0} R^{\beta}\left(\frac{1}{\val{B_{R}}}\int_{B_{R}}\val{v(x)}^{q}\right)^{\frac1q}<+\io,
$$
G. Seregin proved in \cite{MR3538409}
that $v\equiv 0$.
\end{enumerate}
Many other results are known, but we can remark at this point that the result \eqref{111kno} in the previous list is much better than the scaling at $L^{3}$, but is still quite far from the natural $L^{6}$ which follows from \eqref{hyp002}. 
Using the notations of Definition \ref{def.145gfd}, we have in three dimensions
 $$
 \underbracket[0.2pt] {L^{3}_{(0)}=\bigcap_{3\le p\le \infty}L^{p}_{(0)}}_{\text{scaling}}\subset \underbracket[0.2pt] {L^{4.5}_{(0)}=\bigcap_{4.5\le p\le \infty}L^{p}_{(0)}}_{\text{Galdi's assumption}}
 \subset {\underbracket[0.2pt]{L^{6}_{(0)}=\bigcap_{6\le p\le \infty}L^{p}_{(0)}}_{\text{standard set}}}.
 $$
 Moreover the same proof as the proof of (1) above works for all dimensions $d\ge 4$,
  assuming only the hypotheses of Conjecture \ref{coj001}, without additional
integrability condition: that conjecture is thus proven in dimension $\ge 4$
(see G.P. Galdi's Chapter X in \cite{MR2808162}) as well as in dimension $2$,
from a paper due to  D.Gilbarg and H.F.Weinberger in \cite{MR501907}.
The \emph{Bernoulli head pressure} $Q$, defined by
 \begin{equation}\label{modpr0}
Q=p+\frac12\val{v}^{2},
\end{equation}
plays a key role in many results.
A most important structural result on the function $Q$ is given by the following result, 
due to D. Chae in \cite{MR3959933}.
\begin{theorem}[Chae's Theorem]\label{thm.chae}
 Let $v$ be a vector field satisfying the assumptions of Conjecture \ref{coj001}.
 Then the Bernoulli head pressure $Q$  takes its values in a compact set $[-M_{0}, 0]$, where $M_{0}\ge 0$.
 Moreover, if $\curl v$ is not identically 0, the head pressure $Q$ is taking only \emph{negative}
 values.
\end{theorem}
Several Liouville theorems are obtained by assuming various properties of the Bernoulli head pressure.
\emph{Let $v$ be a vector field satisfying the assumptions of Conjecture \ref{coj001}.}
\begin{enumerate}
\item[(7)]\label{145kno} Assuming $
 \nabla\sqrt{\val Q}\in L^{2}
$  implies $v\equiv 0$,  a result due to  D.~Chae in \cite{MR4028998}.
Note that this assumption means that
\begin{equation}\label{132ert}
\int_{\R^{3}}\frac{\val{(\nabla Q)(x)}^{2}}{\val{Q(x)}} dx<+\io,
\end{equation}
noting that Theorem \ref{thm.chae} implies that for all $x\in \R^{3}$, $Q(x)<0$, so that Condition \eqref{132ert} makes sense.
\item[(8)] D.~Chae's paper \cite{MR4028998} contains  better results than (7), and in particular, if
\begin{equation}\label{132er+}
\int_{\R^{3}}\frac{\val{(\nabla Q)(x)}^{2}}{\val{Q(x)}\Lg(\frac{eM_{0}}{\val{Q(x)}})} dx<+\io,
\end{equation}
this proves that $v$ must be trivial; here, following Theorem \ref{thm.chae}, we have used the fact that for all $x$ in $\R^{3}$, we have 
$
-M_{0}\le Q(x)<0,
$
securing the fact that
$
\frac{eM_{0}}{\val{Q(x)}}\ge e
$
and thus
$
\Lg\bigl(\frac{eM_{0}}{\val{Q(x)}}\bigr)\ge 1,$
so that we have of course $\eqref{132ert}\Longrightarrow\eqref{132er+}$.
Some more general hypotheses involving iterated logarithms are given in \cite{MR4028998}.
\item[(9)]\label{245kno}  Assuming $
\sup\frac{\val v^{2}}{\val{Q}}<+\io
$  implies $v\equiv 0$,
a result due to  D.~Chae in \cite{MR4205084}.
\item[(10)]\label{345kno} 
Assuming 
$$
\limsup_{\val x\rightarrow+\io}\frac{\val{v}^{2}}{\val{Q}\
{\Lg(\frac{eM_{0}}{\val Q})}
}<+\io,
$$
 implies $v\equiv 0$; this is 
a  logarithmic improvement of (9) in \cite{MR4903793} due to 
J.Bang and Z.Yang.\footnote{Note that one can prove that $\val{Q}>0$ with limit $0$ at infinity, so that the condition in (8) makes sense since     $1/\val{Q}$ tends to $+\io$ at infinity. Also that condition is obviously weaker than the condition in (7).}
\item[(11)]
The paper \cite{MR4719441},
by D. Chamorro \& G. Vergara-Hermosilla is investigating the case of Lebesgue spaces with variable exponents and provide a Liouville theorem for $v$ in that type of space.
\end{enumerate}
The Bernoulli head pressure plays an important r\^{o}le in the papers \cite{MR3275850},
\cite{MR4469405}.
History of this topic within a broader context is also available in G.P. Galdi's Chapter I of \cite{MR2808162}.
In the 
paper
\cite{MR2545826},
by G. Koch,
N. Nadirashvili,
G. Seregin
and  V. \v Sver\'ak, 
several Liouville theorems are proven for axisymmetric solutions with no swirl of the time-dependent Navier-Stokes system; 
the same type of assumption is used in the paper \cite{MR3275850},
by
M.V. Korobkov,
K. Pileckas
and R.~Russo.
We refer also the reader to the
survey paper \cite{MR3833509} by
G.~Seregin and T. Shilkin, and to the book \cite{MR3822765}, authored by T.-P. Tsai.
\subsection{A preliminary formulation of our results}
It seems better to formulate right now some versions of our results, although some technical features
of our paper will need  further clarification. This section is devised to hopefully motivate the reader to go on and to provide a glimpse at our new results.
We write  here the statements of our results and postpone the proofs to Sections \ref{sec.3333} and \ref{sec.new4} of the paper. 
\par
We shall start
with some helpful  regularity results, going a little bit beyond Theorem \ref{thm.known1}, although with the same homogeneity.
Recalling the definition of the Wiener\footnote{A couple of facts on the Fourier transformation are 
given in Section \ref{sec.fourier} of our Appendix.} 
algebra\footnote{Since $L^{1}(\R^{3})$ is a Banach algebra
for \emph{convolution} of functions,
$\mathcal W$ is a Banach algebra
for \emph{multiplication} of functions, and we have, thanks to the Riemann-Lebesgue Lemma,
$$
\mathcal W\subset C^{0}_{(0)}(\R^{3})\subset L^{\io}(\R^{3}),
$$
where the two inclusions are strict embeddings.
{\it Stricto sensu}, we should distinguish between the scalar-valued multiplicative algebra,
$$
\mathcal W(\R^{3}, \R)=\{f\in \mathscr S'(\R^{3}, \R), \hat f\in L^{1}(\R^{3},\C)\},
$$
and the vector-valued multiplicative algebra,
$$
\mathcal W(\R^{3}, \R^{3})=\{v\in \mathscr S'(\R^{3}, \R^{3}), \hat v\in L^{1}(\R^{3},\C)\}.
$$
The reader may note that a relevant choice of multiplication  for $\mathcal W(\R^{3}, \R^{3})$ is the vector product; in that case for $v, w\in \mathcal W(\R^{3}, \R^{3})$, we have 
$$
\mathcal F\bigl(v\times w\bigr)=\hat v\star \hat w,
$$
where the convolution $\star$ of vector fields is defined by
\begin{equation}\label{}
(V\star W)(\xi)=\int \bigl( V(\eta)\times W(\xi-\eta)\bigr) d\eta (2\pi)^{-3/2}.
\end{equation}
}
$$
\mathcal W=\{v\in \mathscr S'(\R^{3}), \widehat v\in L^{1}(\R^{3})\},
$$
we define  for $s\in \R$,
\begin{align}\label{}
\mathcal V^{s,\omega}&=\{v\in \mathcal W, \valjp{D}^{s}v\in \mathcal W\}, \quad\valjp{\xi}^{s}=(1+\val \xi^{2})^{s/2},
\\
\mathcal V^{\infty,\omega}&=\cap_{s\ge 0}\mathcal V^{s,\omega},
\\
\mathcal A&=\{v\in \mathcal V^{\infty, \omega},\ \nabla v\in H^{\infty}(\R^{3})=W^{\infty,2}(\R^{3})\}.
\end{align}
We note that $\mathcal A$ can be defined directly as
\begin{equation}\label{}
\mathcal A=\Biggl\{v\in \mathcal W, \text{such that }
\left[\begin{array}{l}
 \forall \alpha\in \N^{3},\  \p_{x}^{\alpha}v\in \mathcal W,
 \\
 \forall \beta\in \N^{3} \text{ with }\val\beta\ge 1,\  \p_{x}^{\beta}v\in L^{2}(\R^{3}).
\end{array}\right.
\Biggr\}
\end{equation}
\begin{theorem}\label{thm.54yt}
Let $v$ be a vector field satisfying the assumptions  of Conjecture \ref{coj001}. Then $v$ belongs to 
 $\mathcal A$, which  is a multiplicative algebra, stable by differentiation and by the action of standard singular integrals.
\end{theorem}
\begin{nb}
 This theorem is proven in Corollary \ref{cor.lkj987}. The functional analytic structure of 
 $\mathcal V^{s,\omega}$ for $s\ge 0$
 is a Banach algebra structure whereas $\mathcal V^{\io,\omega}$ and $\mathcal A$
are Fr\'echet algebras.
\end{nb}
Since we have $\mathcal A\subset C^{\io}_{(0)}$, the latter space standing for the $\moo$ functions on $\R^{3}$ with limit $0$ at infinity as well as all their derivatives, Theorem \ref{thm.54yt} implies Theorem \ref{thm.known1}; the regularity results given by this theorem will be useful for the proofs of Theorems \ref{thm.galnew++},
\ref{thm.chaenew++},
 thanks to the singular integrals invariance property,
which fails to be true for the algebra $C^{\io}_{(0)}$.
On the other hand, the global homogeneity properties of $\mathcal A$ are not much better than the properties of $C^{\io}_{(0)}$, although we have introduced some discrepancy between small and large frequencies in the definition of $\mathcal A$: if $\beta_{1}$ is a smooth bounded function defined on
$\R^{3}$, such that $0\notin \supp \beta_{1}$, 
we get for $v\in \mathcal A$,
\begin{equation}\label{}
\beta_{1}(D) v\in L^{2}(\R^{3}),\text{\quad as well as all its derivatives.}
\end{equation}
Indeed, we have 
$
\mathcal F\bigl(D^{\gamma}\beta_{1}(D) v\bigr)(\xi)=\xi^{\gamma}\beta_{1}(\xi)\hat v(\xi)
$
and with $\rho_{0}>0$,
\begin{multline*}
\norm{D^{\gamma}\beta_{1}(D) v}_{L^{2}}^{2}\le \int_{\val{\xi}\ge \rho_{0}} \val\xi^{-2}\val\xi^{2\val \gamma+2}\norm{\beta_{1}}_{L^{\io}}^{2}
\val{\hat v(\xi)}^{2}d\xi
\\
\le \rho_{0}^{-2}\norm{\beta_{1}}_{L^{\io}}^{2}
\int\val\xi^{2\val \gamma+2}
\val{\hat v(\xi)}^{2}<+\io,\quad\text{\small since $v$ belongs to $\mathcal A$.}
\end{multline*}
We have also with $\alpha_{0}$ smooth compactly supported in $\R^{3}$, 
\begin{equation}\label{}
\alpha_{0}(D) v\in \mathcal W\cap L^{6}(\R^{3}),\text{\quad as well as all its derivatives.}
\end{equation}
Indeed, we have 
$$
D^{\gamma}\alpha_{0}(D) v=\alpha_{0}(D)\underbracket[0.1pt]{D^{\gamma} v}_{
\substack
{{\in \mathcal W, \text{ since }}
\\
v\in \mathcal V^{\io,\omega}}}
\
\text{\footnotesize which belongs to $\mathcal V^{\io,\omega}$
since $\alpha_{0}(D)$ is a  Fourier multiplier.}
$$
Moreover, we have 
$
\curl\bigl(D^{\gamma}\alpha_{0}(D) v\bigr)=D^{\gamma}\alpha_{0}(D) \curl v,
$
which belongs to $L^{2}$
since $D^{\gamma}\alpha_{0}(D)$ is a  Fourier multiplier and $\curl v\in L^{2}$.
Also we have that  $D^{\gamma}\alpha_{0}(D) v$  belongs to $\mathcal W$
since $D^{\gamma} v$ belongs to $\mathcal W$ and $\alpha_{0}(D)$ is a Fourier multiplier, implying that $D^{\gamma}\alpha_{0}(D) v\in \mathcal W$. 
We are also able to prove that $v$ belongs to $L^{6}$, which implies that 
$D^{\gamma}\alpha_{0}(D)v$ belongs to $L^{6}$.
Our first Liouville theorem is a generalization of the result (\ref{111kno}) on page \pageref{111kno}.
\begin{theorem}\label{thm.galnew++}
Let us assume that the assumptions of Conjecture \ref{coj001} are fulfilled for a vector field $v$. 
Moreover we assume that there exists $\alpha_{0}\in \mooc(\R^{3}_{\xi})$ whose support contains a neighborhood of $0$ such that
$$
\alpha_{0}(D) v\in L^{9/2}(\R^{3}).
$$
Then $v\equiv 0$.
\end{theorem}
\begin{nb}
This theorem is proven in Theorem \ref{thm.galnew}. 
 This result means that the assumption $v\in L^{9/2}$ is too strong and that it is enough to assume that $v_{[0]}$, the projection of $v$ onto the space of vectors with spectrum in a given  neighborhood of 0 belongs to $L^{9/2}$.
 A key point for doing the latter is that the part of $v$ with large spectrum belongs to $L^{2}$,
 namely that for any $\beta_{1}\in \mooc(\R^{3}_{\xi})$ such that $0\notin\supp \beta_{1}$,
 we do have $\beta_{1}(D) v\in L^{2}(\R^{3})$.
 \end{nb}
 Our next Liouville theorem is a generalization of the result (\ref{333kno}) on page \pageref{333kno} and is proven by Theorem \ref{thm.chaenew}.
 \begin{theorem}\label{thm.chaenew++}
Let us assume that the  assumptions of Conjecture \ref{coj001} are fulfilled for a vector field $v$. 
Moreover we assume that there exists $\alpha_{0}\in \mooc(\R^{3}_{\xi})$ whose support contains a neighborhood of $0$ such that
$$
 \alpha_{0}(D)\Delta v\in L^{6/5}(\R^{3}).
$$
Then $v\equiv 0$.
\end{theorem}
In the sequel $Q$ will stand for the Bernoulli head pressure defined in \eqref{modpr0}.
\begin{theorem}\label{thm.54ez}
Let us assume that the  assumptions of Conjecture \ref{coj001} are fulfilled for a vector field $v$.
Moreover we assume that there exists $\alpha_{0}\in \mooc(\R^{3}_{\xi})$ whose support contains a neighborhood of $0$ such that 
\begin{equation}\label{}
\alpha_{0}(D) \nabla Q\quad\text{belongs to $L^{6/5}(\R^{3}).$}
\end{equation}
 Then the vector field $v$ is identically 0.
\end{theorem}
\begin{nb}
 This theorem is proven in Section \ref{sec.rsq547}.
\end{nb}
\begin{corollary}\label{cor.87zz}
Let us assume that the  assumptions of Conjecture \ref{coj001} are fulfilled for a vector field $v$.
Moreover we assume that
\begin{equation}\label{}
\nabla Q\quad\text{belongs to $L^{6/5}(\R^{3})$.}
\end{equation}
 Then the vector field $v$ is identically 0.
\end{corollary}
That Corollary follows also from Corollary \ref{cor.54gf} below since we have 
$$
\Delta Q=\nu^{-1}v\cdot \nabla Q+\val{\curl v}^{2},
$$
so that since $v\in L^{6}$, having $\nabla Q\in L^{6/5}$ and $\curl v\in L^{2}$
implies readily $\Delta Q\in L^{1}$.
It seems that even Corollary \ref{cor.87zz} contains a new result. We shall see that we can obtain rather easily that $\nabla Q$ belongs to $L^{3/2}=L^{\frac{15}{10}}$, although our assumption is stronger  at infinity with $\nabla Q\in L^{\frac{12}{10}}$: indeed we have for $K$ compact subset of $\R^{3}$,
$$
\int_{K^{c}}\val{(\nabla Q)(x)}^{\frac{15}{10}}dx\le \norm{\nabla Q}_{L^{\io}(K^{c})}^{\frac3{10}}
\int_{K^{c}}\val{(\nabla Q)(x)}^{\frac{12}{10}}dx, 
$$
so that if we have $\nabla Q\in L^{\frac{6}{5}}$ on the complement of $K$, since $\nabla Q$ is a bounded function (as an element of the Wiener algebra $\mathcal W$), then we obtain that $\nabla Q$ belongs to $L^{3/2}$ on $K^{c}$.
The next result is more technical with a very strong assumption,
which turns out to be useful for the proof of Theorems 
\ref{thm.galnew++}, \ref{thm.chaenew++}, 
\ref{thm.54ez}.
\begin{theorem}\label{thm.54ek}
Let us assume that the  assumptions of Conjecture \ref{coj001} are fulfilled for a vector field $v$.
Moreover we assume that, 
\begin{align}
&\text{\bf either }(\Delta Q)_{+} \in L^{1}(\R^{3})
\text{ \bf or }(\Delta Q)_{-} \in L^{1}(\R^{3}).
\end{align}
 Then the vector field $v$ is identically $0.$
\end{theorem}
\begin{nb}
 Note that we did not assume that $\Delta Q$ belongs to $L^{1}(\R^{3})$, but only that $(\Delta Q)_{+}$ or  $(\Delta Q)_{-}$ belong to $L^{1}(\R^{3})$.
 This theorem is proven in Theorem \ref{thm.5468}.
\end{nb}
\begin{corollary}\label{cor.54ek}
Let us assume that the  assumptions of Conjecture \ref{coj001} are fulfilled for a vector field $v$.
Moreover we assume that there exists a set $L$ of $\R^{3}$ with finite Lebesgue measure such that 
\begin{equation}\label{}
\forall x\in L^{c}, \ (\Delta Q)(x) \ge 0.
\end{equation}
 Then the vector field $v$ is identically 0.
\end{corollary}
\begin{corollary}\label{cor.54em}
Let us assume that the assumptions of Conjecture \ref{coj001} are fulfilled for a vector field $v$.
Moreover we assume that there exists a set $L$ of $\R^{3}$ with finite Lebesgue measure such that 
\begin{equation}\label{}
\forall x\in L^{c}, \ (\Delta Q)(x) \le 0.
\end{equation}
 Then the vector field $v$ is identically 0.
\end{corollary}
\begin{corollary}\label{cor.54gf}
Let us assume that the  assumptions of Conjecture \ref{coj001} are fulfilled for a vector field $v$.
Moreover we assume that, 
\begin{equation}\label{}
\Delta Q \in L^{1}(\R^{3}).
\end{equation}
 Then the vector field $v$ is identically 0.
\end{corollary}
We can see  that Corollary \ref{cor.54gf} implies the standard result of D.~Chae in \cite{MR3162482},
listed as (3) on page \pageref{333kno}.
Indeed we have the classical  formula, coming from \eqref{SNSI}, namely
\begin{multline}\label{dsw566}
(\Delta Q)(x)=\val{(\curl v)(x)}^{2}+\nu^{-1}v(x)\cdot (\nabla Q)(x)
\\=\val{(\curl v)(x)}^{2}-
\proscal3{(\curl^{2}v)(x)}{v(x)},
\end{multline}
so that  since $v$ belongs to $L^{6}$ and the assumption in \cite{MR3162482} is $\curl^{2} v\in L^{6/5}$, the equation \eqref{dsw566} and the hypothesis \eqref{hyp002} implies $\Delta Q\in L^{1}$
and we can apply Corollary  \ref{cor.54gf} to obtain the triviality of $v$.
We note as well  that 
if $\nabla Q$ belongs to $L^{6/5}$, Formula \eqref{dsw566}
implies that $\Delta Q$ belongs to $L^{1}$,
so that we can use again Corollary  \ref{cor.54gf}  to obtain the triviality of $v$.
\section{\color{magenta}Preliminaries}
\subsection{A linear lemma}
We would like to introduce some classical statements,
incorporating a remark on the Fourier transform of $v$, which will be useful for us later on.
In particular,
we want to prove that whenever a divergence-free vector field in $\R^{3}$ has a curl in $L^{2}$,
tends to 0 at infinity, then it belongs to $L^{6}$. We do not use the non-linear equation \eqref{SNSI} 
in that section.
\begin{lemma}\label{lem.reg001}
 Let $v$ be a divergence-free vector field on $\R^{3}$  such that \eqref{hyp001}, \eqref{hyp002} are satisfied.
 Then 
 for 
 \begin{equation}\label{sdf987}
 1<q_{1}<\frac65<q_{2}\le 2,\quad
 \text{$\widehat v$ belongs to $L^{q_{1}}(\R^{3})+L^{q_{2}}(\R^{3})\subset L^{1}_{\text{loc}}(\R^{3})$},
\end{equation}
 and 
 $v$ belongs to $L^{6}(\R^{3}, \R^{3})$. We have also 
 \begin{equation}\label{estl06}
 \norm{v}_{L^{6}(\R^{3}, \R^{3})}\le \sigma_{0}\norm{\curl v}_{L^{2}(\R^{3}, \R^{3})},
\end{equation}
 where $\sigma_{0}$ is a universal constant 
\footnote{The best constants for these inequalities were found by G. Talenti in \cite{MR463908}. See also Theorem 8.3 in the book \cite{MR1817225} by E. Lieb \& M. Loss.\label{foot.lieb}}. 
 Moreover the $3\times 3$
 matrix $\nabla v=(\p_{j} v_{k})_{1\le j,k\le 3}$ belongs to 
 $L^{2}(\R^{3}, \R^{9})$ and we have 
 \begin{equation}\label{dssee5}
\norm{\nabla v}_{L^{2}(\R^{3}, \R^{9})}=\sqrt{\sum_{1\le j,k\le 3}\norm{\p_{j}v_{k}}_{L^{2}(\R^{3},\R^{3})}^{2}}=\norm{\curl v}_{L^{2}(\R^{3}, \R^{3})}.
\end{equation}
In particular, we have
\begin{align}
&\nabla v\in L^{2}(\R^{3}, \R^{9}), \quad v\in  L^{6}(\R^{3}, \R^{3}),\label{41lk59}\\
& (\curl v)\times v\in  L^{3/2}(\R^{3}, \R^{3}).\label{41lkh9}
\end{align}
\end{lemma}
\begin{proof}
Let $v$ be a tempered vector field in $\mathscr S'(\R^{3}, \R^{3})$.
Noting that polynomials are multipliers of $\mathscr S'(\R^{3})$,
and using the vector triple product formula, 
$$
i\xi\times\bigl(i\xi\times \hat v(\xi)\bigr)= (i\xi\cdot \hat v(\xi))i\xi-(i\xi\cdot i\xi)\hat v(\xi),
$$
we get that 
\begin{equation}\label{}
\curl^{2}v=\grad(\dive v)-\Delta v.
\end{equation}
We get in particular, for a divergence-free $v$ in $\mathscr S'(\R^{3}, \R^{3})$,
\begin{equation}\label{125lkj}
i\xi\times\bigl(i\xi\times \hat v(\xi)\bigr)=\val{\xi}^{2}\hat v(\xi), \quad \text{i.e.\quad}
\curl^{2}v=-\Delta v.
\end{equation}
On the other hand, since $v$ is  a tempered distribution, and for $\chi$ a cutoff function such as given by \eqref{fun001}, and $\tilde \chi=1-\chi$ (the function $\tilde \chi$ is valued in [0,1],   vanishing on $\mathbb B^{3}$ and is equal to 1 
on $(2\mathbb B^{3})^{c}$),
we have, for any $\varepsilon\in(0,1/2]$,
\begin{multline}\label{iug547}
\hat v(\xi)=\chi\bigl(\frac\xi\varepsilon\bigr)\hat v(\xi)
+\tilde\chi\bigl(\frac\xi\varepsilon\bigr)\chi(\xi)\hat v(\xi)+
\tilde\chi\bigl(\frac\xi\varepsilon\bigr)\tilde\chi(\xi)\hat v(\xi)
\\
=
\chi\bigl(\frac\xi\varepsilon\bigr)\hat v(\xi)
+\tilde\chi\bigl(\frac\xi\varepsilon\bigr)\chi(\xi)\hat v(\xi)+
\tilde\chi(\xi)\hat v(\xi),
\end{multline}
since we can check easily\footnote{\eqref{54lkss} is  true for $\val \xi\le 1$, since then $\chi(\xi)=1$ thus $\tilde\chi(\xi)=0$, also true if $\val \xi\ge 2\varepsilon$, since then $\chi(\xi/\varepsilon)=0$,
thus $\tilde\chi(\xi/\varepsilon)=1$, true thus everywhere since $2\varepsilon\le 1$. } 
that 
\begin{equation}\label{54lkss}
\tilde\chi\bigl(\frac\xi\varepsilon\bigr)\tilde\chi(\xi)=\tilde\chi(\xi).
\end{equation}
Using \eqref{125lkj}, we define  
\begin{multline}\label{fds658}
W_{1,\varepsilon}(\xi)=
\tilde\chi\bigl(\frac\xi\varepsilon\bigr)\chi(\xi)\hat v(\xi)=
\tilde\chi\bigl(\frac\xi\varepsilon\bigr)\chi(\xi)\val\xi^{-2}i\xi\times\bigl(i\xi\times\hat v(\xi)\bigr)
\\=\tilde\chi\bigl(\frac\xi\varepsilon\bigr)\chi(\xi)\val\xi^{-2}i\xi\times\widehat{\curl v}(\xi),
\end{multline}
and we note that
$$
\bigl\vert
\tilde\chi\bigl(\frac\xi\varepsilon\bigr)\chi(\xi)\val\xi^{-2}i\xi
\bigr\vert=\underbrace{\tilde\chi\bigl(\frac\xi\varepsilon\bigr)\chi(\xi)}
_{\substack{\text{bounded by 1,}\\\text{supported}
\text{ in }\val\xi\le 2
}}
\val\xi^{-1}\text{is bounded in } L^{p_{1}}(\R^{3}), \quad 1\le p_{1}<3.
$$
Since $\curl v$ belongs to $L^{2}$ and thus $\widehat{\curl v}\in L^{2}$,
this implies that 
\begin{equation}\label{1216hg}
\text{$(W_{1,\varepsilon})_{0<\varepsilon\le 1/2}$ given by \eqref{fds658}
is bounded
in $L^{q_{1}}$, $1\le q_{1}<\frac65.$}
\end{equation}
We define now
\begin{equation}\label{fds677}
W_{2}(\xi)=\tilde\chi(\xi)\hat v(\xi)=\tilde\chi(\xi)\val\xi^{-2}i\xi\times\widehat{\curl v}(\xi),
\quad
\supp \tilde \chi\subset\{\val{\xi}\ge 1\},
\end{equation}
and we obtain that $\xi\mapsto\tilde\chi(\xi)\val\xi^{-2}i\xi$ belongs to $L^{p_{2}}$ with $p_{2}\in (3,+\io]$, so that 
{\begin{equation}\label{}
\text{$W_{2}$ given by \eqref{fds677}
belongs to 
$L^{q_{2}}$, $\frac65<q_{2}\le 2$}.
\end{equation}}
Gathering the above information, we get that 
\begin{equation}\label{55jkk8}
\hat v(\xi)=\underbrace{\chi\bigl(\frac\xi\varepsilon\bigr)\hat v(\xi)}_{=W_{0,\varepsilon}(\xi)}+W_{1,\varepsilon}(\xi)+W_{2}(\xi).
\end{equation}
Let us now choose $q_{1}\in (1,6/5)$; from \eqref{1216hg}, we may find a sequence $(\varepsilon_{k})$ in $(0,1/2)$ with limit $0$ such that the sequence 
$(W_{1,\varepsilon_{k}})$ converges weakly in $L^{q_{1}}$ towards $W_{1}$ and this implies also weak convergence to $W_{1}$ in $\mathscr S'$. From \eqref{55jkk8}, we get that 
the sequence 
$(W_{0,\varepsilon_{k}})$ converges weakly in $\mathscr S'$ towards $W_{0}$ and we have 
\begin{equation}\label{}
\hat v=W_{0}+W_{1}+W_{2}, \quad \supp W_{0}\subset\{0\}, \quad W_{1}\in L^{q_{1}}, \quad W_{2}\in L^{q_{2}}.
\end{equation}
We get in particular that $\check{\widehat{W_{0}}}$ is a polynomial and the Hausdorff-Young Inequality implies  
\begin{multline}\label{}
v=P+\check{\widehat{W_{1}}}+\check{\widehat{W_{2}}}, 
\quad 
P \text{ polynomial, }
\quad \check{\widehat{W_{1}}}\in L^{q'_{1}}, \quad \check{\widehat{W_{2}}}\in L^{q'_{2}},
\\ 1<q_{1}<6/5<6<q'_{1}<+\io,\quad 6/5<q_{2}\le 2\le q'_{2}<6.
\end{multline}
According to Lemma \ref{lem.jhgf} in our Appendix, we deduce that 
$\check{\widehat{W_{1}}}, \check{\widehat{W_{2}}}$
both belong to $\mathcal L_{(0)}$
and since we have assumed that it is also the case for $v$,
this entails that 
the polynomial
$P$ belongs as well to $\mathcal L_{(0)}$ so that Lemma \ref{cla.548uyt} in our Appendix proves that $P=0$.
We have thus proven that
\begin{equation}\label{drs128}
\forall q_{1}\in (1,\frac65), \forall q_{2}\in (\frac65, 2], \quad
\hat v\in L^{q_{1}}+L^{q_{2}},\quad
v\in L^{q'_{1}}+L^{q'_{2}},
\end{equation}
which is \eqref{sdf987}.
Now using that $\hat v$ belongs to $L^{1}_{\text{loc}}$, we get,  
\begin{equation}\label{hgf654}
\int_{\R^{3}}\val{\hat v(\xi)}^{2}\val{\xi}^{2}d\xi=\int_{\R^{3}}\val{\xi\times\hat v(\xi)}^{2}d\xi
=\norm{\curl v}_{L^{2}}^{2}<+\io,
\end{equation}
so that 
$v$ belongs to the Homogeneous Sobolev Space $\dot H^{1}(\R^{3})$ 
and using  Lemma \ref{lem.21ml} in our Appendix,
we obtain that $v\in L^{6}(\R^{3})$.
Moreover, we note  that, since $\hat v$ belongs to $L^{1}_{\text{loc}}$,  \eqref{hgf654}
implies that
\begin{align*}
\sum_{1\le j,k\le 3}\norm{\p_{j}v_{k}}_{L^{2}(\R^{3},\R^{3})}^{2}&=\int_{\R^{3}} \xi_{j}^{2}
\bigl(\sum_{1\le j,k\le 3}\val{\widehat{v_{k}}(\xi)}^{2} \bigr)d\xi
\\&=
\int_{\R^{3}} \sum_{1\le j\le 3}\xi_{j}^{2}
\val{\widehat{v}(\xi)}^{2}d\xi
=\int_{\R^{3}}\val{\xi}^{2}\val{\hat v(\xi)}^{2}d\xi=\norm{\curl v}_{L^{2}}^{2},
\end{align*}
concluding the proof of the lemma.
\end{proof}
\subsection{The stationary Navier-Stokes system}
Let us consider a vector field $v \in L^{2}_{\textrm{loc}}(\R^{3};\R^{3})$
satisfying the stationary Navier-Stokes system for incompressible fluids, i.e. 
such that
\begin{equation}\label{SNSI+}
-\nu \Delta v+ \sum_{1\le j\le 3}\p_{j}(v_{j}v)+\nabla p=0, \quad \dive v=0,
\end{equation}
where $\nu>0$ stands for the viscosity, the \emph{pressure} $p$ is a scalar distribution on $\R^{3}$ and the $v_{j}$ are the components of the vector field $v$.
The Leray projection $\mathbb P$ is defined by the Fourier multiplier\footnote{When $m(\xi)$
is a bounded $3\times 3$
matrix,
the Fourier multiplier $m(D)$ is defined on $L^{2}(\R^{3}, \R^{3})$
by
\begin{equation}\label{four01}
(m(D) w)(x)=\int e^{ix\cdot \xi} m(\xi)\hat w(\xi) d\xi(2\pi)^{-\frac32}, \quad \text{i.e.\quad}
m(D)=\mathcal F^{-1} m\mathcal F=\mathcal F^{*} m\mathcal F, 
\end{equation}
where $\mathcal F$ stands for the Fourier transformation 
(see e.g. our Appendix, Section \ref{sec.fourier}).
}
\begin{equation}\label{leray1}
\leray(D)=I_{3}-\val{D}^{-2}(D\otimes D), \quad D=-i\nabla.
\end{equation}
\begin{lemma}\label{lem.leray1}
The operator $\leray$ defined by \eqref{leray1} is the projection in 
$L^{2}(\R^{3},\R^{3})$ onto the subspace
$L^{2}_{0}(\R^{3},\R^{3})$ of $L^{2}$ vector fields with null divergence.
Moreover with $\rot$ standing for the curl operator,
we have 
\begin{equation}\label{leray2}
\leray=\val{D}^{-2}\rot^{2}.
\end{equation}
\end{lemma}
\begin{proof}
 Indeed, as a real bounded Fourier multiplier, it is a selfadjoint bounded operator on $L^{2}(\R^{3};\R^{3})$ and a projection since the real Fourier multiplier
\begin{equation}\label{ortler}
\pleray(D)=\val{D}^{-2}(D\otimes D)=\Delta^{-1}\grad \dive,
\end{equation}
has the (matrix)\footnote{The $3\times3$ matrix $\xi\otimes \xi$ is the symmetric matrix
$
(\xi_{j}\xi_{j})_{1\le j,k\le 3}.
$} symbol $\val{\xi}^{-2}(\xi\otimes \xi)$ and $\pleray(\xi)^{2}=\pleray(\xi)$.
Moreover, a direct calculation shows that
\begin{align}
\rot(\xi)&=i\begin{pmatrix}0&-\xi_{3}&\xi_{2}\\ \xi_{3}&0&-\xi_{1}\\ -\xi_{2}&\xi_{1}&0\end{pmatrix},
\quad
\rot(\xi)^{*}=\rot(\xi),
\\
\rot(\xi)^{2}&=\val{\xi}^{2} I_{3}-\xi\otimes \xi,\label{curlca}
\end{align}
the latter formula proving \eqref{leray2}. 
Note that the matrix   $\rot (\xi)$ is purely imaginary and skew-symmetric,
so is a natural candidate for a formally self-ajoint operator,
(as would be a real symmetric matrix).
The reader will find in Section \ref{sec.fourier} on page \pageref{subsec.leray}
of our Appendix the proof that the range of $\mathbb P$ is $L^{2}_{0}(\R^{3}, \R^{3})$.
\end{proof}
\begin{lemma}\label{lem.newppp}
Let $v$ be a divergence-free vector field on $\R^{3}$  such that \eqref{hyp001}, \eqref{hyp002} are satisfied.
Then we have $\curl v\in L^{2}(\R^{3}, \R^{3}), v\in L^{6}(\R^{3}, \R^{3}), \val v^{2}\in L^{3}(\R^{3}, \R),$
so that 
$
\curl v\times v
$
and $v\cdot \nabla v$ belong to  $L^{3/2}(\R^{3}, \R^{3})$. Moreover, we have 
 \begin{equation}\label{218218}
(\curl v)\times v= (v\cdot \nabla) v-\frac12\nabla( \val v^{2}).
\end{equation}
\end{lemma}
\begin{proof} The first assertions follow readily from Lemma \ref{lem.reg001}.
We have,
using Lemma \ref{lem.reg001}\footnote{Note in particular that all products
$
v_{j}\p_{k} v_{l}
$
make sense
as a $L^{3/2}$ function which is  the product of the $L^{6}$ function $v_{j}$
by the $L^{2}$ function $\p_{k} v_{l}$.
},
\begin{align*}
(\curl v)\times v&=\mattre{\p_{2}v_{3}-\p_{3}v_{2}}
{\p_{3}v_{1}-\p_{1}v_{3}}
{\p_{1}v_{2}-\p_{2}v_{1}}
\times\mattre{v_{1}}{v_{2}}{v_{3}}
=\mattre{v_{3}(\p_{3}v_{1}-\p_{1}v_{3})-v_{2}(\p_{1}v_{2}-\p_{2}v_{1})}
{v_{1}(\p_{1}v_{2}-\p_{2}v_{1})-v_{3}(\p_{2}v_{3}-\p_{3}v_{2})}
{v_{2}(\p_{2}v_{3}-\p_{3}v_{2})-v_{1}(\p_{3}v_{1}-\p_{1}v_{3})}
\\
&=\begin{pmatrix}
{(v_{1}\p_{1}+v_{2}\p_{2}+v_{3}\p_{3}) (v_{1} )  -\frac12 \p_{1}(v_{1}^{2} +v_{2}^{2}+v_{3}^{2})}
\\[0.2cm]
{(v_{1}\p_{1}+v_{2}\p_{2}+v_{3}\p_{3}) (v_{2} )  -\frac12 \p_{2}(v_{1}^{2} +v_{2}^{2}+v_{3}^{2})}
\\[0.2cm]
{(v_{1}\p_{1}+v_{2}\p_{2}+v_{3}\p_{3}) (v_{3} )  -\frac12 \p_{3}(v_{1}^{2} +v_{2}^{2}+v_{3}^{2})}
\end{pmatrix}
\\
&=(v\cdot \nabla) v-\frac12\nabla{\val v^{2}},
\end{align*}
proving the lemma.
 \end{proof}
\begin{lemma}\label{lem.654oof}
Under the assumptions 
\eqref{hyp001}, \eqref{hyp002},
the system \eqref{SNSI} can be written as 
 \begin{equation}\label{newway}
\nu \rot^{2} v+\mathbb P(\rot v\times v)=0, \quad \dive v=0,
\end{equation}
or equivalently
as
\begin{equation}\label{new002}
\nu \rot^{2} v+\rot v\times v+\nabla Q=0, \quad \dive v=0,
\end{equation}
where $Q$ is the Bernoulli head pressure as given by \eqref{modpr0}.
We have in particular that 
\begin{equation}\label{ber000}
\nabla Q=-\tilde{\mathbb P}\bigl(\rot v\times v\bigr),\quad\tilde{\mathbb P}=I-\mathbb P,
\end{equation}
where $\mathbb P$ is given by \eqref {leray2}.
\end{lemma}
\begin{proof}
 Indeed, we have  from \eqref{curlca} and $\dive v=0$, 
 \begin{equation}\label{ouy527}
 \rot^{2} v=-\Delta v+\grad \dive v=-\Delta v,
\end{equation}
and moreover, \eqref{SNSI} implies that
$
\dive\bigl(\p_{j}(v_{j}v)\bigr)+\Delta p=0.
$ 
Since each $v_{j}v$ belongs to $L^{3}$ (since $v$ belongs to $L^{6}$ from Lemma \ref{lem.reg001}), we may define (still with Einstein convention)\footnote{We may notice that $\tilde p$ is real-valued, thanks to Lemma \ref{lem.fourea}, since the matrix-valued Fourier multiplier
$\bigl(\xi_{j}\xi_{k}\val\xi^{-2}\bigr)_{1\le j,k\le 3}$ is an even real-valued symmetric matrix of $\xi$.\label{foot.003}} , 
\begin{equation}\label{211oiu}
\tilde p=\underbracket[0.2pt]{      \val{D}^{-2}   \dive \p_{j}    }_{\substack{
\text{Fourier multiplier}
\\
\text{of order 0}
}}(v_{j}v),\quad\text{which belongs to $L^{3}(\R^{3}, \R).$}
\end{equation}
Note that $\tilde p$ is real-valued, thanks to the reality of $v$ and to Lemma \ref{lem.54oi},
since $\xi_{j}\xi\val\xi^{-2}$ is even and real-valued.
We have 
$
-\Delta p=\dive\bigl(\p_{j}(v_{j}v)\bigr)=-\Delta \tilde p,
$
so that $p-\tilde p$ is harmonic. {We may assume that $p$ belongs to $\mathscr S'(\R^{3}, \R)$}
and since $\tilde p\in L^{3}(\R^{3}, \R)$, we get that $p-\tilde p$ is a harmonic polynomial,
thanks to Remark \ref{rem.000}.
Moreover, using the notation \eqref{l0inft} and Lemma \ref{lem.jhgf}, we get that $\tilde p$ belongs to $\mathcal L_{(0)}$($\supset L^{3}$);
{we may as well assume that $p$ belongs to $\mathcal L_{(0)}$} (it will be the case if 
there exists $q\in [1,+\io)$ and a compact set $K$ such that 
$p\in L^{q}(K^{c})$), then $p-\tilde p$ is a polynomial in $\mathcal L_{(0)}$
which entails 
\begin{equation}\label{fea541}
p=\tilde p,
\end{equation}
thanks to Lemma \ref{cla.548uyt}.
As a consequence, we get using \eqref{ortler},
$$
\nabla p= \grad\val{D}^{-2}\dive\p_{j}(v_{j}v)=\underbracket[0.2pt]{\val{D}^{-2}\grad\dive}_{-\widetilde{\mathbb P}}\underbrace{\p_{j}(v_{j}v)}_{\substack{
\in L^{3/2}(\R^{3}, \R^{3})\\ \text{cf.\eqref{41lk59}}
}},
$$
and thus
$
\p_{j}(v_{j}v)+\nabla p=(I-\widetilde{\mathbb P})\bigl(\p_{j}(v_{j}v)\bigr)=\mathbb P\bigl(\p_{j}(v_{j}v)\bigr),
$
entailing from Lemma \ref{lem.newppp}
\begin{equation}\label{987lkj}
\p_{j}(v_{j}v)+\nabla p=\mathbb P\bigl(\underbrace{(\rot v\times v)}_{\substack{
\in L^{3/2}\\\text{cf.}\eqref{41lk59}
}}+\frac12\underbrace{\nabla \val v^{2}}_{L^{3/2}}\bigr).
\end{equation}
We note now that 
\begin{multline*}
\mathbb P\bigl(\nabla(\val v^{2})\bigr)=\nabla(\val v^{2})+\val{D}^{-2}\grad\dive \bigl(
\nabla(\val v^{2})\bigr)
\\
=\nabla(\val v^{2})+\grad\val{D}^{-2}\Delta(\val{v}^{2})
=\nabla(\val v^{2})-\nabla(\val v^{2})=0,
\end{multline*}
so that \eqref{987lkj}, \eqref{ouy527} imply 
$$
-\nu \Delta v+\p_{j}(v_{j}v)+\nabla p=\nu \rot^{2}v+\mathbb P\bigl(\rot v\times v\bigr),
$$
concluding the proof of \eqref{newway}.
Moreover, we have 
\begin{multline}\label{}
-\tilde{\mathbb P}\bigl(\rot v\times v\bigr)=\val{D}^{-2}\grad\dive\Bigl(\p_{j}(v_{j}v)-\frac12\nabla \val v^{2}\Bigr)
\\\underbracket[0.1pt]{=}_{\substack{\eqref{211oiu}\\
p=\tilde p}}\nabla p+\frac12\nabla(\val v^{2}) 
=\nabla Q, 
\end{multline}
where $Q$ is the Bernoulli head pressure as given by \eqref{modpr0}
(note that $Q$ belongs to $L^{3}$).
As a consequence we get that 
$$
\mathbb P(\rot v\times v)=\rot v\times v-\tilde{\mathbb P}\bigl(\rot v\times v)=\rot v\times v+\nabla Q,
$$
concluding the proof of the lemma.
\end{proof}
\begin{remark}\rm
 We shall now stick with the expression \eqref{newway} for the stationary Navier-Stokes system
 for incompressible fluids, but we have also proven that the assumptions
 \eqref{hyp001}, \eqref{hyp002} secure the fact that the non-linear  term
 $\rot v\times v$ belongs to $L^{3/2}$,
 so that Theorem \ref{thm.singsing}
 ensures that the formulation  \eqref{newway} is meaningful (and equivalent to the more familiar
 \eqref{SNSI+}).
 Besides its compactness, the formulations \eqref{newway}, \eqref{new002} (along with $\dive v=0$) have several advantages: in the first place, we can always multiply pointwisely $(\rot v)(x)\times v(x)$ by
 $v(x)$ and we get a pointwise identity
 (here $\proscal3{a}{b}$ stands for the canonical dot-product on $\R^{3}$),
 $$
 \proscal3{(\rot v)(x)\times v(x)}{v(x)}=\det\bigl((\rot v)(x), v(x), v(x)\bigr)=0.
 $$
 No integration by parts is necessary to obtain this cancellation,
 which is a purely algebraic matter.
 So the difficulty is somehow concentrated 
 on the term
 $$
 \proscal3{(\nabla Q)(x)}{v(x)},
 $$
 which always make sense, but could fail to be in $L^{1}$.
 Also the formulation \eqref{newway}  is coordinate-free and closer to the geometric expression of the Navier-Stokes equation on a Riemannian manifold. 
 Formulation \eqref{newway}
 was also useful for the definition of \emph{Critical Determinants} in  the joint paper of the author 
 with F.~Vigneron \cite{MR4498943}.
 \end{remark}
 \subsection{Some first regularity results involving the Wiener algebra}
We gather in the following Lemma the first conclusions which could be drawn from the \emph{linear} Lemma \ref{lem.reg001} together with non-linear elliptic equation \eqref{SNSI}, written in the form
\eqref{new002}, \eqref{newway}.
Our improvements are linked to the Wiener algebra
\begin{equation}\label{wiener}
\mathcal W=\{v\in \mathscr S'(\R^{d}), \widehat v\in L^{1}(\R^{d})\}=\mathcal F\bigl(L^{1}(\R^{d})\bigr).
\end{equation}
It follows from the classical Riemann-Lebesgue Lemma that 
\begin{equation}\label{777kjh}
\mathcal W\subset C^{0}_{(0)}\subset L^{\io},
\end{equation}
where $C^{0}_{(0)}$ stands for the (uniformly) continuous functions on $\R^{d}$ with limit 0 
at infinity\footnote{Note also that all the injections in \eqref{777kjh} are continuous but not onto.}.
Since $L^{1}(\R^{d})$ is a Banach algebra for \emph{convolution}, we get that $\mathcal W$ is a Banach algebra for \emph{multiplication}.
\begin{theorem}\label{thm.hgf564}
 Let $v$ be a vector field on $\R^{3}$ such that 
such that
 \eqref{SNSI}, \eqref{hyp001}, \eqref{hyp002} are satisfied.
 Then $\rot^{2}v$ belongs to $L^{2}(\R^{3}, \R^{3})$
 and there exists a universal constant $\sigma_{1}$ such that 
 \begin{equation}\label{azo295}
\norm{\rot^{2}v}_{L^{2}(\R^{3}, \R^{3})}\le \sigma_{1}\nu^{-2}\norm{\rot v}_{L^{2}(\R^{3}, \R^{3})}^{3}.
\end{equation}
The vector field $v$ belongs to the Wiener algebra $\mathcal W$ defined in  \eqref{wiener} and there exists a universal constant $\sigma_{2}$ such that  
 \begin{equation}\label{fgd295}
 \norm{v}_{L^{\io}(\R^{3}, \R^{3})}\le
\norm{\hat v}_{L^{1}(\R^{3}, \R^{3})}\le \sigma_{2}\nu^{-1}\norm{\rot v}_{L^{2}}^{2}.
\end{equation}
\end{theorem}
\begin{proof}
 We know from Lemma \ref{lem.reg001} that $v$ belongs to $L^{6}$ (along with the estimate \eqref{estl06}).
 Since \eqref{hyp002}  is $\rot v\in L^{2}$, we get, since 
 $\frac12+\frac16=\frac23$, 
 \begin{equation}\label{225225}
 \rot v\times v\in L^{3/2}, \text{ with }\norm{ \rot v\times v}_{L^{3/2}}\le \norm{\rot v}_{L^{2}}\norm{v}_{L^{6}}
\hskip-6pt\underbracket[0pt]{\!\le\!\!}_{\text{from \eqref{estl06}}} \sigma_{0}\norm{\rot v}_{L^{2}}^{2},
\end{equation}
where $\sigma_{0}$ is a universal constant.
Since the operator $\mathbb P$ is a standard Singular Integral (cf. Theorem \ref{thm.singsing}), it is a bounded endomorphism of $L^{3/2}$ (with norm noted   $\trinorm{\mathbb P}_{3/2}$)
and we get
\begin{equation}\label{}
\norm{\mathbb P\bigl(\rot v\times v\bigr)}_{L^{3/2}}\le \trinorm{\mathbb P}_{3/2}\sigma_{0}\norm{\rot v}_{L^{2}}^{2},
\end{equation}
which implies, thanks to \eqref{newway},
\begin{equation}\label{try874}
\rot^{2}v\in L^{3/2}\text{ with }\norm{\rot^{2}v}_{L^{3/2}}\le \nu^{-1}\trinorm{\mathbb P}_{3/2}\sigma_{0}\norm{\rot v}_{L^{2}}^{2}.
\end{equation}
We note now that
$\curl(\rot v)$ belongs to $L^{3/2}$  and we have also $\dive (\rot v)=0$ so that 
$$
\nabla \rot v =\nabla \rot^{2} \val{D}^{-2}\rot v=\underbracket[0.1pt]{\nabla \rot \val{D}^{-2}}_{\text{Singular Integral}}\underbrace{\curl(\rot v)}_{\in L^{3/2}}, \quad\text{thus in $L^{3/2}$.}
$$
As a consequence, we obtain that $\rot v\in \dot{W}^{1,\frac32}$ an we have also
(cf. Corollary \ref{cor.350}),
\begin{equation}\label{aze852}
\dot{W}^{1,\frac32}(\R^{3})\subset L^{3}(\R^{3})\quad\text{   since $\frac{1-0}3=\frac23-\frac13,$}
\end{equation}
and this proves that $ \rot v $ belongs to $L^{3}$ along with
\begin{equation}\label{229229}
\norm{\rot v}_{L^{3}}\underbracket[0pt]{\le}_{\eqref{aze852}} \sigma_{1}\norm{\rot^{2}v}_{L^{3/2}}
\underbracket[0pt]{\le}_{\eqref{try874}} \sigma_{2}\nu^{-1}\norm{\rot v}_{L^{2}}^{2},
\end{equation}
where $\sigma_{1}, \sigma_{2}$ are universal constants.
Since $\frac13+\frac16=\frac12$,
we obtain that 
\begin{equation}\label{227hgf}
\rot v\times v\in L^{2}\quad\text{ with \quad}
\norm{\rot v\times v}_{L^{2}}
\le \norm{\rot v}_{L^{3}}\norm{v}_{L^{6}}\hskip-10pt
\underbracket[0pt]{\le}_{\substack{\eqref{229229},\eqref{estl06}}}\hskip-10pt\sigma_{3}\nu^{-1}\norm{\rot v}_{L^{2}}^{3},
\end{equation}
where  $\sigma_{3}$ is a universal constant; this entails as well from the equation \eqref{newway},
that $\rot^{2}v$ belongs to $L^{2}$ with 
\begin{equation}\label{azs295}
\norm{\rot^{2}v}_{L^{2}}\le \sigma_{4}\nu^{-2}\norm{\rot v}_{L^{2}}^{3},
\end{equation}
proving \eqref{azo295}.
\begin{remark}\rm
It might be a good place to check the units consistency of the above inequality: we have,
using the notations of Footnote \ref{foot01} (page \pageref{foot01}) and
\begin{align}
&[\nu]=\mathtt{L^{2}T^{-1}}, [v]=\mathtt {LT^{-1}}, [\ \norm{\rot v}_{L^{2}}]=\mathtt{
\bigl((L^{-1}LT^{-1})^{2}L^{3}\bigr)^{1/2}}
=\mathtt{L^{3/2} T^{-1}},
\label{oig1146}\\
&[\ \norm{\rot^{2}v}_{L^{2}}]=\mathtt{\bigl(( L^{-2}LT^{-1})^{2} L^{3}\bigr)^{1/2}}=\mathtt{L^{1/2} T^{-1}},
\label{oig1147}\\
&[\nu^{-2}\norm{\rot v}_{L^{2}}^{3}]=\mathtt{(L^{2}T^{-1})^{-2}
\bigl((L^{-1}L T^{-1})^{2}L^{3}\bigr)^{3/2}}=\mathtt{L^{\frac92-4} T^{2-3}}\text{ indeed }=\mathtt{L^{1/2} T^{-1}}.\label{oig1148}
\end{align}
As mentioned above in the first section, it is also important to check the soundness of this inequality with respect to the change of function \eqref{126126}. With $w$ as in \eqref{126126}, we have
$$
\int_{\R^{3}}\val{(\rot^{2}w)(x)}^{2} dx=\int_{\R^{3}}\Val{\lambda^{3}(\rot^{2}v)(\lambda x)}^{2} dx
=\lambda^{6-3}\int_{\R^{3}}\val{(\rot^{2}v)(y)}^{2} dy=\lambda^{3}\norm{\rot^{2} v}_{L^{2}}^{2},
$$
and 
$$
\Bigl(\int_{\R^{3}}\val{(\rot w)(x)}^{2} dx\Bigr)^{3}
=\Bigl(\int_{\R^{3}}\Val{\lambda^{2}(\rot v)(\lambda x)}^{2} dx\Bigr)^{3}=\lambda^{12-(3\times 3)}
\Bigl(\int_{\R^{3}}\val{(\rot v)(y)}^{2} dy\Bigr)^{3},
$$
so that Inequality \eqref{azs295} holds true as well for $w$.
Needless to say, that inequality is \emph{non-linear} in the sense that it holds only for solutions of \eqref{SNSI} and cannot be true in any non-trivial vector space since the two sides of \eqref{azs295} have different homogeneities with respect to $v$.
As a sharp contrast, Inequality \eqref{estl06} is indeed a linear one, true for any divergence-free $v\in \mathscr S'\cap \mathcal L_{(0)}$ such that $\curl v\in L^{2}$ 
and both sides are homogeneous of degree 1 with respect to $v$.
We leave to the reader the verification that \eqref{estl06} is also units-consistent as well as invariant with respect to the change  of function \eqref{126126}.
As a final comment on these matters, we could point out that an inequality inconsistent with respect to the units must be wrong and an inequality uninvariant
with respect to the change  of function \eqref{126126} must be trivial,
so that checking these two invariance properties is a necessary (but not sufficient) safeguarding device.
\end{remark}
Let us resume the course of our proof of Theorem \ref{thm.hgf564}.
We know also from Lemma \ref{lem.reg001} that $\hat v\in L^{1}_{\text{loc}}$ and with the informations \eqref{azs295} and \eqref{hyp002}, we may consider for $\rho>0$ the following inequalities:
we have
\begin{multline*}
\int_{\R^{3}}\val{\hat v(\xi)} d\xi=
\int_{\val \xi\le \rho}
{\val \xi}^{-1}\val{\hat v(\xi)} \val\xi d\xi+\int_{\val \xi\ge \rho}\val{\xi}^{-2}\val{\hat v(\xi)}\val{\xi}^{2} d\xi
\\
\le \Bigl(\int_{\val \xi\le \rho}\val{\xi}^{-2} d\xi\Bigr)^{1/2}
\norm{\rot v}_{L^{2}(\R^{3})}
+
\Bigl(\int_{\val \xi\ge \rho}\val{\xi}^{-4} d\xi\Bigr)^{1/2}
\norm{\rot^{2} v}_{L^{2}(\R^{3})}
\\
=2\rho^{1/2}\sqrt{\pi}\norm{\rot v}_{L^{2}(\R^{3})}
+2\rho^{-1/2}\sqrt{\pi}\norm{\rot^{2} v}_{L^{2}(\R^{3})},
\end{multline*}
so that, minimizing the right-hand-side, we find that, 
$$
\norm{\hat v}_{L^{1}}\le 4\sqrt{\pi}\sqrt{\norm{\rot v}_{L^{2}(\R^{3})}\norm{\rot^{2} v}_{L^{2}(\R^{3})}},
$$
and using \eqref{azs295}, we get 
$
\norm{\hat v}_{L^{1}}\le 4\sqrt{\pi}
\sigma_{4}^{1/2}\nu^{-1}
\norm{\rot v}_{L^{2}(\R^{3})}^{2},
$
concluding the proof of the theorem.
\end{proof}
\begin{nb} We leave to the reader the verification that 
$
[\ \norm{\hat v}_{L^{1}}]=\mathtt{L T^{-1}}=[\ \norm{\rot v}_{L^{2}}^{2}].
$
 To obtain this, we note that in the formula
 $
 \widehat v(\xi)=\int_{\R^{3}_{x}} e^{-ix\cdot \xi} v(x) dx (2\pi)^{-3/2},
 $
 we have $[\xi]=\mathtt{L^{-1}}$, so that 
 \begin{equation}\label{unit12}
 [ \hat v(\xi)]=\mathtt{LT^{-1}L^{3}}=\mathtt{L^{4}T^{-1}},
\end{equation}
and 
 $[\ \norm{\hat v}_{L^{1}}]=\mathtt{LT^{-1}L^{3}L^{-3}}=\mathtt{LT^{-1}}$.
 Let us check however that the extreme inequality in \eqref{fgd295}, that is 
 \begin{equation}\label{gft987}
  \norm{v}_{L^{\io}(\R^{3}, \R^{3})}\le
\sigma_{2}\nu^{-1}\norm{\rot v}_{L^{2}}^{2},
\end{equation}
 is invariant with respect to the change of function \eqref{126126}. 
 Indeed, applying that inequality to $w$ given by \eqref{126126}, we get
 $$
 \lambda \norm{v}_{L^{\io}(\R^{3}, \R^{3})}\le\sigma_{2}\nu^{-1}\int \Val{\lambda^{2}\rot v(\lambda x )}^{2} dx
 =\sigma_{2}\nu^{-1}\lambda^{4-3}\norm{\rot v}_{L^{2}}^{2},
 $$ 
 which is the same inequality as \eqref{gft987}.
\end{nb}
The following lemma   is a good provisional summary of  our regularity results obtained so far.
\begin{lemma}\label{lem.reg1}Let $v$ be a vector field on $\R^{3}$ such that 
such that
 \eqref{SNSI}, \eqref{hyp001}, \eqref{hyp002} are satisfied.
 Then we have, with $\gamma_{0}=\norm{\rot v}_{L^{2}}$,
 \begin{align}
&v\in L^{6}\cap \mathcal W,\quad\hskip30pt\norm{v}_{L^{6}}\lesssim\gamma_{0}, \quad 
\norm{v}_{L^{\io}}\le \norm{v}_{\mathcal W}\lesssim
\nu^{-1}\gamma_{0}^{2}, 
\label{reg001}\\
&\rot v\in L^{2}\cap L^{6},\quad
\hskip20pt\norm{\rot v}_{L^{2}}=\gamma_{0}, \quad
\norm{\rot v}_{L^{6}}\lesssim \nu^{-2}\gamma_{0}^{3},
\label{reg002}\\
&\rot v\times v\in L^{\frac32}\cap L^{3},\hskip10pt \norm{\rot v\times v}_{L^{3/2}}\lesssim \gamma_{0}^{2}, 
\quad \norm{\rot v\times v}_{L^{3}}\lesssim
\nu^{-2}\gamma_{0}^{4},
\label{reg003}\\
&\rot^{2} v\in L^{\frac32}\cap L^{3},
\quad \hskip15pt\norm{\rot^{2} v}_{L^{3/2}}\lesssim \nu^{-1}\gamma_{0}^{2},
\quad \norm{\rot^{2} v}_{L^{3}}\lesssim \nu^{-3}\gamma_{0}^{4}.
\label{reg004}
\end{align}
\end{lemma}
\begin{nb}
 Here the inequalities of type $a\lesssim b$ mean that $a\le \sigma b$, where $\sigma$ is a universal positive constant.
\end{nb}
\begin{proof}
 Property \eqref{reg001} follows from Lemma \ref{lem.reg001} and \eqref{fgd295}.
 In \eqref{reg002}, $\rot v\in L^{2}$ is the assumption  \eqref{hyp002};
 we postpone the proof of the second part of \eqref{reg002} for a while.
 The fact that $\rot v\times v$ belongs to $L^{3/2}$ with the afferent estimate is
  proven in \eqref{225225}. From \eqref{229229}, we get that $\rot v\in L^{3}$
  with 
  $$
  \norm{\rot v}_{L^{3}}\lesssim \nu^{-1}\gamma_{0}^{2},
  $$
  so that the already proven \eqref{reg001} gives \eqref{reg003}.
  Moreover \eqref{reg004} follows from the already proven \eqref{reg003} and the equation \eqref{newway}.
  Since $\rot v$ is divergence-free and $\rot^{2}v$ belongs to $L^{2}$ (cf. \eqref{azs295}),
  we may apply Lemma \ref{lem.reg001} to $\rot v$
  and obtain that  
$\rot v$ belongs to $\dot H^{1}(\R^{3})$, which is a subset of $L^{6}(\R^{3})$, proving the last part of \eqref{reg002} along with
$$
\norm{\rot v}_{L^{6}}\lesssim \norm{\rot^{2} v}_{L^{2}}\lesssim \nu^{-2}\gamma_{0}^{3}.
$$
The proof of the lemma is complete.
 \end{proof}
 \begin{lemma}
 Let $v$ be as in Lemma \ref{lem.reg1} and let $Q$ be the Bernoulli head pressure given by 
 \eqref{modpr0}. Then we have, with the notations of Lemma \ref{lem.reg1},
\begin{align}
&Q\in L^{3}\cap \mathcal W, \quad\norm{Q}_{L^{3}}\lesssim \gamma_{0}^{2},
\quad\norm{Q}_{\mathcal W}\lesssim \nu^{-2}\gamma_{0}^{4},
\label{ber001}
\\
&\nabla Q\in L^{3/2}\cap L^{3},\quad \norm{\nabla Q}_{L^{3/2}}\lesssim \gamma_{0}^{2}, 
\quad \norm{\nabla Q}_{L^{3}}\lesssim
\nu^{-2}\gamma_{0}^{4}.
\label{ber002}
\end{align}
\end{lemma}
\begin{proof}
 The estimate \eqref{ber002} follows readily from 
 \eqref{reg003} in Lemma \ref{lem.reg1}  and \eqref{ber000}. Moreover, from 
 \eqref{211oiu}, \eqref{fea541} and  \eqref{modpr0}, we obtain that
 \begin{equation}\label{ber003}
Q= \bigl(\val{D}^{-2}   \dive \p_{j}\bigr) (v_{j}v)+\frac12\val{v}^{2},
\end{equation}
and the $L^{3}$ boundedness of the singular integral $\val{D}^{-2}   \dive \p_{j}$ proves that $Q$
belongs to $L^{3}$ with the sought estimate, thanks to \eqref{reg001}.
The identity \eqref{ber003} implies that 
\begin{equation}\label{}
\widehat Q(\xi)=-\sum_{1\le j\le 3}\val{\xi}^{-2} \xi_{j} \xi\cdot(\widehat{v_{j}}\ast \widehat v)(\xi)+
\frac12(\hat v\oast\hat v)(\xi),
\end{equation}
with 
\begin{align}
(\widehat{v}_{j}\ast \widehat v)(\xi)&=
\int {\hat v_{j}(\eta)}{\hat v(\xi-\eta)}d\eta(2\pi)^{-3/2},
\label{tcs15+}
\\
(\hat v\oast\hat v)(\xi)&=\int \proscal3{\hat v(\eta)}{\hat v(\xi-\eta)}d\eta(2\pi)^{-3/2},
\label{tcs159}
\end{align}
so that we get
$
\norm{\widehat Q}_{L^{1}}\lesssim\norm{\widehat v}_{L^{1}}^{2},
$
and the sought result, thanks to \eqref{reg001}.
\end{proof} 
\begin{remark}\rm
 The Wiener algebra $\mathcal W$ satisfies \eqref{777kjh} and has the same ``homogeneity''
 as $C^{0}_{(0)}$ and $L^{\io}$, say in the sense that the norm of $x\mapsto v(\lambda x)$
 is the norm of $v$ for the three algebras in \eqref{777kjh}; however, the key improvement following from the property that $v$ belongs to $\mathcal W$ is the fact that standard singular integrals
  (as given by Theorem \ref{thm.singsing}) are bounded on $\mathcal W$, thanks to the inequality,
  \begin{equation}\label{}
  \norm{m(D) v}_{\mathcal W}\le \norm{m}_{L^{\io}(\R^{3})} \norm{v}_{\mathcal W}.
\end{equation}
This type of property fails to be true for the algebras $C^{0}_{(0)}$ and $L^{\io}$.
\end{remark}
\begin{remark}\rm
Using Lemma \ref{lem.newppp},
we note that, with Einstein convention, 
$$
\rot v\times v=\p_{j}(v_{j}v)-\frac12\nabla{\val v}^{2},$$
so that with the singular integral $R_{0}$ given by 
\begin{equation}\label{ddd001}
R_{0,j}=\val{D}^{-2}\dive\p_{j},
\quad
R_{0}(v\otimes v)=\sum_{1\le j\le 3}R_{0,j}(v_{j}v),
\end{equation}
we have 
\begin{align*}
\dive(\rot v\times v)&=\dive \p_{j}(v_{j}v)+\frac12\val{D}^{2}({\val v}^{2})
\\&=\val{D}^{2}\Bigl[\val{D}^{-2}\dive \p_{j}(v_{j}v)+\frac12\val{v}^{2}\Bigr]
\\&=\val{D}^{2}\Bigl[R_{0}(v\otimes v)+\frac12\val{v}^{2}
\Bigr]
\\&=\val{D}^{2}S_{0}(v\otimes v),
\end{align*}
where $S_{0}$ is a singular integral (note that $v\otimes v$ belongs to $L^{3}$ so that
$S_{0}(v\otimes v)$ makes sense and belongs to $L^{3}$).
As a consequence we have 
$$
\val{D}^{-2}\grad\dive (\rot v\times v)=\grad\Bigl[ S_{0}(v\otimes v)\Bigr],
$$
which makes sense as an element of $W^{-1,3}(\R^{3})$ since $S_{0}(v\otimes v)$ belongs to $L^{3}$.
\end{remark}
In the sequel of this section, we want to formulate and prove some refinements of the smoothness results obtained 
by
G.P.~Galdi  in \cite{MR2808162}: in the latter book the author  is proving that functions satisfying the assumptions of Conjecture \ref{coj001} are in fact smooth, with limit 0 at infinity, as well as all their derivatives.
We shall use the following notation, for $\xi\in \R^{d}$, we set\footnote{We remark here that $\valjp{\xi}$
is a multiplier of $\mathscr S$ (that is 
$f\mapsto g$ with $g(\xi)=\valjp{\xi} f(\xi)$ is continuous from $\mathscr S(\R^{d})$ into itself; 
this is not the case of the multiplication by $\val \xi$, which is singular at the origin).
Of course multipliers of $\mathscr S$ can be readily extended to multipliers of $\mathscr S'$ by duality.
We justify below in Claim \ref{cla.2929} the choice $(2+\val\xi^{2})$
in Formula \eqref{233233};
in particular using the more familiar $(1+\val\xi^{2})$ does not change the definition of the space, but would not be suitable for the definition of the norm of the Banach algebra.}
\begin{equation}\label{233233}
\valjp{\xi}=\bigl(2+\val \xi^{2}\bigr)^{1/2}.
\end{equation}
\begin{definition}\rm Let $s$ be a real number. We define the normed vector space
\begin{equation}\label{}
\mathcal V^{s,\omega}=\{v\in \mathscr S'(\R^{d}), \valjp{\xi}^{s}\hat v(\xi)\in L^{1}(\R^{d})\}=
\valjp{D}^{-s}\mathcal F\bigl(
L^{1}(\R^{d})\bigr)=\valjp{D}^{-s}\mathcal W,
\end{equation}
\begin{equation}\label{}
\norm{v}_{\mathcal V^{s,\omega}}=\norm{\valjp{\xi}^{s}\hat v(\xi)}_{L^{1}(\R^{d})}=
\norm{\mathcal F(\valjp{D}^{s}v)}_{L^{1}(\R^{d})}=\norm{\valjp{D}^{s}v}_{\mathcal W}.
\end{equation}
We define as well
\begin{equation}\label{amw854}
\mathcal V^{\io,\omega}=\cap_{s\ge 0}\mathcal V^{s,\omega}.
\end{equation}
\end{definition}
\begin{nb}
 We note that a tempered distribution $v$ belongs to $\mathcal V^{s,\omega}$ means
 $$
 {\valjp{D}}^{s}v\in \mathcal V^{0,\omega}=\mathcal W\subset L^{\io},
 $$
 so this implies that $v\in W^{s,\io}$. As a result we have that $\mathcal V^{s,\omega}\subset  W^{s,\io}$
 with the same homogeneity; one should think of $\omega$ as a ``larger'' version of 
 $\io$ and this explains our notation.\end{nb}
\begin{lemma}\label{lem.nh55}
For all real $s$, $\mathcal V^{s,\omega}$ is a Banach space, isometrically isomorphic to $\mathcal W$.
 Let $m(D)$ be a singular integral introduced in Theorem \ref{thm.singsing};
then $m(D)$ is a continuous endomorphism of $\mathcal V^{s,\omega}$.
\end{lemma}
\begin{proof}
 Let us consider a multiplier $m(D)$ as in Theorem \ref{thm.singsing} and $v\in \mathcal V^{s,\omega}.$
We have
$m(D)=\mathcal F^{-1} m(\xi) \mathcal F$ and 
$
\valjp{\xi}^{s}\mathcal F\bigl(m(D) v\bigr)(\xi)=\valjp{\xi}^{s} m(\xi)\hat v(\xi).
$
Since $m(\xi)$ is bounded, we get,
$$
\norm{m(D)v}_{\mathcal V^{s,\omega}}=\int \valjp{\xi}^{s}\val{m(\xi)\hat v(\xi)}d\xi
\le C\int \valjp{\xi}^{s}\val{\hat v(\xi)}d\xi=C\norm{v}_{\mathcal V^{s,\omega}},
$$ 
concluding the proof of the lemma.
\end{proof}
\begin{remark}\label{rem.kj87}\rm
We have
$\mathcal V^{0,\omega}=\mathcal W$ and, 
 when $s$ is a non-negative integer, we have 
 \begin{equation}\label{wiene+}
\mathcal V^{s,\omega}=\{v\in\mathcal W, \forall \alpha\in \N^{d} \text{ with $\val{\alpha}\le s$},\  \p_{x}^{\alpha} v\in \mathcal W \}.
\end{equation}
This is proven in \eqref{ytr147} in our Appendix.
\end{remark}
\begin{proposition}\label{pro.alg123} Let $s$ be a non-negative real number.
Then $\mathcal V^{s,\omega}$ is a Banach algebra for the multiplication of functions.
\end{proposition}
\begin{proof}
It is obvious that $\mathcal V^{s,\omega}$ is a Banach space isomorphic to  $\mathcal V^{0,\omega}=\mathcal W$.
Let us check the algebra property: let $v_{1}, v_{2}$ be in $\mathcal V^{s,\omega}$ for some $s\ge 0$. 
We have thus $v_{1}, v_{2}\in\mathcal W$ and 
$\widehat{v_{1}v_{2}}=\hat{v}_{1}\ast\hat{v}_{2}$ (convolution of $L^{1}$ functions),
and
we have 
\begin{multline}\label{res555}
\valjp{\xi}^{s}\widehat{v_{1}v_{2}}(\xi)=\int \hat{v}_{1}(\eta)
\hat{v}_{2}(\xi-\eta) 
\valjp{\xi}^{s}
d\eta 
\\
=\int \frac{\valjp{\xi}^{s}}{\valjp{\xi-\eta}^{s}\valjp{\eta}^{s}}
\bigl[\hat{v}_{1}(\eta)\valjp{\eta}^{s}\bigr]
\bigl[\hat{v}_{2}(\xi-\eta)\valjp{\xi-\eta}^{s}\bigr]
d\eta.
\end{multline}
\begin{claim}\label{cla.2929}
 For $\xi_{1}, \xi_{2}\in \R^{d}$, we have 
$\frac{2+\val{ \xi_{1}+\xi_{2}}^{2}}{(2+\val{\xi_{1}}^{2})(2+\val{\xi_{2}}^{2})}\le 1$
(cf. Lemma \ref{cla.214jhg} in our  Appendix).
\end{claim}
\no Going back to \eqref{res555}, we find that 
\begin{multline*}
\int\valjp{\xi}^{s}\val{\widehat{v_{1}v_{2}}(\xi)}d\xi\le
\iint 
\bigl\vert\hat{v}_{1}(\eta)\valjp{\eta}^{s}\bigr\vert
\bigl\vert\hat{v}_{2}(\xi-\eta)\valjp{\xi-\eta}^{s}\bigr\vert
d\eta
d\xi \\=\norm{\hat{v}_{1}(\eta)\valjp{\eta}^{s}}_{L^{1}}
\norm{\hat{v}_{2}(\eta)\valjp{\eta}^{s}}_{L^{1}}=\norm{v_{1}}_{\mathcal V^{s,\omega}}
\norm{v_{2}}_{\mathcal V^{s,\omega}}.
\end{multline*}
\end{proof}
\begin{lemma}\label{lem.210210}
 Let $v$ be a vector field satisfying the assumptions of Conjecture \ref{coj001}. Let $W=\widehat v$ be the Fourier transform of $v$. Then $W$ belongs to $L^{1}(\R^{d})$, i.e. $v$ belongs to the Wiener algebra $\mathcal W$.
 Moreover, we have for any $p\in [1,\frac65)$, any $q\in [1,2]$,
 \begin{align}
&\int_{\val \xi\le 1}\val{W(\xi)}^{p}d\xi<+\io,\label{aer981}
\\
&\int_{\val \xi\ge  1}\val{W(\xi)}^{q}d\xi<+\io,\quad\
 \label{aer982}
 \\
 \forall \rho>0,\quad 
&\int_{\val \xi\ge \rho}\val{W(\xi)}^{2}d\xi\le \nu^{-4}\rho^{-4}\norm{\rot v}_{L^{2}}^{6}.
\label{tra297}
\end{align}
\end{lemma}
\begin{remark}\rm
 This lemma proves in particular that 
 \begin{align}\label{}
&W\in L^{1}(\R^{d}),\quad \sup_{\rho>0}\norm{\mathbf 1_{\{\val \xi\ge \rho\}} \rho^{2}W(\xi)}_{L^{2}(\R^{d})}\le \nu^{-2}\norm{\rot v}_{L^{2}}^{3}<+\io,
\\
&\mathbf 1_{\{\val \xi\le 1\}} W(\xi)\in \bigcap_{0\le \theta<\frac15} L^{1+\theta}(\R^{d}).
\end{align}
\end{remark}
\begin{proof}
 According to Lemma \ref{lem.reg1}, we know that $W$ belongs to $L^{1}$ and we also 
have,
since $\rot v, \nabla v$ belong to $L^{2}$ (thanks to Lemma \ref{lem.reg001}),
as well as $-\Delta v=\rot^{2}v$ (from Lemma \ref{lem.reg1}),
we get 
 \begin{equation}\label{}
 \int \val\xi^{2}\val{W(\xi)}^{2}<+\io, \quad  \int \val\xi^{4}\val{W(\xi)}^{2}<+\io.
\end{equation}
We have also for any $r\in[2,3)$,
$$
\mathbf 1_{\{\val \xi\le 1\}}{W(\xi)}=
\underbracket[0.2pt]{\bigl[\val\xi^{-1}\mathbf 1_{\{\val \xi\le 1\}}\bigr]}_{\in L^{r}(\R^{3})} 
\underbracket[0.2pt]{\bigl[\val\xi\val{W(\xi)}\bigr]}_{\in L^{2}(\R^{3})} ,
$$
so that $\mathbf 1_{\{\val \xi\le 1\}}{W(\xi)}$ belongs to $L^{p}$
with
$$
\frac1p=\frac{1}r+\frac12=\frac{2+r}{2r},
$$
implying that 
$
\mathbf 1_{\{\val \xi\le 1\}}{W(\xi)}\in \cap_{1\le p<6/5} L^{p},
$
yielding \eqref{aer981}.
We have also 
\begin{equation}\label{rea413}
\mathbf 1_{\{\val \xi\ge 1\}}{W(\xi)}=
\underbracket[0.2pt]{\bigl[\val\xi^{-2}\mathbf 1_{\{\val \xi\ge 1\}}\bigr]}_{\in L^{\io}(\R^{3})} 
\underbracket[0.2pt]{\bigl[\val\xi^{2}\val{W(\xi)}\bigr]}_{\in L^{2}(\R^{3})},
\quad\text{which belongs to $L^{2}$},
\end{equation}
as well as for any $s\in [2,+\io],$
$$
\mathbf 1_{\{\val \xi\ge 1\}}{W(\xi)}=
\underbracket[0.2pt]{\bigl[\val\xi^{-2}\mathbf 1_{\{\val \xi\ge 1\}}\bigr]}_{\in L^{s}(\R^{3})} 
\underbracket[0.2pt]{\bigl[\val\xi^{2}\val{W(\xi)}\bigr]}_{\in L^{2}(\R^{3})},
\quad\text{which belongs to $L^{q}$},
$$
with
$$
\frac1q=\frac{1}{s}+\frac12=\frac{2+s}{2s}, \quad
q=\frac{2s}{2+s}, \text{ browsing }[1,2],$$
yielding \eqref{aer982}. 
To obtain \eqref{tra297}, we use \eqref{rea413} with $\{\val \xi\ge \rho\}$ replacing $\{\val \xi\ge 1\}$,
yielding the upperbound $\rho^{-4}\norm{\rot^{2}v}_{L^{2}}^{2}$ for the left-hand-side of \eqref{tra297}; 
a simple interpolation in \eqref{reg004} gives
$$
\norm{\rot^{2}v}_{L^{2}}\lesssim \nu^{-2}\gamma_{0}^{3},
$$
which yields \eqref{tra297}.
\end{proof}
\begin{nb}
 Let us check that \eqref{tra297} is units-consistent:
 we have, using \eqref{oig1146}, $[\rho]=\mathtt{ L^{-1}}$, 
 $$
 [\  \nu^{-4}\rho^{-4}\norm{\rot v}_{L^{2}}^{6}]=\mathtt{(L^{2}T^{-1})^{-4}L^{4}}\mathtt{(L^{\frac32}T^{-1})^{6}}
 =\mathtt{L^{5}T^{-2}},
 $$
 whereas the units of the left-hand-side of  \eqref{tra297} are indeed
 $$
 [\hat v]^{2} \mathtt{L^{-3}}\underbracket[0pt]=_{\eqref{unit12}}\mathtt{(L^{4}T^{-1})^{2}L^{-3}}=\mathtt{L^{5}T^{-2}}.
 $$
 Moreover \eqref{tra297} is invariant by the change of function  \eqref{126126}: applying that inequality  to $v_{\lambda}$ given by $v_{\lambda}(x)=\lambda v(\lambda x)$ we get,
 using 
 $
 \widehat{v}_{\lambda}(\xi)=\lambda^{1-3}\hat v(\xi\lambda^{-1}),
 $
 $$
 \int_{\val \xi\ge \rho}\Val{\lambda^{-2}\hat v(\xi\lambda^{-1})}^{2}d\xi\le \nu^{-4}\rho^{-4}
 \Bigl(\int_{\R^{3}}
 \val{\lambda^{2}(\rot v)(\lambda x)}^{2} dx
 \Bigr)^{3},
 $$
 that is
 \begin{multline*}
 \lambda^{-1}
\int_{\val \eta\ge \rho\lambda^{-1}}\Val{\hat v(\eta)}^{2}d\eta
=\lambda^{-4}
 \int_{\val \eta\ge \rho\lambda^{-1}}\Val{\hat v(\eta)}^{2}\lambda^{3}d\eta
 \\\le \nu^{-4}\rho^{-4}\lambda^{12}
 \Bigl(\int_{\R^{3}}
 \val{(\rot v)(y)}^{2} \lambda^{-3}dy
 \Bigr)^{3}
 =
  \nu^{-4}\rho^{-4}\lambda^{3}
 \Bigl(\int_{\R^{3}}
 \val{(\rot v)(y)}^{2} dy
 \Bigr)^{3}
 \\=
 \lambda^{-1} \nu^{-4}(\rho \lambda^{-1})^{-4}
 \Bigl(\int_{\R^{3}}
 \val{(\rot v)(y)}^{2} dy
 \Bigr)^{3},
\end{multline*}
 which is the same inequality as  \eqref{tra297} where $\rho$ is replaced by $\rho/\lambda$. Since $\rho$ is an arbitrary positive parameter, this is the same family of inequalities.
 \end{nb}
\begin{lemma}\label{lem.212212}
 Let $v$ be a vector field satisfying the assumptions of Conjecture \ref{coj001}. Then $\widehat v$ belongs to $L^{1}(\R^{d})$ and for all $\kappa\in(-\frac12, \frac12)$, we have
 \begin{equation}\label{}
\int_{\R^{d}}\val\xi^{\kappa}\val{\widehat v(\xi)}d\xi\le 2\sqrt{\pi}
\Bigl(
\frac{\norm{\rot v}_{L^{2}}}{\sqrt{1+2\kappa}}
+\frac{\norm{\rot^2{v}}_{L^{2}}}{\sqrt{1-2\kappa}}
\Bigr)<
+\io.
\end{equation}
In particular we find that $v$ belongs to $\mathcal V^{\kappa,\omega}$ for any $\kappa\in [0,\frac12).$
\end{lemma}
\begin{proof}
The fact that $\hat v$ belongs to $L^{1}(\R^{d})$ follows from Lemma \ref{lem.210210} and moreover we have 
\begin{multline*}
\int_{\{\val \xi\le 1\}} \val\xi^{\kappa}\val{\hat v(\xi)}d\xi=\int_{\{\val \xi\le 1\}} \val\xi^{\kappa-1}\val \xi\val{\hat v(\xi)}d\xi
\quad\text{\footnotesize (and since $2\kappa-2>-3$, i.e., $\kappa>-1/2$)}
\\
\le \norm{\mathbf 1_{
\{\val \xi\le 1\}
} \val\xi^{\kappa-1}}_{L^{2}}\norm{\rot v}_{L^{2}}=\frac{2\sqrt{\pi}}{\sqrt{1+2\kappa}}\norm{\rot v}_{L^{2}}.
\end{multline*}
Similarly, we obtain
\begin{multline*}
\int_{\{\val \xi\ge 1\}} \val\xi^{\kappa}\val{\hat v(\xi)}d\xi=\int_{\{\val \xi\ge 1\}} \val\xi^{\kappa-2}\val \xi^{2}\val{\hat v(\xi)}d\xi
\quad\text{\footnotesize (and since $2\kappa-4<-3$) i.e. $\kappa<1/2$)}
\\
\le \norm{\mathbf 1_{
\{\val \xi\ge 1\}
} \val\xi^{\kappa-2}}_{L^{2}}\norm{\rot^{2} v}_{L^{2}}=\frac{2\sqrt{\pi}}{\sqrt{1-2\kappa}}\norm{\rot^{2} v}_{L^{2}},
\end{multline*}
completing the proof of the lemma (note that from \eqref{reg004}, we have  $\rot^{2}v\in L^{2}$).
\end{proof}
 \begin{remark}\rm
 We shall see later on that we can remove the upper bound restriction on $\kappa$, that is to prove that
  for any $\kappa>-1/2$,
 \begin{equation}\label{fds896}
\int_{\R^{d}} \val\xi^{\kappa}\val{\hat v(\xi)}d\xi<+\io,
\end{equation}
but we shall not be able to remove the restriction $\kappa>-1/2$.
Moreover, we have to keep in mind that for any $\varepsilon>0$ we have proven  
\begin{equation}\label{tre658}
\int_{\val \xi\le 1}\val{\xi}^{-\frac12+\varepsilon}\val{\hat v(\xi)} d\xi<+\io,
\end{equation}
which is somehow a regularity result for the Fourier transform of $v$ at the origin.
We expect that this regularity result for $\hat v$ could be translated into a decay result for $v$ at infinity.
Except for the fact that we cannot make $\varepsilon=0$ in \eqref{tre658}, we may note, quite heuristically, that $\hat v$ could not be more singular than $\val\xi^{-5/2}$ near the origin, 
which is related to a decay of type $\val x^{\frac52-3}=\val x^{-\frac12}$ at infinity for $v$,
compatible with a vector $v$ in $L^{6,\infty}(\R^{3})$.
We know also that $\rot v\times v$  as well as $\rot^{2}v$
belong to $L^{3/2}$, so that 
$\val \xi^{2}\hat v(\xi)$ belongs to $L^{3}$, which means in particular that
\begin{equation}\label{vcw634}
\int_{\val \xi\le 1}\val{\xi}^{6}\val{\hat v(\xi)}^{3} d\xi<+\io,
\end{equation}
and this indicates that $\hat v$ could not be more singular than $\val \xi^{-3}$.
As a result, there is actually a discrepancy between \eqref{tre658} and \eqref{vcw634},
so that  \eqref{tre658} appears as a better estimate.
\end{remark}
\begin{theorem}\label{thm.214trq}
  Let $v$ be a vector field satisfying the assumptions of Conjecture \ref{coj001}.
  Then for all $s\ge 0$, $v$ belongs to $\mathcal V^{s,\omega}$. 
\end{theorem}
\begin{proof}[Proof of the theorem]
Note that Lemma \ref{lem.212212} provides us with the information $v\in\mathcal V^{s,\omega}$ for $s\in [0,\frac12)$.
We begin with another lemma.
\begin{lemma}\label{lem.newpp+}
Let $v$ be a vector field satisfying the assumptions of Conjecture \ref{coj001}.
Then we have
 \begin{equation}\label{21821+}
\mathbb P\bigl((\curl v)\times v\bigr)=\sum_{1\le j\le 3}\p_{j}\mathbb P \bigl(v_{j}v\bigr)-
\frac12\mathbb P\nabla{(\val v^{2})},
\end{equation}
and 
we have with $(\vec{e}_{\ell})_{1\le \ell \le 3}$ standing for the canonical basis of $\R^{3}$, 
\begin{equation}\label{2182++}
\mathbb P\bigl((\curl v)\times v\bigr)=\sum_{1\le \ell\le 3}\p_{\ell}u_{\ell},
\quad
\text{with }
u_{\ell}=\mathbb P\bigl(v_{\ell} v-\frac12 \val{v}^{2}\vec{e}_{\ell}\bigr)\in \cap_{0\le \kappa<1/2}\mathcal V^{\kappa,\omega}.
\end{equation}
We shall use the following notation: 
$$
S_{0}(v\otimes v)=\Bigl(\mathbb P\bigl(v_{\ell} v-\frac12 \val{v}^{2}\vec{e}_{\ell}\bigr)\Bigr)_{1\le \ell \le 3},\quad \nabla\cdot
S_{0}(v\otimes v)
=\sum_{1\le \ell\le 3}\p_{\ell}\Bigl(\mathbb P\bigl(v_{\ell} v-\frac12 \val{v}^{2}\vec{e}_{\ell}\bigr)\Bigr),
$$
so that we have 
\begin{equation}\label{wlf197}
\mathbb P\bigl((\curl v)\times v\bigr)=\nabla\cdot
S_{0}(v\otimes v).
\end{equation}
\end{lemma}
\begin{nb}
 We note that the left-hand-side of \eqref{21821+} makes sense as an $L^{3/2}\cap L^{3}$ function,
 thanks to \eqref{reg003} in Lemma \ref{lem.reg1}. The right-hand-side  of \eqref{21821+} makes 
 sense since $v_{j}v, \val{v}^{2}$ belong to $L^{3}$ so that $\mathbb P(v_{j}v), \mathbb P(\val v^{2})$ make sense and are in $L^{3}$; moreover, from Lemma \ref{lem.reg1}, we have 
 $$
 \p_{k}(v_{j}v)=\underbracket[0.1pt]{(\p_{k}v_{j})}_{L^{2}} \underbracket[0.1pt]{v}_{L^{6}\cap L^{\io}}+ 
 \underbracket[0.1pt]{v_{j}}_{L^{6}\cap L^{\io}}\underbracket[0.1pt]{(\p_{k}v)}_{L^{2}},
 \quad\text{which belongs thus to $L^{2}\cap L^{3/2}$,}
 $$
 so that 
 $\mathbb P\bigl(\p_{k}(v_{j}v)\bigr)$ makes sense and belongs to $L^{2}\cap L^{3/2}$.
 Although this lemma looks somewhat ``obvious'', or a consequence of  Lemma \ref{lem.newppp},
 it deserves a proof.
 \end{nb}
 \begin{proof}[Proof of the lemma]
 Let $\phi$ be a vector field in $\mathscr S(\R^{3}, \R^{3})$: we have, with $\rho_{\varepsilon}$ a standard mollifier,  
 \begin{align*}
 &\poscal{(v\cdot \nabla)v}{\phi}_{\mathscr S',\mathscr S}
 -\poscal{\p_{j}(v_{j}v)}{\phi}_{\mathscr S',\mathscr S}
 \\
 &\hskip10pt = \poscal{\underbracket[0.2pt]{\p_{j}v}_{L^{2}}}{\underbracket[0.2pt]{v_{j}}_{L^{6}}
 \underbracket[0.2pt]\phi_{L^{3}}}_{L^{2},L^{2}}
 + \poscal{v_{j}v}{\p_{j}\phi}_{\mathscr S',\mathscr S}
 \\ &\hskip10pt =\lim_{\varepsilon\rightarrow 0}
 \poscal{\p_{j}v}{(\rho_{\varepsilon}\ast v_{j})\phi}_{L^{2},L^{2}}+ \poscal{v_{j}v}{\p_{j}\phi}_{\mathscr S',\mathscr S}
 \\
  &\hskip10pt =-\lim_{\varepsilon\rightarrow 0}
 \poscal{v}{\p_{j}\bigl[(\rho_{\varepsilon}\ast v_{j})\phi\bigr]}_{\mathscr S',\mathscr S}
 + \poscal{v_{j}v}{\p_{j}\phi}_{\mathscr S',\mathscr S}
\\
  &\hskip10pt =-\lim_{\varepsilon\rightarrow 0}
 \poscal{\underbrace{v}_{L^{6}}}{\underbrace{(\rho_{\varepsilon}\ast v_{j})}_{L^{6}}
 \underbrace{(\p_{j}\phi)}_{L^{3/2}}
 }_{L^{6},L^{6/5}}
  + \poscal{v_{j}v}{\p_{j}\phi}_{L^{3},L^{3/2}}\quad\text{\tiny (here we've used that $\dive v=0$)}
 \\ &\hskip10pt =-\int v_{j}(x)\proscal3{v(x)}{(\p_{j}\phi)(x)} dx
 +\int v_{j}(x)\proscal3{v(x)}{(\p_{j}\phi)(x)} dx=0,
\end{align*}
so that we get the formulas,
\begin{align}\label{}
&(\rot v)\times v=\sum_{1\le j\le 3}\p_{j}(v_{j}v)-\frac12\nabla{(\val v^{2})},
\label{afre87}\\
&\mathbb P\bigl((\rot v)\times v\bigr)= \sum_{1\le j\le 3}\p_{j}\bigl(\mathbb P(v_{j}v)\bigr)
-\frac12 \sum_{1\le \ell\le 3}\p_{\ell}\mathbb P(\val v^{2}\vec{e}_{\ell}),
\label{dfry87}\\
&\mathbb P\bigl((\rot v)\times v\bigr)\in 
\sum_{1\le \ell\le 3}\p_{\ell}
\Bigl(\underbracket[0.2pt]{
\mathbb P(v_{\ell }v)-\frac12\mathbb P(\val{v}^{2}\vec{e}_{\ell})
}_{=u_{\ell}\in \mathcal V^{\kappa,\omega} \text{ for $\kappa\in [0,1/2)$}}\Bigr),
\label{dfre87}
\end{align}
 since we note that for all $\kappa\in [0,1/2)$, we have $v\in \mathcal V^{\kappa,\omega}$ (Lemma \ref{lem.212212}) and from Proposition 
 \ref{pro.alg123}, we get 
 $$
 v_{\ell}v\in \mathcal V^{\kappa,\omega}, \mathbb P(v_{\ell}v)\in \mathcal V^{\kappa,\omega},
 \val v^{2}\vec{e}_{\ell}\in \mathcal V^{\kappa,\omega}, \mathbb P\val v^{2}\vec{e}_{\ell}\in \mathcal V^{\kappa,\omega},
 $$
 concluding the proof of the lemma.
\end{proof}
Let us go back now to the proof of Theorem \ref{thm.214trq}. With $(u_{\ell})_{1\le \ell\le 3}$
given by Lemma \ref{lem.newpp+},
Equation \eqref{newway} reads,
\begin{equation}\label{}
\nu \rot^{2}v+\sum_{1\le \ell\le 3}\p_{\ell}u_{\ell}=0.
\end{equation}
With $W=\widehat v, U_{\ell}=\widehat{u_{\ell}}$, 
we obtain $W,U_{\ell}$ in $L^{1}(\R^{3})$ and  for $\kappa\in [0,\frac12)$, we have 
\begin{gather}\label{}
\nu\val{\xi}^{2} W(\xi)+i\sum_{1\le \ell\le 3}\xi_{\ell}U_{\ell}(\xi)=0,
\qquad
\int \bigl(\val{W(\xi)}+\val{U_{l}(\xi)}\bigr)\valjp{\xi}^{\kappa} d\xi<+\io,
\end{gather}
the latter inequality following from \eqref{2182++}.
This means that 
\begin{align*}
&\val\xi\Bigl(\nu\underbracket[0.2pt]{\val{\xi} W(\xi)}_{\in L^{1}_{\text{loc}}(\R^{3})}+i\sum_{1\le \ell\le 3}\underbracket[0.2pt]{\frac{\xi_{\ell}}{\val\xi}U_{\ell}(\xi)}_{\in L^{1}(\R^{3})}\Bigr)=0,
\text{\quad which implies }
\\
&
\nu{\val{\xi} W(\xi)}+i\sum_{1\le \ell\le 3}{\frac{\xi_{\ell}}{\val\xi}U_{\ell}(\xi)}=0,
\text{\quad (equality of $L^{1}_{\text{loc}}(\R^{3})$ functions),}
\end{align*}
and thus for any $\kappa\in [0,1/2)$, we have, 
\begin{equation}\label{asw591}
\nu\val{\xi}^{1+\kappa} W(\xi)+i\sum_{1\le \ell\le 3}\frac{\xi_{\ell}}{\val\xi}\val{\xi}^{\kappa}U_{\ell}(\xi)=0,
\end{equation}
so that 
\begin{equation}\label{aqc149}
\nu\int\val{W(\xi)}\val\xi^{1+\kappa} d\xi\le \sum_{1\le \ell\le 3}\int \val{\xi}^{\kappa} \val{U_{\ell}(\xi)} d\xi<+\io,
\end{equation}
proving that $v$ belongs to $\mathcal V^{\kappa+1,\omega}$,
since, recalling the notation \eqref{233233}, we have 
\begin{align*}
\int\valjp{\xi}^{1+\kappa}\val{W(\xi)} d\xi&\le \int_{\val \xi\le 1}\valjp{\xi}^{1+\kappa}\val{W(\xi)} d\xi
+ \int_{\val \xi\ge 1}\valjp{\xi}^{1+\kappa}\val{W(\xi)} d\xi
\\
&\le 3^{\frac{1+\kappa}2}\int_{\val \xi\le 1}\val{W(\xi)} d\xi
+3^{\frac{1+\kappa}2}
\int_{\val \xi\ge 1}\val{\xi}^{1+\kappa}\val{W(\xi)} d\xi
\\
\text{\tiny (using \eqref{aqc149})}&\le  3^{\frac{1+\kappa}2}\norm{v}_{\mathcal W}
+3^{\frac{1+\kappa}2}\nu^{-1}\sum_{1\le \ell\le 3}\int \val{\xi}^{\kappa} \val{U_{\ell}(\xi)} d\xi<+\io.
\end{align*}
We know now that $v\in \mathcal V^{\kappa,\omega}$ for all $\kappa\in [0,\frac12)\cup [1,\frac32),$
and since  the space $\mathcal V^{\kappa,\omega}$ is decreasing when $\kappa$ increases,
we get that
$$
v\in \cap_{\kappa\in [0,\frac32)}\mathcal V^{\kappa,\omega}.
$$
Going back to the equation
$$
\nu \rot^{2}v+\nabla\cdot
S_{0}(v\otimes v)=0,
$$
we get, using the notations of Lemma \ref{lem.newpp+},
$$
\nu \val \xi W(\xi)
+i\frac{\xi}{\val{\xi}}\cdot \mathbf U(\xi)=0, \quad \mathbf U(\xi)=\mathcal F\bigl(S_{0}(v\otimes v)\bigr)(\xi),$$
and thus for $\kappa\in [0,3/2)$, we have 
$
\nu {\val \xi }^{1+\kappa}W(\xi)
+i\frac{\xi}{\val{\xi}}\cdot \val{\xi}^{\kappa}\mathbf U(\xi)=0,
$
so that
\begin{equation}\label{wta654}
\nu\int_{\R^{3}}\val{W(\xi)}\val{\xi}^{1+\kappa} d\xi
\le \int_{\R^{3}}\val{\mathbf U(\xi)}\val\xi^{\kappa}d\xi.
\end{equation}
We note that $v\in \mathcal V^{\kappa,\omega}$ and thus  
$v\otimes v$ belongs to $\mathcal V^{\kappa,\omega}$, as well as $S_{0}(v\otimes v)$,
which implies that the right-hand-side of \eqref{wta654} is finite for any $\kappa\in [0,3/2)$.
Since $W$ belongs to $L^{1}(\R^{3})$, we obtain that 
$v\in \mathcal V^{\kappa+1,\omega}$, that is $v\in \mathcal V^{\kappa,\omega}$ for $\kappa\in [0,5/2)$.
Iterating that reasoning, we get the sought result.
\end{proof}
\begin{corollary}\label{cor.kjha45}
 Let $v$ be a vector field satisfying the assumptions of Conjecture \ref{coj001}. Then $v$ is a $C^{\io}$ function on $\R^{3}$ such that for all multi-indices $\alpha\in \N^{3}$,
 we have
 $$
 \lim_{\val{x}\rightarrow+\io}(\p_{x}^{\alpha} u)(x)=0.
 $$
\end{corollary}
\begin{proof}
 We simply note that for $k\in \N$,
 $
 \mathcal V^{k,\omega}\subset C^{k}_{(0)},
 $
 the $C^{k}$ functions with limit $0$ at infinity as well as their derivatives of order less than $k$.
\end{proof}
\begin{remark}\rm
 The result of Corollary
 \ref{cor.kjha45} was already obtained in the works of G.P.~Galdi (Theorem X.1.1 in \cite{MR2808162}).
 However, the result of Theorem 2.14 is slightly better, since the inclusion 
 $
 \mathcal W\subset C^{0}_{(0)}\quad\text{is strict}
 $,
 although the scalings of the two spaces are identical, and identical to the scaling of $L^{\io}$.
 Moreover, as said above, the algebras $\mathcal V^{s,\omega}$ for $s\ge 0$ are stable by standard singular integrals,
 which is a convenient feature, failing to hold true for $C^{0}_{(0)}$ (and also for $L^{\io}$).
 In some sense the ``Sobolev spaces'' based upon $\mathcal W$ appear naturally in the problem under scope  and keep a regularity information which seems to be optimal. On the other hand, with $R_{0}$ defined in \eqref{ddd001} and the Bernoulli head pressure $Q$ given by \eqref{modpr0}, we have 
\begin{equation}\label{modpre}
 Q=R_{0}(v\otimes v)+\frac12\val v
^{2},
\end{equation}
and we get immediately from Theorem \ref{thm.214trq} that 
\begin{equation}\label{xcv687}
Q\in 
\mathcal V^{\infty,\omega}=\cap_{s\ge 0}\mathcal V^{s,\omega}.
\end{equation}
We note also that $Q$ is real-valued since 
each $R_{0,j}(\xi)$ in \eqref{ddd001} is even and real-valued
(see as well Footnote \ref{foot.003} on page \pageref{foot.003}).
 \end{remark}
\subsection{Regularity results, continued}
We want first to go back to Lemma \ref{lem.newppp} and formulate some variation.
\begin{lemma}\label{lem.218}
 Let $v$ be a vector field on $\R^{3}$ belonging to $\mathcal W$,
 such that $\curl v$ belongs to $L^{2}(\R^{3}, \R^{3})$ and $\dive v=0$. Then
 the matrix $\nabla v$ belongs to $L^{2}(\R^{3}, \R^{9})$, $v\in L^{6}(\R^{3}, \R^{3})$
 and we have
 \begin{align}
(\curl v)\times v&= (v\cdot \nabla) v-\frac12\nabla( \val v^{2}), \label{0182++}
\\
\mathbb P\bigl((\curl v)\times v\bigr)&=\mathbb P\bigl((v\cdot \nabla) v\bigr).\label{218+++}
\end{align}
\end{lemma}
\begin{nb}
 Let us note that both sides of each equality make sense:
 Lemma \ref{lem.reg001} ensures that $\nabla v$ belongs to $L^{2}(\R^{3}, \R^{9})$ since $\hat v$ belongs to $L^{1}(\R^{3})$, $\rot v\in L^{2}$ and $\dive v=0$.
 Lemma \ref{lem.21ml} gives the property $v\in L^{6}(\R^{3}, \R^{3})$.
 As a result, since $\mathcal W\subset L^{\io}$, 
 we get $(\curl v)\times v\in L^{2}\cap L^{3/2}$, so that Theorem \ref{thm.singsing} gives a meaning to both left-hand-sides.
 Checking $v_{j}\p_{j}v$, we see that $v_{j}\in L^{\io}\cap L^{6}$ and $\p_{j}v$ belongs to $L^{2}$, so that this product is in $L^{2}\cap L^{3/2}$.
 Moreover $\val{v}^{2}$ belongs to  $\mathcal W\cap L^{3}$, thanks to Proposition \ref{pro.alg123} and to the fact that $v$ belongs to $L^{6}$.
\end{nb}
\begin{proof}
 The proof of Lemma \ref{lem.newppp} provides \eqref{0182++} {\it ne varietur}.
 We calculate now
 $$
 \mathcal F\bigl[\nabla (\val v^{2})\bigr](\xi)=i\xi U(\xi), \quad U\in L^{1}(\R^{3},\C),
 \quad\text{since $\val v^{2}$ belongs to $\mathcal W$.}
 $$
 As a consequence
 $
 \mathcal F\bigl[\nabla (\val v^{2})\bigr]
 $
 belongs to $L^{1}_{\text{loc}}$, so that 
 $$
 \mathbb P(\xi) \Bigl(\mathcal F\bigl[\nabla (\val v^{2})\bigr]\Bigr)(\xi)=
 U(\xi) \mathbb P(\xi) \xi=0,
 $$
 proving
 \eqref{218+++}.
\end{proof}
\begin{remark}{\rm 
 It would have simplified a bit the proof if we had assumed, as we could, that $v$ is a vector field satisfying the assumptions of Conjecture \ref{coj001}, which imply all our assumptions in the above lemma. Then, in the last step, we could have used that $\xi U(\xi)$ belongs to $L^{1}(\R^{3}, \R^{3})$,
 thanks to Theorem \ref{thm.214trq}.
Our formulation is keeping that lemma linear and not valid only for our solution of a non-linear equation.}
\end{remark}
\begin{theorem}\label{thm.214t++}
  Let $v$ be a vector field satisfying the assumptions of Conjecture \ref{coj001}.
  Then for all integers $k\ge 0$, $\nabla v$ belongs to $W^{k,2}\cap\mathcal V^{k,\omega}$. 
\end{theorem}
\begin{proof}
Theorem \ref{thm.214trq} gives $v\in \mathcal V^{\infty, \omega}$, so that we need only to prove 
$$\nabla v\in W^{\io, 2}=\cap_{k\ge 0}W^{k,2}=\{u\in L^{2}, \forall \alpha\in \N^{3}, \p_{x}^{\alpha} u\in L^{2}\}.
$$
We note that, with $W=\hat v$, we already know 
from the hypothesis $\rot v\in L^{2}, \dive v=0$ and Lemma \ref{lem.reg001} on the one hand, and from Theorem \ref{thm.214trq}  on the other hand that
\begin{equation}\label{initia}
\val\xi W(\xi)\in L^{1}\cap L^{2}.
\end{equation}
We are going to prove inductively for $k$ integer $\ge 1$ that
\begin{equation}\label{induct}
\val\xi ^{k}W(\xi)\in  L^{2},
\end{equation}
a true statement for $k=1$ from \eqref{initia}. Let $k\ge 1$ be given and let us assume that \eqref{induct} holds true.
Thanks to Lemma \ref{lem.218}, 
Equation \eqref{newway} can be written as
\begin{equation*}\label{}
\nu \val{\xi}^{2}W(\xi)+i\mathbb P(\xi)\xi_{j} (W_{j}\ast W)(\xi)=0,\quad W(\xi)\cdot \xi=0,
\end{equation*}
i.e.
\begin{equation}\label{244244}
\nu \val{\xi}^{2}W(\xi)+i\mathbb P(\xi)\xi_{j} \int W_{j}(\eta)W(\xi-\eta)d\eta=0,\quad W(\xi)\cdot \xi=0,
\end{equation}
noting that $W_{j}\ast W$ belongs to $L^{1}$.
We can do better and write the second term in the sum as 
\begin{multline*}
i\mathbb P(\xi) \int
\sum_{1\le j\le 3} \underbracket[0.2pt]{W_{j}(\eta)}_{L^{1}}
\underbracket[0.2pt]{(\xi_{j}-\eta_{j})W(\xi-\eta)}_{
\substack{L^{1}\cap L^{2}\\
\text{from \eqref{initia}}}
}
d\eta+i
\mathbb P(\xi) \int \underbracket[0.2pt]{\sum_{1\le j\le 3}\eta_{j}W_{j}(\eta)}_{=0} W(\xi-\eta) d\eta
\\
=\sum_{1\le j\le 3}i\mathbb P(\xi) \int \underbracket[0.2pt]{W_{j}(\eta)}_{L^{1}}
\underbracket[0.2pt]{(\xi_{j}-\eta_{j})W(\xi-\eta)}_{L^{1}\cap L^{2}} d\eta,
\quad\text{which belongs to $L^{1}\cap L^{2}$},
\end{multline*}
since 
$
L^{1}\ast L^{1}\subset L^{1},\ L^{1}\ast L^{2}\subset L^{2}
$
and $\mathbb P(\xi)$ is a bounded matrix.
We can also write for any $k\ge 1$,
using \eqref{244244},
\begin{equation}\label{}
\nu\val \xi^{k+1} W(\xi)+i\mathbb P(\xi)\val{\xi}^{k}\int\trace{\Bigl(\frac\xi{\val \xi}\otimes W(\eta)\Bigr)} W(\xi-\eta)d\eta=0.
\end{equation}
We infer then 
\begin{align*}
&\nu^{2}\val \xi^{2k+2} \val{W(\xi)}^{2}
\le \val{\xi}^{2k}
\Val{\int \val{W(\eta)}\val{W(\xi-\eta)}d\eta}^{2}
=
\Val{\int \val{\xi}^{k} \val{W(\eta)}\val{W(\xi-\eta)}d\eta}^{2}
\\
&\le 
\Val{\int \bigl(\val{\xi-\eta}+\val{\eta}\bigr)^{k} \val{W(\eta)}\val{W(\xi-\eta)}d\eta}^{2}
\qquad\text{\tiny(now we use \eqref{gwp197} to get to the next line) }
\\
&\le 2^{2k-2}
\Val{\int \bigl(\val{\xi-\eta}^{k}+\val{\eta}^{k}\bigr) \val{W(\eta)}\val{W(\xi-\eta)}d\eta}^{2}
\\
&=
 2^{2k-2}
 \Val{\int  \val{W(\eta)}\val{\xi-\eta}^{k}\val{W(\xi-\eta)}d\eta+
 \int  \val{W(\eta)}\val{\eta}^{k}\val{W(\xi-\eta)}d\eta
 }^{2}
 \\
 &\le 2^{2k-1} \Biggl\vert{\int\underbrace{\val{W(\eta)}}_{L^{1}}
 \underbrace{\val{\xi-\eta}^{k}\val{W(\xi-\eta)}}_{\substack{  L^{2}\text{ from the induction} \\\text{  hypothesis on $k\ge 1$}    } }
 d\eta}
 \Biggr\vert^{2}
 \\&\hskip170pt +2^{2k-1} \Biggl\vert{\int \underbrace{ \val{W(\eta)}\val{\eta}^{k}}_{\substack{
 L^{2} \text{ from the induction}\\\text{ hypothesis on $k\ge 1$}
} }
 \underbrace{\val{W(\xi-\eta)}}_{L^{1}}d\eta}\Biggr\vert^{2}
 \\
 &=\val{\Omega_{k,1}(\xi)}^{2}+\val{\Omega_{k,2}(\xi)}^{2}, \quad\text{with $\Omega_{k,1}, \Omega_{k,2}$ in $L^{2}$},
\end{align*}
so that
$$
\nu^{2}\int\val \xi^{2k+2} \val{W(\xi)}^{2}d\xi\le \norm{\Omega_{k,1}}_{L^{2}}^{2}+\norm{\Omega_{k,2}}_{L^{2}}^{2}<+\io,
$$
proving \eqref{induct} for $k+1$ and the theorem.
\end{proof}
\begin{proposition}\label{pro.gfd987}
Let us define
\begin{equation}\label{247247}
\mathcal A=\{v\in \mathcal V^{\io, \omega}, \text{such that $\forall \alpha\in \N^{d}$ with $\val\alpha\ge 1$}, \p_{x}^{\alpha}v\in L^{2}\}.
\end{equation}
where $\mathcal V^{\io,\omega}$ is defined in \eqref{amw854}. Then $\mathcal A$ is an algebra, stable
under differentiation and  under the action of singular integrals given by Theorem \ref{thm.singsing}.
\end{proposition}
\begin{proof} We note first that $v\in \mathcal A$ is equivalent to $v\in\mathcal W$ (the Wiener algebra)
and such that 
\begin{equation}\label{gqf582}
\forall k\in \N, \quad\int_{\R^{d}} \val\xi^{k}\val{\hat v(\xi)} d\xi+\int_{\R^{d}}  \val\xi^{2k+2}\val{\hat v(\xi)}^{2} d\xi<+\io.
\end{equation}
The set $\mathcal A$ is obviously a vector space, subspace of $\mathcal W$.
Let $v_{1}, v_{2}\in \mathcal A$. We set $W_{j}=\hat v_{j}, j=1,2,$ and $W=\widehat{v_{1}v_{2}}=W_{1}\ast W_{2}$. We have $W_{1}, W_{2}\in L^{1}$ and thus $W_{1}\ast W_{2}\in L^{1}$; for $k\in \N$,
we have 
\begin{align*}
\val\xi^{k+1} &\val{W(\xi)}\le \int\val\xi^{k+1} \val{W_{1}(\eta)}\val{W_{2}(\xi-\eta)}d\eta
\quad\text{\tiny (we use now the triangle inequality and \eqref{gwp197})}
\\
&\le 2^{k}\int\underbrace{\val\eta^{k+1} \val{W_{1}(\eta)}}_{L^{2}\cap L^{1}}
\underbrace{\val{W_{2}(\xi-\eta)}}_{L^{1}}d\eta+2^{k}\int \underbrace{\val{W_{1}(\eta)}}_{L^{1}}
\underbrace{\val{\xi-\eta}^{k+1}\val{W_{2}(\xi-\eta)}}_{L^{2}\cap L^{1}}d\eta
\\
&=\Omega_{1,k+1}(\xi)+\Omega_{2,k+1}(\xi), \quad \Omega_{1,k+1}, \Omega_{2,k+1}\in L^{1}\ast (L^{2}\cap L^{1})\subset L^{2}\cap L^{1}\subset L^{2},
\end{align*}
 which implies that
 $
 \int_{\R^{d}}  \val\xi^{2k+2}\val{W(\xi)}^{2} d\xi<+\io;
 $
also from the previous argument, we obtain that  for $k\ge 1$,
we  have
\begin{align*}
\val\xi^{k} &\val{W(\xi)}\le \Omega_{1,k}(\xi)+\Omega_{2,k}(\xi),
\quad \Omega_{1,k}, \Omega_{2,k}\in L^{1}\ast L^{1}\subset L^{1},
\end{align*}
and 
for $k=0$, we have $W=W_{1}\ast W_{2}\in L^{1}\ast L^{1}\subset L^{1}$,
so that \eqref{gqf582} is proven for $v=v_{1}v_{2}$.
To get stability of $\mathcal A$
by differentiation, we note that for $v\in \mathcal A$,
we have 
$
\widehat{\p_{j}v}(\xi)=i\xi_{j} \hat v(\xi)
$
so that \eqref{gqf582} is satisfied for $\p_{j}v$ as well.
Note that \eqref{gqf582} is stable by a Fourier multiplier $m(D)$ for a bounded $m$, proving the last statement in the proposition.
\end{proof}
Theorem \ref{thm.54yt} follows from the next result.
\begin{corollary}\label{cor.lkj987}
 Let $v$ be a vector field satisfying the assumptions of Conjecture \ref{coj001}.
 Then $v$ belongs to $\mathcal A$ defined by \eqref{247247}.
The Bernoulli head pressure $Q$ defined by \eqref{modpr0} belongs also to the algebra $\mathcal A$, thanks to the expression \eqref{modpre}.
\end{corollary}
\begin{proof}
 This is an immediate consequence of Theorem \ref{thm.214t++} and Proposition \ref{pro.gfd987}.
\end{proof}
\begin{nb}{
 The reader may have noticed that the proof of Theorem \ref{thm.214t++}  is non-linear, in the sense that it is using the fact that $v$ satisfies a non-linear equation \eqref{newway}; however the proof of Proposition \ref{pro.gfd987} is purely linear.
 It may be also a good opportunity to mention that \eqref{newway} is \emph{not} a differential equation,
 but a non-local equation involving a singular integral, in fact close to a pseudo-differential equation, but not exactly in this category since  the symbol of the multiplier is singular at the origin,
 so that low frequencies (i.e. small values of $\val{\xi}$) are playing a key r\^{o}le, in particular for the Liouville problem.
 }
\end{nb}
The following lemma will be useful later on, but clearly belongs to this section.
\begin{lemma}\label{lem.223}
Let $\alpha$ be a $L^{\io}_{\textrm{\rm loc}}$ function  defined on $\R^{3}$ such that
\begin{equation}\label{249249}
\exists k_{0}\in \N, \quad
\val{\alpha(\xi)}\le C_{0}\val \xi^{1+k_{0}}, \text{ for almost all $\xi\in \R^{3}$.} 
\end{equation}
$(1)$ Then we have 
$
\alpha(D)\mathcal A\subset L^{2}(\R^{3}),
$
where $\mathcal A$ is defined in \eqref{247247}.
\par\no
$(2)$
Moreover we have 
$
\{w\in \mathcal A,  0\notin \spec{w}\}\subset L^{2}(\R^{3}).
$
\end{lemma}
\begin{proof}
 Let $v\in \mathcal A$; then $\hat v\in\mathcal W$ and 
we have \begin{align*}
 \int_{\R^{3}} \val{\alpha(\xi)}^{2}\val{\hat v(\xi)}^{2} d\xi 
\le C_{0}^{2}\int_{\R^{3}}
 \val \xi^{2+2k_{0}}
 \val{\hat v(\xi)}^{2} d\xi<+\io,
\end{align*}
thanks to \eqref{gqf582}, since $v\in \mathcal A$, proving $(1)$ in the lemma.
Now, if $w\in \mathcal A$ is such that $0\notin \spec{w}$,
we find that the $L^{1}$ function $\hat w$ is vanishing on $r_{0}\mathbb B^{3}$ for some $r_{0}>0$.
Let $\beta_{0}$ be a smooth compactly supported function, equal to 1 on $\frac13 r_{0}\mathbb B^{3}$, vanishing outside $\frac23 r_{0}\mathbb B^{3}$ and 
$
\beta_{1}=1-\beta_{0}.
$
Then $\beta_{1}$ is smooth, vanishing on $\frac13 r_{0}\mathbb B^{3}$, equal to 1 outside 
$\frac23 r_{0}\mathbb B^{3}$.
The operator $\beta_{1}(D)=1-\beta_{0}(D)$ is a Fourier multiplier bounded on every $L^{p}$ for $p\in [1,+\io]$ since $\beta_{0}(D)$ is the convolution with the Schwartz function 
$\check{\hat \beta}_{0}$; moreover $\beta_{1}$ satisfies \eqref{249249} with $k_{0}=0$.
From the now proven $(1)$ in the lemma, we get that
$
\beta_{1}(D) w \in L^{2}.
$
We have also $\beta_{1}(D) w =w$ since
$$
\beta_{1}(\xi) \hat w(\xi)=\hat w(\xi) -\beta_{0}(\xi) \hat w(\xi)=\hat w(\xi),
$$
concluding the proof of the lemma.
\end{proof}
\section{\color{magenta}More on the Bernoulli head pressure}\label{sec.3333}
\subsection{First facts}
We have a vector field $v$ on $\R^{3}$ such that $v$ belongs to the algebra $\mathcal A$ defined in \eqref{247247} and such that Theorem \ref{thm.214t++}
is satisfied (thus as well as Corollary \ref{cor.kjha45}).
We recall as well  the decay properties given by Lemma \ref{lem.reg1}. The vector field $v$ satisfies the non-linear stationary system for incompressible fluids which can be written either as 
\eqref{newway} with 
 \begin{equation}\label{newway+}
\nu \rot^{2} v+\mathbb P(\rot v\times v)=0, \quad \dive v=0,
\end{equation}
or as
\begin{align}\label{mod+++0}
\begin{cases}
&\nu \rot^{2}v+(\rot v\times v)+\nabla Q=0, \quad \dive v=0,
\\
&Q=R_{0}(v\otimes v)+\frac12\val v^{2}.
\end{cases}
\end{align}
with the singular integral $R_{0}$ is defined in \eqref{ddd001}. We note that $Q$ belongs to the algebra 
$\mathcal A$  and that, from \eqref{newway+}, we find, 
\begin{equation}\label{}
\nabla Q=\val{D}^{-2}\grad \dive(\rot v\times v)=(\mathbb P-I)(\rot v\times v)=-\widetilde{\mathbb P}(\rot v\times v),
\end{equation}
where the projection $\widetilde{\mathbb P}$ is given by \eqref{ortler}.
We get as well that 
\begin{equation}\label{mmm002}
-\Delta Q=-\dive \nabla Q=\dive(\rot v\times v)\hskip-8pt\underbracket[0pt]{=}_{\text{cf. \eqref{4edsw6}}}\hskip-8pt\proscal3{\rot^{2}v}{v}-\val{\rot v}_{\R^{3}}^{2},
\end{equation}
and using the equation \eqref{mod+++0}, we get 
$$
-\Delta Q=-\nu^{-1}\underbrace{\proscal3{\rot v\times v}{v}}_{=\det(\rot v, v, v)=0}-\nu^{-1}\proscal3{\nabla Q}{v}-\val{\rot v}_{\R^{3}}^{2},
$$
yielding
\begin{equation}\label{315315}
\Delta Q=\nu^{-1}\proscal3{\nabla Q}{v}+\val{\rot v}_{\R^{3}}^{2}.
\end{equation}
Multiplying pointwisely  the first line of \eqref{mod+++0} by $v$, we get 
\begin{equation}\label{rot002}
\nu\proscal3{\rot^{2}v}{v}+\proscal3{v}{\nabla Q}=0,
\end{equation}
a formula which will be useful later on.
\subsection{The Bernoulli head pressure $Q$ is a non-positive function}
\begin{theorem}[Chae's Theorem]\label{thm.chae+}
 Let $v$ be a vector field satisfying the assumptions of Conjecture \ref{coj001}.
 Then either  $v$ is identically $0$,
 or the Bernoulli head pressure $Q$ defined by \eqref{modpre} satisfies
  $Q(\R^{3})=[-M_{0}, 0)$, where $M_{0}>0$.
\end{theorem}
\begin{nb}
 This result was proven by D.~Chae in \cite{MR3959933},
 with a new multiplier method:
 instead of mutliplying the equation by any $\chi v$ (say with $\chi$ radial smooth compactly function), he chose to use a very particular function $\chi$, appearing as a carefully chosen  function of $Q$.
 We reproduce below some slight modifications of his arguments.
\end{nb}
\begin{proof}
 Since $Q$ is belonging to the algebra $\mathcal A$ (see Corollary \ref{cor.lkj987}), it is a real-valued continuous function with limit 0 at infinity. 
  As a consequence $Q(\R^{3})$ is included in a compact set $[-M_{0}, M_{0}]$
 for some $M_{0}>0$. 
 \par
 \emph{We want first  to prove that $Q$ takes its values in $(-\io, 0]$.}
 Let us choose $ 0<\varepsilon_{0}<\varepsilon_{1}<M_{0}$, and 
 an increasing function $\theta_{0}\in C^{\io}(\R, [0,1])$ 
 such that $\supp \theta_{0}=[0,+\io)$, $\theta\bigl((0,1)\bigr)=(0,1)$, $\theta_{0}([1,+\io))=\{1\}$.
 Let us define $f:\R\longrightarrow \R$, 
 \begin{equation}\label{}
f(s)=\theta_{0}\Bigl(\frac{s-\varepsilon_{0}}{\varepsilon_{1}-\varepsilon_{0}}\Bigr),\quad
\text{\small a $C^{\io}$  increasing function with support $[\varepsilon_{0},+\io)$},
\end{equation}
 and 
 \begin{equation}\label{}
F(s)=\int_{\varepsilon_{0}}^{s} f(t) dt,\quad
\text{\small a $C^{\io}$ function with support $[\varepsilon_{0},+\io)$ and  $F'=f$.} 
\end{equation}
We note that $\R^{3}\ni x\mapsto f(Q(x))$
 is $C^{\io}$ with compact support, since 
 $$
 \supp f(Q)\subset\{x\in \R^{3}, Q(x)\ge \varepsilon_{0}\},\text{\tiny which is compact since
 $\varepsilon_{0}>0$ and  $Q$ tends to 0 at infinity. }
 $$ 
 We may thus calculate\footnote{Here we use the notation $\poscal{a}{\phi}=\int_{\R^{3}}
 \proscal3{a(x)}{\phi(x)}
 dx$ for $a\in C^{\io}(\R^{3}, \R^{3})$, $\phi\in C^{\io}_{c}(\R^{3}, \R^{3})$.}
\begin{align*}
0&=\poscal{\nu \rot^{2}v+(\rot v\times v)+\nabla Q}{f(Q)v}
\\
&=\nu\poscal{\rot v}{f(Q) \rot v}+\nu\poscal{\rot v}{f'(Q)  \nabla Q\times v}
\\
&\hskip25pt+\int \underbracket[0.2pt]{\det(\rot v, v, v)}_{=0}f(Q) dx+\poscal{\nabla\bigl(F(Q)\bigr)}{v}.
\end{align*}
If $\varepsilon_{0}$ is a regular value of $Q$, the Lebesgue measure of the set $\{x\in \R^{3}, Q(x)=\varepsilon_{0}\}$ is zero, so that 
\begin{multline*}
\poscal{\nabla\bigl(F(Q)\bigr)}{v}=\int_{\{Q(x)\ge \varepsilon_{0}\}}\dive\bigl(
F(Q) v \bigr)dx=\int_{\{Q(x)> \varepsilon_{0}\}}\dive\bigl(
F(Q) v \bigr)dx
\\
=\int_{\{Q(x)=\varepsilon_{0}\}} \underbrace{F(Q(x))}_{=F(\varepsilon_{0})=0}\proscal3{v}{\vec{n}}d\sigma
=0,
\end{multline*} 
where $\vec{n}$ is the unit exterior normal vector to the boundary of the open set 
$\{ Q(x)>\varepsilon_{0}\}$
and $d\sigma$ is the Euclidean measure on that boundary,
yielding
\begin{equation}\label{wdq671}
0=\int f(Q(x))\val{\rot v(x)}^{2}dx+\poscal{\rot v}{\nabla\bigl\{f(Q)\bigr\} \times v}.
\end{equation}
We notice that
\begin{multline}\label{54nhga}
\poscal{\rot v}{\nabla\bigl\{f(Q)\bigr\} \times v}=
-\int\det\bigl(\rot v, v, \nabla\bigl\{f(Q)\bigr\} \bigr)dx 
\\
=\poscal{-\bigl(\rot v\times v\bigr)}{\nabla\bigl\{f(Q)\bigr\} }
=\poscal{\nu \rot^{2}v+\nabla Q}{\nabla\bigl\{f(Q)\bigr\} }
\\
=\nu\poscal{\rot v}{\underbracket[0.2pt]{\rot\bigl[ \nabla\bigl\{f(Q)\bigr\} \bigr]}_{\substack{
=0,  \text{ since}\\ \text{$\curl \grad=0$}}
}}+\int f'(Q)\val{\nabla Q}^{2} dx.
\end{multline}
Eventually, if $\varepsilon_{0}$ is a regular value of $Q$, from \eqref{wdq671}, \eqref{54nhga},
we get
\begin{equation}\label{}
\int f(Q(x))\val{\rot v(x)}^{2}dx+\int f'(Q(x))\val{(\nabla Q)(x)}^{2} dx=0.
\end{equation}
 Since $f'(Q(x))\ge 0$ (the functions $f$ and $\theta_{0}$ are increasing)
 we get,
 \begin{equation}\label{}
\int f(Q(x))\val{\rot v(x)}^{2}dx=0,
\end{equation}
yielding, since $f(s)>0$ for $s> \varepsilon_{0}$,
\begin{equation}\label{}\forall \varepsilon_{0}>0,\text{\small regular value of $Q$}, \quad
\{x, Q(x)>\varepsilon_{0}\}\subset\{x, (\rot v)(x)=0\}.
\end{equation}
Thanks to Sard's theorem, the one-dimensional Lebesgue measure
 of the set
\begin{equation}\label{sard0+}
 Q\Bigl[\bigl(\nabla Q\bigr)^{-1}(\{0_{\R^{3}}\})\Bigr]
\end{equation}
 is zero.
As a consequence,
we can find arbitrarily small positive $\varepsilon_{0}$ which is a regular value of $Q$, so that we get eventually 
\begin{equation}\label{327327}
\Omega:=\{x, Q(x)>0\}\subset\{x, (\rot v)(x)=0\}.
\end{equation} 
We note that, on the open set $\Omega$,  we have $\rot v=0$ and thus $\rot^{2}v=0$ ($\Omega$ is open)
 and this implies 
from the equation \eqref{mod+++0} that $\nabla Q=0$ on $\Omega$.
Let us consider the open set $\Omega=\{x, Q(x)>0\}$ and its connected components $(\Omega_{j})_{j\in \N}$
so that
$$
\Omega=\cup_{j\in \N}\Omega_{j}, \quad \text{$\Omega_{j}$ open and connected, two by two disjoint.}
$$
We have $\nabla Q=0$ on $\Omega_{j}$ and thus there is a constant $q_{j}>0$ such that $Q=q_{j}$ on $\Omega_{j}$.
Assuming $\p \Omega_{j}\not=\emptyset$,
let $x\in \p \Omega_{j}=\overline{\Omega}_{j}\backslash \Omega_{j}$: then we have by continuity of $Q$ that $Q(x)=q_{j}>0$; since  $x$ is not in $\Omega_{j}$ it  belongs to some $\Omega_{k}$ with $\Omega_{k}\cap \Omega_{j}=\emptyset$. Since $x=\lim_{l} x_{l}$, $x_{l}\in \Omega_{j}$, for $l$ large enough $x_{l}$ must belong to the open set $\Omega_{k}$, but this is not possible since $\Omega_{k}\cap \Omega_{j}=\emptyset$.
This proves that $\p \Omega_{j}=\emptyset$.
As a result, the open sets $\Omega_{j}$ are also closed and thus for all $j\in \N$,
$$
\Omega_{j}\in \{\emptyset, \R^{3}\}.
$$
If all the $\Omega_{j}$ are empty, so it is the case for $\Omega$,
entailing that 
\begin{equation}\label{mws555}
Q(\R^{3})\subset(-\io,0].
\end{equation}
 If one $\Omega_{j}$ is not empty, we get that 
$\Omega=\R^{3}$, proving from \eqref{327327} that $\rot v\equiv 0$,
and 
Lemma \ref{lem.reg001} gives that $v\equiv 0$.
Moreover if $Q$ vanishes identically, we get from \eqref{315315} that $\val{\rot v}^{2}=\Delta Q=0$,
so that $\rot v\equiv 0$,
and 
we get as above that $v\equiv 0$.
Thus we may assume that the function $Q$ is smooth non-positive, with a negative minimum
$-M_{0}$, with $M_{0}>0$.
Now since $Q$ has limit 0 at infinity, the minimum $-M_{0}$ is attained at some point $x_{0}\in \R^{3}$,
so that 
$$[-M_{0},0)
\underbracket[0pt]{\subset}_{\substack
{
\text{connexity}\\ \text{of $\R^{3} $}
}}
Q(\R^{3})\subset[-M_{0},0].
$$
We claim now that $Q(\R^{3})=[-M_{0},0).$
Indeed, 
Formula \eqref{315315} implies that 
 \begin{equation}\label{}
 -\Delta Q+\nu^{-1} v\cdot \nabla Q\le 0, \quad \text{where $v,Q$ are smooth functions.}
\end{equation}
We already know that $Q$ is non-positive on $\R^{3}$
and we consider the closed set $Q^{-1}(\{0\})$. If it is empty we obtain the result of our Claim and  the last sought result in 
Theorem \ref{thm.chae+};
assuming that it is not empty, we take $x_{0}$ such that $Q(x_{0})=0$. For any $R>0$, on the open connected bounded set $B(x_{0}, R)$,
the maximum principle implies that the maximum of $Q$ is attained at the boundary,
which implies that $Q$ is identically 0,
a case in which we have seen that $v$ must be identically vanishing.
This completes the proof of Theorem  \ref{thm.chae+}.
\end{proof}
The following lemma will be useful later on.
\begin{lemma}\label{lem.2154tt}
Let $Q$ be the Bernoulli head pressure defined in  \eqref{mod+++0}.
Then we have the pointwise equalities, 
 \begin{align}\label{328328}
\Delta Q=\val{\rot v}_{\R^{3}}^{2}-\proscal3{\rot^{2}v}{v}=
\nu^{-1}\proscal3{\nabla Q}{v}+\val{\rot v}_{\R^{3}}^{2}.
\end{align}
Moreover, if $Q(x_{1})$ is a local maximum of $Q$ in a neighborhood of $x_{1}\in \R^{3}$,
we have 
$$
\nabla Q(x_{1})=(\rot v)(x_{1})=(\rot^{2} v)(x_{1})=0.
$$
\end{lemma}
\begin{proof}
Formulas \eqref{328328} follow from \eqref{mmm002}, \eqref{315315}.
 We note also that,
 since $Q(x_{1})$ is a local maximum of $Q$, we have   $$
\nabla Q(x_{1})=0,\ 
Q''(x_{1})\le 0, \ \trace{Q''(x_{1})}\le 0,
 $$
 where $Q''$ stands for the Hessian matrix of $Q$.
 We note that, at $x_{1}$, we have 
 $$
0\le  -\trace{Q''(x)}=-\Delta Q=\dive (\rot v\times v)=\rot^{2}v\cdot v-\val{\rot v}^{2},
 $$
 so that  we have  at $x_{1}$,
 \begin{align}
 &\val{\rot v}^{2}\le \rot^{2}v\cdot v,\label{qq002}\\
 & \nu \rot^{2}v+\bigl(\rot v\times v\bigr)=0,\label{qq003}
\end{align}
entailing  at $x_{1}$
 $$
 \nu \rot^{2}v\cdot v+\underbrace{\bigl(\rot v\times v\bigr)\cdot v}_{=0}=0,
 $$
 which implies 
 $
 (\rot^{2}v\cdot v)(x_{1})=0,
 $
which gives from \eqref{qq002}, $(\rot v)(x_{1})=0$,
yielding from \eqref{qq003},  $(\rot^{2} v)(x_{1})=0$,
 proving the lemma.
\end{proof}
\subsection{More properties of the Bernoulli head pressure}
\begin{lemma}\label{lem.gg01}
 Let $v$ be a vector field on $\R^{3}$ satisfying 
\eqref{SNSI}, \eqref{hyp001}, \eqref{hyp002} and the Bernoulli head pressure $Q$ be defined by \eqref{modpr0}. 
Then for all $\tau>0$, 
the set
\begin{equation}\label{511}
Q_{\tau}=Q^{-1}\bigl((-\io,-\tau]\bigr)=\{x\in \R^{3}, Q(x)\le -\tau\},
\end{equation}
is compact. Moreover for $0<\tau_{2}\le \tau_{1}$, we have $Q_{\tau_{1}}\subset Q_{\tau_{2}}$, as well as 
\begin{align}\label{uyd268}
\cup_{\tau>0}Q_{\tau}=\R^{3}.
\end{align}
In particular, if $(\tau_{j})_{j\ge 1}$ is a decreasing sequence of positive numbers with limit 0,
we get that the sequence of sets $(Q_{\tau_{j}})_{j\ge 1}$ is increasing with union $\R^{3}$.
\end{lemma}
\begin{proof} The continuity of $Q$ ensures that $Q^{-1}\bigl((-\io,-\tau]\bigr)$ is closed. Moreover, it is also a bounded set since $\lim_{R\rightarrow+\io}\norm{Q}_{L^{\io}(\{\val x\ge R\})}=0$: this implies that for $\tau>0$, there exists $R_{\tau}\ge 0$ such that if $\val x> R_{\tau}$, we get 
$
Q(x)\in (-\tau, \tau)
$
and thus
$$
Q^{-1}\bigl((-\tau, \tau)^{c}\bigr)\subset \bar B(0, R_{\tau}),
$$
proving \eqref{511}. 
The equality \eqref{uyd268} follows from Theorem \ref{thm.chae+}
and the remaining part of the lemma is trivial.
The proof of the lemma is complete.
\end{proof}
\begin{claim}\label{cla.gg01}
Let $g$ be a continuous function from the real line into itself with $0\notin\supp g$. Then the continuous function
$\R^{3}\ni x\mapsto g\bigl(Q(x)\bigr)\in \R$ is compactly supported.
 \end{claim}
\begin{proof}[Proof of the claim]
We know that 
 $0$ belongs to the open set $(\supp g)^{c}$ and thus, for some $\tau_{0}>0$,
we have $(-\tau_{0}, \tau_{0})\subset(\supp g)^{c}$.
As a consequence, using the notations of Lemma \ref{lem.gg01}, we get 
$$
\supp (g\circ Q)
\subset{\{x\in\R^{3}, \val{Q(x)}\ge \tau_{0}\}}
=\{x\in\R^{3}, Q(x)\le -\tau_{0}\}=Q_{\tau_{0}},
$$
which is a compact set according to Lemma \ref{lem.gg01},
proving the claim.
\end{proof}
\subsection{Assumptions on the Laplacean of the Bernoulli Head Pressure}
\begin{claim}\label{cla.41po}
 Let $\chi$ be a smooth compactly supported function defined on $\R^{n}$, equal to 1 in a neighborhood of the origin and let us define for $\lambda>0$,
 $\chi_{\lambda}(x)=\chi(x/\lambda)$.
 Then we have
 $$
 \lim_{\lambda\rightarrow+\io}\int \chi_{\lambda}(x) (\Delta Q)(x) dx=0.
 $$
\end{claim}
\begin{proof}
 We have 
 $
 \Delta Q=\val{\rot v}^{2}-\rot^{2}v\cdot v
 $
 and thus
\begin{align*}
\int \chi_{\lambda} (\Delta Q) dx&=\int \chi_{\lambda} \val{\rot v}^{2} dx-\poscal{\rot^{2}v}{\chi_{\lambda}v}
\\
&=\int \chi_{\lambda} \val{\rot v}^{2} dx-\poscal{\rot v}{\chi_{\lambda}\rot v}
-\poscal{\rot v}{\nabla\chi_{\lambda}\times v}
\\&=\poscal{\underbrace{\rot v}_{L^{2}}}{\underbrace{v}_{L^{6}}\times \underbrace{\nabla\chi_{\lambda}}_{L^{3}}}=\alpha(\lambda), 
\end{align*}
with\footnote{For $p\ge 3$ we have for $\lambda\ge 1$, 
$$
\int_{\R^{3}} \val{(\nabla \chi_{\lambda})(x)}^{p} dx=\int_{\R^{3}} \val{(\nabla \chi)(x/\lambda)}^{p} \lambda^{-p}dx
=
\int_{\R^{3}} \val{(\nabla \chi)(y)}^{p} \lambda^{3-p}dy\le \norm{\nabla \chi}_{L^{p}}^{p}.
$$} 
$
\val{\alpha(\lambda)}\le \norm{\rot v}_{L^{6}(\supp \nabla \chi_{\lambda})}\norm{v}_{L^{2}(\supp \nabla \chi_{\lambda})}\norm{\nabla \chi}_{L^{3}},
$
so that 
$\lim_{\lambda\rightarrow+\io}\alpha(\lambda)=0$,
proving the claim.
\end{proof}
\begin{claim}\label{cla.kj55}
 Assuming that $-\varepsilon_{0}$ is a regular value of $Q$, we get that 
 \begin{equation}\label{cxz596}
\int_{\{Q(x)<-\varepsilon_{0}\}} \val{\rot v}^{2}dx=\int_{\{Q(x)<-\varepsilon_{0}\}}\Delta Qdx,
\end{equation}
\end{claim}
\begin{proof}
 Indeed, 
 we have 
\begin{multline}\label{tre874}
\int_{\{Q(x)<-\varepsilon_{0}\}} \proscal3{v}{\nabla Q} dx=
\int_{\{Q(x)<-\varepsilon_{0}\}}\dive\bigl((Q+\varepsilon_{0})v\bigr)dx
\\
=
\int_{\{Q(x)=-\varepsilon_{0}\}}(Q+\varepsilon_{0})\proscal3{v}{\vec{\texttt{\hskip 2pt n}}} d\sigma=0,
\end{multline}
where $\vec{n}$ is the exterior unit normal to the boundary of the regular open set $$\{x, Q(x)< -\varepsilon_{0}\},$$ and $d\sigma$ is the Euclidean measure on that boundary.
Using \eqref{315315}, this proves
 the claim.
\end{proof}
\begin{proposition}\label{pro.417aez}
 Let $Q$ be the Bernoulli head pressure defined in  \eqref{mod+++0}. Then if $\Delta Q$ belongs to $L^{1}(\R^{3})$, we have $v=0$.
\end{proposition}
\begin{proof}
We have  from Lemma \ref{lem.2154tt} and \eqref{rot002},
\begin{equation}\label{3215+}
\val{\rot v}^{2}-\Delta Q=\rot^{2}v\cdot v=-\nu^{-1}v\cdot\nabla Q,
\end{equation}
and thus we find that $\rot^{2}v\cdot v$ belongs to $L^{1}(\R^{3})$.
We find as well that for $\varepsilon_{0}>0$, the set $\{x, Q(x)\le  -\varepsilon_{0}\}$ is compact since $Q$ goes to 0 at infinity. As a consequence of Claim \ref{cla.kj55}, we obtain that 
\begin{equation}\label{cxz59+}
\int_{\{Q(x)<-\varepsilon_{0}\}} \val{\rot v}^{2}dx=\int_{\{Q(x)<-\varepsilon_{0}\}}\Delta Q\ dx,
\end{equation}
and using again that
from Sard's Theorem, we can find arbitrarily small positive $\varepsilon_{0}$ which is a regular value of $-Q$, so that the Lebesgue Dominated Convergence Theorem gives that
$$
\int_{\{Q(x)<0\}} \val{\rot v}^{2}dx=\int_{\{Q(x)<0\}}\Delta Q\, dx,
$$
and from Theorem \ref{thm.chae+}
we get
\begin{equation}\label{cxw628}
\int_{\R^{3}} \val{\rot v}^{2}dx=\int_{\R^{3}}\Delta Q\ dx.
\end{equation}
 Now Claim \ref{cla.41po} and $\Delta Q\in L^{1}$ imply that the right-hand-side of 
 \eqref{cxw628} is zero,
 entailing that $\rot v=0$ as well as $v$ (we have $\dive v=0, \rot v\in L^{2}$ so we may apply Lemma \ref{dssee5} to get $\nabla v=0$ so that $v$ is constant and in $\mathcal L_{(0)}$, thus $0$).
\end{proof}
We always have the identity
 \begin{equation}\label{cxz5++}
\int_{\{Q(x)<-\varepsilon_{0}\}} \val{\rot v}^{2}dx=\int_{\{Q(x)<-\varepsilon_{0}\}}\Delta Qdx,
\end{equation}
provided $-\varepsilon_{0}$ is a regular value of $Q$. This means also that 
 \begin{equation}\label{+xz5++}
\int_{\{Q(x)<-\varepsilon_{0}\}} \val{\rot v}^{2}dx+\int_{\{Q(x)<-\varepsilon_{0}\}}\bigl(\Delta Q\bigr)_{-}dx=\int_{\{Q(x)<-\varepsilon_{0}\}}\bigl(\Delta Q\bigr)_{+}dx
\end{equation}
Applying the Beppo\hskip1pt--Levi Theorem\footnote
{We note that the functions $(\Delta Q)_{\pm}$ are non-negative and that for a sequence $(\varepsilon_{k})_{k\ge 0}$ of positive numbers decreasing to 0, such that $-\varepsilon_{k}$ is a regular value of $Q$ (such sequences are existing, thanks to Sard's Theorem), we have 
$$
\{x, Q(x)<-\varepsilon_{k}\}\subset\{x, Q(x)<-\varepsilon_{k+1}\}\ \text{and} \ \cup_{k\ge 0}
\{x, Q(x)<-\varepsilon_{k}\}\underbracket[0pt]{=}_{\eqref{uyd268}}\{x, Q(x)<0\}=\R^{3}.
$$
}
, we get 
 \begin{equation}\label{cxz+++}
\int_{\{Q(x)<0\}} \val{\rot v}^{2}dx+\int_{\{Q(x)<0\}}\bigl(\Delta Q\bigr)_{-}dx=\int_{\{Q(x)<0\}}\bigl(\Delta Q\bigr)_{+}dx,
\end{equation}
which gives, thanks to Theorem \ref{thm.chae+},
\begin{equation}\label{cx++++}
\int_{\R^{3}} \val{\rot v}^{2}dx+\int_{\R^{3}}\bigl(\Delta Q\bigr)_{-}dx=\int_{\R^{3}}\bigl(\Delta Q\bigr)_{+}dx.
\end{equation}
The following result proves Theorem \ref{thm.54ek}.
\begin{theorem}\label{thm.5468} Let $Q$ be the Bernoulli head pressure defined in  \eqref{mod+++0}. If 
$\bigl(\Delta Q\bigr)_{+}$ \emph{or} $\bigl(\Delta Q\bigr)_{-}$
belongs to $L^{1}(\R^{3})$, then $v=0$.
\end{theorem}
\begin{proof}
Indeed if $\bigl(\Delta Q\bigr)_{+}$ belongs to $L^{1}(\R)$, then Equality \eqref{cx++++} implies that 
 $\bigl(\Delta Q\bigr)_{-}$ also belongs to $L^{1}(\R)$. Similarly, if 
 $\bigl(\Delta Q\bigr)_{-}$ belongs to $L^{1}(\R)$, then Equality \eqref{cx++++} implies that 
 $\bigl(\Delta Q\bigr)_{+}$ also belongs to $L^{1}(\R)$. In both cases, we get $\Delta Q\in L^{1}$ and we may apply Proposition \ref{pro.417aez} to conclude.
 Note that when the hypothesis of the theorem does not hold,
 i.e. when 
 $$
 \int_{\R^{3}}\bigl(\Delta Q\bigr)_{-}dx=\int_{\R^{3}}\bigl(\Delta Q\bigr)_{+}=+\io,
 $$
 Equality \eqref{cx++++} still holds true, but does not imply $\int_{\R^{3}} \val{\rot v}^{2}dx=0$.
\end{proof}
\begin{remark}\rm
 In particular, assuming $C^{2}v\in L^{6/5}$, as  in D.~Chae's \cite{MR3162482},
 implies $\Delta Q\in L^{1}$, thanks to \eqref{3215+}.
\end{remark}
\begin{remark}\rm
 From \eqref{328328}, we have 
$
\nu \Delta Q=
\proscal3{\nabla Q}{v}+\nu \val{\rot v}_{\R^{3}}^{2}.
$
In particular if 
\begin{equation}\label{}
\proscal3{\nabla Q}{v}=\dive(Qv)\ge 0,
\end{equation}
we find that $(\Delta Q)_{-}= 0$, so that Theorem \ref{thm.5468} entails that $v=0$. 
Similarly, if we have 
\begin{equation}\label{}
\dive(Qv)+\nu \val{\rot v}_{\R^{3}}^{2}\ge 0,
\end{equation}
we get that $v=0$. 
\end{remark}
\begin{corollary}
 Let $Q$ be the Bernoulli head pressure defined in  \eqref{mod+++0}. If $Q$ is subharmonic,
 i.e. $\Delta Q\ge 0$,  then $v=0$.
\end{corollary}
\begin{proof}
 The assumption implies $(\Delta Q)_{-}=0$, so that the hypothesis of Theorem \ref{thm.5468} is fulfilled.
\end{proof}
\begin{corollary}
 Let $Q$ be the Bernoulli head pressure defined in  \eqref{mod+++0}. If $Q$ is superharmonic,
 i.e. $\Delta Q\le 0$,  then $v=0$.
\end{corollary}
\begin{proof}
 The assumption implies $(\Delta Q)_{+}=0$, so that the hypothesis of Theorem \ref{thm.5468} is fulfilled.
\end{proof}
\section{\color{magenta}Proof of the other results}\label{sec.new4}
\subsection{Improvements of Galdi's result for $v\in L^{9/2}$}
We know that  $v$ and $Q$ are  $\moo$ functions in the algebra $\mathcal A$ defined in \eqref{247247}, which go to zero at infinity as well as all their derivatives. We keep in mind the decay results of Lemma \ref{lem.reg1} and Formula \eqref{modpre},
ensuring that $$Q\in  L^{3}\cap\mathcal W\subset L^{3}\cap L^{\io}.$$
We want to state and prove an improved  version of a theorem due to G.P.~Galdi,
formulated in (1) page \pageref{111kno},
involving only the low frequencies of the vector field $v$ satisfying \eqref{mod+++0}, \eqref{newway+}.
\par Let us start with a somewhat generic computation and a review of Galdi's proof: in the sequel $\chi$ stands for a smooth compactly supported function. We calculate\footnote{Here for $a,b$ smooth vector fields with $b$ compactly supported,  we define
$$
\poscal{a}{b}=\int_{\R^{3}}\proscal3{a(x)}{b(x)}dx.
$$}
\begin{align}
0&=\poscal{\nu \rot^{2}v+(\rot v\times v)+\nabla Q}{\chi v}\notag\\
&=\nu\int \chi\val{\rot v}^{2}dx+\nu\poscal{\rot v}{\nabla \chi\times v}-\poscal{Qv}{\nabla \chi}.
\label{411411}
\end{align}
We note that, replacing $\chi$ by $\chi_{\lambda}$ defined by
\begin{equation}\label{}
\chi_{\lambda}(x)=\chi(x/\lambda), \quad \lambda>0,\quad \text{$\chi\in \mooc(\R^{3})$, equal to 1 near the origin, }
\end{equation}
we get that 
$$
(\nabla \chi_{\lambda})(x)=(\nabla \chi)(x\lambda^{-1}) \lambda^{-1}, \quad
\nabla \chi_{\lambda}\text{ bounded in $L^{p}$ for $p\ge 3, \lambda\ge 1$.}
$$
We obtain that 
\begin{multline*}
\val{\poscal{\rot v}{\nabla \chi_{\lambda}\times v}}\le \norm{\rot v}_{L^{2}}\norm{\nabla \chi_{\lambda}}_{L^{3}}\norm{ v}_{L^{6}(\supp \nabla \chi_{\lambda})}
\\=
\norm{\rot v}_{L^{2}}\norm{\nabla \chi}_{L^{3}}\underbrace{\norm{ v}_{L^{6}(\supp \nabla \chi_{\lambda})}}_{\substack{\text{with limit $0$}\\\text{when $\lambda\longrightarrow +\io$}\\
\text{since $v\in L^{6}$}}},
\end{multline*}
and in particular
\begin{equation}\label{413413}
\lim_{\lambda\rightarrow+\io}\poscal{\rot v}{\nabla \chi_{\lambda}\times v}=0, \quad \text{and also }
\lim_{\lambda\rightarrow+\io}\int \chi_{\lambda}\val{\rot v}^{2}dx=\norm{\rot v}_{L^{2}}^{2}.
\end{equation}
We check now the last term in \eqref{411411}.
We have with a singular integral $S_{0}$,
\begin{multline*}
\val{\poscal{Qv}{\nabla \chi_{\lambda}}}\le \norm{\nabla \chi_{\lambda}}_{L^{3}}\norm{Qv}_{L^{3/2}(\supp \nabla \chi_{\lambda})}
\\\le \norm{\nabla \chi}_{L^{3}}\norm{S_{0}(v\otimes v)}_{L^{9/4}}
{\bf\norm{v}_{L^{9/2}(\supp \nabla \chi_{\lambda})}}
\\\le 
\trinorm{S_{0}}_{9/4}\norm{\nabla \chi}_{L^{3}}\norm{v}_{L^{9/2}}^{2}
{\bf\norm{v}_{L^{9/2}(\supp \nabla \chi_{\lambda})}},
\end{multline*}
yielding, \emph{provided $v$ belongs to $L^{9/2}$,}
\begin{equation}\label{414414}
\lim_{\lambda\rightarrow+\io}\poscal{Qv}{\nabla \chi_{\lambda}}=0.
\end{equation}
The {properties} \eqref{411411}, \eqref{413413}, \eqref{414414} imply $\rot v=0$ and thus from \eqref{dssee5}
in Lemma \ref{lem.reg001} we find that $\nabla v=0$;
as a consequence $v$ is constant and since $v\in \mathcal L_{(0)}$ we get that $v=0$.
\begin{nb}
 We  have typed some terms above in boldface to emphasize the fact that a localization argument
 was at play, and we want to keep that in mind for the proof to come.
\end{nb}
We claim now that a simple modification of the above argument allows to microlocalize the assumption and to limit the requirement on $v$ to the low frequencies.
We note first that \eqref{413413} holds true under the sole assumption that $v$ belongs to $L^{6}$,
which we already know.
Our improvement will be located in the treatment of the term $\poscal{Qv}{\nabla \chi_{\rho}}$. We have
\begin{align}\text{with }
v=v_{[0]}+v_{[1]},& \quad \supp\widehat{v_{[0]}}\subset 2\mathbb B^{3}, \quad
\mathring{\mathbb B}^{3}
\cap\supp\widehat{v_{[1]}}=\emptyset,
\label{fdr443}\\
\poscal{Qv}{\nabla \chi_{\rho}}&=\poscal{S_{0}(v\otimes v)}{v\nabla \chi_{\rho}}
\notag\\
& =\poscal{S_{0}(v\otimes v)}{v_{[0]}\nabla \chi_{\rho}}+\poscal{S_{0}(v\otimes v)}{v_{[1]}\nabla \chi_{\rho}}.
\label{fdr445}
\end{align}
We have also $v_{[0]}=\alpha_{0}(D) v, \supp \alpha_{0}\subset 2\mathbb B^{3}$ and $v_{[0]}\in L^{6}$,
\begin{equation}\label{dqa578}
v_{[1]}=\alpha_{1}(D) v, \mathring{\mathbb B}^{3}\cap \supp \alpha_{1}=\emptyset,
\text{ so that, from \eqref{aer982}, }v_{[1]}\in L^{2},
\end{equation}
as well as 
\begin{align}\label{7ml5iu}
S_{0}(v\otimes v)&=S_{0}\bigl[(v_{[0]}+v_{[1]})(v_{[0]}+v_{[1]})\bigr]
\\&=S_{0}\bigl[v_{[0]}\otimes v_{[0]}\bigr]+S_{0}\bigl[v_{[0]}\otimes v_{[1]}\bigr]
+S_{0}\bigl[v_{[1]}\otimes v_{[0]}\bigr]
+S_{0}\bigl[v_{[1]}\otimes v_{[1]}\bigr].\notag
\end{align}
We note that
$$
\norm{\overbracket[0.2pt]{S_{0}(\underbrace{v}_{L^{\io}}\otimes \underbrace{v}_{L^{6}})}^{L^{6}}
\underbrace{v_{[1]}}_{L^{2}}\underbrace {\nabla \chi_{\rho}}_{L^{3}}}_{L^{1}}
\le \norm{\nabla \chi}_{L^{3}}\norm{v_{[1]}}_{L^{2}(\supp \nabla \chi_{\rho})}\norm{v}_{L^{6}}\norm{v}_{L^{\io}},
$$
which proves that 
\begin{equation}\label{}
\lim_{\rho\rightarrow+\io}\poscal{S_{0}(v\otimes v)}{v_{[1]}\nabla \chi_{\rho}}=0.
\end{equation}
We are left with $\poscal{S_{0}(v\otimes v)}{v_{[0]}\nabla \chi_{\rho}}$ and we have
\begin{multline*}
v_{[0]}S_{0}(v\otimes v)=
v_{[0]}S_{0}\bigl[v_{[0]}\otimes v_{[0]}\bigr]
\\
+\underbracket[0.2pt]
{\underbrace{v_{[0]}}_{L^{6}}     \overbracket[0.2pt]{S_{0}\bigl[   \underbrace{v_{[0]}}_{\mathcal W\subset L^{\io}}  \otimes \underbrace{v_{[1]}}_{L^{2}}  \bigr]}^{L^{2}}
}_{L^{3/2}}
+\underbracket[0.2pt]{v_{[0]}S_{0}\bigl[v_{[1]}\otimes v_{[0]}\bigr]}_{\substack{\text{also in $L^{3/2}$,}
\\\text{similar to the previous term}}}
\\
+\underbracket[0.2pt]{\underbrace{v_{[0]}}_{L^{6}}   S_{0}\bigl[   \underbrace{v_{[1]}}_{\mathcal W\subset L^{\io}}  \otimes \underbrace{v_{[1]} }_{L^{2}} \bigr]}_{L^{3/2}}.
\end{multline*}
As a consequence, we have 
\begin{multline}\label{tf77jj}
\poscal{S_{0}(v\otimes v)}{v_{[0]}\nabla \chi_{\rho}}=
\poscal{v_{[0]}S_{0}\bigl[v_{[0]}\otimes v_{[0]}\bigr]}{\nabla \chi_{\rho}}
\\
+\poscal{  \underbrace{v_{[0]}S_{0}\bigl[v_{[0]}\otimes v_{[1]}\bigr] }_{\in L^{3/2}}    }{\nabla \chi_{\rho}}
+\poscal{  \underbrace{ v_{[0]}S_{0}\bigl[v_{[1]}\otimes v_{[0]}\bigr] }_{\in L^{3/2}}    }{\nabla \chi_{\rho}}
+\poscal{  \underbrace{ v_{[0]}S_{0}\bigl[v_{[1]}\otimes v_{[1]}\bigr] }_{\in L^{3/2}}    }{\nabla \chi_{\rho}},
\end{multline}
and thus,
using the notation $\trinorm{S_{0}}_{p}$ for the operator-norm of the singular integral $S_{0}$ on $L^{p}$,
\begin{align*}
\val{\poscal{S_{0}(v\otimes v)}{v_{[0]}\nabla \chi_{\rho}}}&\le 
\val{\poscal{v_{[0]}S_{0}\bigl[v_{[0]}\otimes v_{[0]}\bigr]}{\nabla \chi_{\rho}}}
\\
&\hskip25pt + 2\norm{v_{[0]}}_{     L^{6}(\supp\nabla\chi_{\rho})      }  
\trinorm{S_{0}  }_{2}      
\norm{v_{[0]}\otimes v_{[1]}}_{L^{2}}           
\norm{\nabla \chi}_{L^{3}}
\\
&\hskip25pt + \norm{v_{[0]}}_{     L^{6}(\supp\nabla\chi_{\rho})      }                     \trinorm{S_{0} }_{2}
\norm{v_{[1]}\otimes v_{[1]}}_{L^{2}}    
\norm{\nabla \chi}_{L^{3}}  
\\
&\hskip35pt \le 
\val{\poscal{v_{[0]}S_{0}\bigl[v_{[0]}\otimes v_{[0]}\bigr]}{\nabla \chi_{\rho}}}
\\
&\hskip55pt + 2\norm{v_{[0]}}_{     L^{6}(\supp\nabla\chi_{\rho})      }  
\trinorm{S_{0}  }_{2}      
\norm{v_{[0]}}_{\mathcal W} \norm{v_{[1]}}_{L^{2}}           
\norm{\nabla \chi}_{L^{3}}
\\
&\hskip55pt + \norm{v_{[0]}}_{     L^{6}(\supp\nabla\chi_{\rho})      }                     \trinorm{S_{0} }_{2}
\norm{v_{[1]}}_{\mathcal W} \norm{v_{[1]}}_{L^{2}}    
\norm{\nabla \chi}_{L^{3}}  
\end{align*}
We are left with the term
$v_{[0]}S_{0}\bigl[\underbrace{v_{[0]}\otimes v_{[0]}}_{\in L^{9/4}}\bigr]$. There we use the assumption $v_{[0]}$ in $L^{9/2}$
and we get 
\begin{equation}\label{hf541j}
\val{\poscal{v_{[0]}\underbrace{S_{0}\bigl[v_{[0]}\otimes v_{[0]}\bigr]}_{\in L^{9/4}}}{\nabla \chi_{\rho}}}
\le \underbracket[0.1pt]{\norm{\nabla \chi_{\rho}}_{L^{3}}}_{=
\norm{\nabla \chi}_{L^{3}}
}\norm{S_{0}\bigl[v_{[0]}\otimes v_{[0]}\bigr]}_{L^{9/4}}
\norm{v_{[0]}}_{L^{9/2}(\supp \nabla \chi_{\rho})},
\end{equation}
and since $v_{[0]}$ belongs to $L^{9/2}$, we obtain from \eqref{hf541j} that 
$$
\lim_{\rho\rightarrow+\io}\poscal{S_{0}(v\otimes v)}{v_{[0]}\nabla \chi_{\rho}}=0,
$$
so that, using \eqref{413413}, \eqref{411411},
we obtain the following result,
proving Theorem \ref{thm.galnew++}.
\begin{theorem}\label{thm.galnew}
Let us assume that the  assumptions of Conjecture \ref{coj001} are fulfilled. 
Moreover we assume that there exists $\alpha_{0}\in \mooc(\R^{3}_{\xi})$ whose support contains a neighborhood of $0$ such that
$$
\alpha_{0}(D) v\in L^{9/2}(\R^{3}).
$$
Then $v\equiv 0$.
\end{theorem}
\begin{nb} 
 This result means that the assumption $v\in L^{9/2}$ is too strong and that it is enough to assume that $v_{[0]}$, the projection of $v$ onto the space of vectors with spectrum in a given  neighborhood of 0 belongs to $L^{9/2}$.
 We have to pay attention to the locality to keep the localization of terms in $\supp \nabla \chi_{\rho}$ and also to treat carefully the non-linear terms. A key point for doing the latter is that the part of $v$ with large spectrum belongs to $L^{2}$. Typically, that improvement of regularity established above allows us to improve 
a  theorem due to G.P.~Galdi,
formulated in $(1)$, page  \pageref{111kno}.
\end{nb}
\subsection{ An improvement of Chae's result for $\rot^{2}v\in L^{6/5}$}\label{sec.kja5}
The discussion on the improvement of Chae's result (here we want to assume only $\alpha_{0}(D)\rot^2 v\in L^{6/5}$) is more technical.
Since $$-\Delta Q=\rot^{2}v\cdot v-\val{\rot v}^{2},$$ to prove that $\Delta Q$ belongs to $L^{1}$,
 it is enough to check that 
$\rot^{2}v\cdot v\in L^{1}$. 
With $\alpha_{0}\in\moo(\R^{3}, [0,1])$ such  that $\alpha_{0}(\xi)=1$ for $\val \xi\le 1$, $\supp \alpha_{0}\subset 2\mathbb B^{3}$,
$\beta_{1}=1-\alpha_{0},$
$\supp \beta_{1}\subset\{\val \xi\ge 1\}, \beta_{1}(\xi)=1$ for $\val \xi\ge 2$,
we have 
\begin{equation}\label{adsf87}
v=v_{[0]}+v_{[1]},\quad  v_{[0]}=\alpha_{0}(D) v,\quad  v_{[1]}=\beta_{1}(D) v, 
\end{equation}
and
\begin{equation}\label{fff001}
\rot^{2}v\cdot v=\overbracket[0.2pt]{\underbrace{\rot^{2}v}_{L^{2}}\cdot \underbrace{v_{[1]}}_{L^{2}}}^{L^{1}}+\rot^{2}v\cdot v_{[0]}.
\end{equation}
It is enough to check
\begin{equation}\label{fff002}
\rot^{2}v\cdot v_{[0]}=\rot^{2}v_{[1]}\cdot v_{[0]}+
\rot^{2}v_{[0]}\cdot v_{[0]}.
\end{equation}
We are left with, $\mu>2$ to be chosen later,  
\begin{equation}\label{}
\rot^{2}v_{[1]}\cdot v_{[0]}=\underbrace{\rot^{2}v_{[1]}\cdot 
\overbracket[0.1pt]{\beta_{1}(\mu D)\alpha_{0}(D)v}^{ \substack{   \text{spectrum in }\\\val \xi\ge 1/\mu    }}
}_{L^{2}\cdot L^{2}\subset L^{1}}+
\overbracket[0.2pt]{\underbrace{\rot^{2}v_{[1]}}_{\val \xi\ge 1}\cdot \underbrace{
\alpha_{0}(\mu D)v_{[0]}}_{\val \xi\le 2/\mu}}^{  \substack{   \text{spectrum in }\\\val \xi\ge1-\frac2\mu    } }.
\end{equation}
The quantity $\beta_{1}\bigl(\frac{2\mu \xi}{\mu-2}\bigr)$ is equal to 1 on $ \frac{2\mu }{\mu-2}\val \xi\ge 2$,
i.e. on $\val \xi\ge 1-\frac2\mu$ and thus we have 
$$
\rot^{2}v_{[1]}\cdot \alpha_{0}(\mu D)v_{[0]}=\underbrace{\beta_{1}\Bigl(\frac{2\mu}{\mu-2}D\Bigr)}_{
\substack{\text{supported in}\\\val \xi\ge\frac{\mu-2}{2\mu} }
}
\bigl(\rot^{2}v_{[1]}\cdot \alpha_{0}(\mu D)v_{[0]}\bigr).
$$
Defining
\begin{equation}\label{fdc129}
w_{1}=\rot^{2}v_{[1]},
\end{equation}
we have, assuming $\mu\ge 6$,  that\footnote{Indeed
$\beta_{1}(\frac{2\mu \xi}{\mu-2})$ is vanishing  on $\val \xi\le\frac{\mu-2}{2\mu}$,
whereas $\alpha_{0}(\mu \xi)$ is vanishing   in $\val \xi\ge \frac2\mu$ and we have 
$$
 \frac2\mu\le \frac{\mu-2}{2\mu}\text{\quad for $\mu\ge 6.$}
$$
}
$\beta_{1}(\frac{2\mu}{\mu-2}D)\alpha_{0}(\mu D)=0$, 
so that 
\begin{multline*}
\rot^{2}v_{[1]}\cdot \alpha_{0}(\mu D)v_{[0]}=\beta_{1}\bigl(\frac{2\mu}{\mu-2}D\bigr)\bigl(w_{1}(x)\alpha_{0}(\mu D) v_{[0]})
\\=\Bigl[\beta_{1}\bigl(\frac{2\mu}{\mu-2}D\bigr),w_{1}\Bigr]\alpha_{0}(\mu D) v_{[0]}.
\end{multline*}
 We have, choosing $\mu=6$,
 \begin{align}
(2\pi)^{\frac32}\Bigl( \bigl[&\beta_{1}(3D),w_{1}\bigr] u\Bigr)(x)
\notag\\
&=\int {3^{-3}\hat{\beta}}_{1}(\frac13(y-x))w_{1}(y) u(y) dy
-w_{1}(x) \int 3^{-3} {\hat{\beta}}_{1}(\frac13(y-x))u(y) dy
\notag \\
 &=3^{-3}
 \int_{0}^{1}\int {\hat{\beta}}_{1}(\frac13(y-x))u(y)(\nabla w_{1})\bigl(x+\theta(y-x)\bigr)\cdot (y-x) dy d\theta
  \notag\\
 \label{fda258}&=3^{-2}
 \int_{0}^{1}\int u(y)(\nabla w_{1})\bigl(x+\theta(y-x)\bigr)\cdot 
\widehat{D \beta_{1}}(\frac13(y-x))dy d\theta,
\end{align}
so that 
\begin{multline}\label{fff003}
\bigl(\rot^{2}v_{[1]}\cdot \alpha_{0}(6D) v_{[0]}\bigr)(x)\\=(2\pi)^{-\frac32}3^{-2}
 \int_{0}^{1}\int \bigl(\alpha_{0}(6D)v_{[0]}\bigr)(y)(\nabla w_{1})\bigl(x+\theta(y-x)\bigr)\cdot 
\widehat{D \beta_{1}}(\frac13(y-x))dy d\theta.
\end{multline}
We define
\begin{multline}\label{dsq897}
\kappa_{\theta}(x)=
\int \bigl(\alpha_{0}(6D)v_{[0]}\bigr)(y)(\nabla w_{1})\bigl(x+\theta(y-x)\bigr)\cdot 
\gamma(y-x)dy d\theta,
\\\text{with\quad} \gamma(z)=3^{-2}\widehat{D\beta_{1}}(z/3),
\end{multline}
and since $D\beta_{1}$ is smooth compactly supported, we  get 
$\gamma\in \mathscr S(\R^{3})$.
On the other hand, we have the well-known formula, with $\omega=\rot v$,
$$
\nu \rot^{2}\omega+(\underbracket[0.2pt]{v}_{L^{6}}\cdot \underbracket[0.2pt]{\nabla) \omega}_{\substack{L^{3/2}\\
\text{as }\rot^{2}v
}}-(\underbracket[0.2pt]{\omega}_{L^{2}}\cdot \underbracket[0.2pt]{\nabla) v}_{L^{2}}=0,
$$
since, checking the term $\nabla \omega$,  we have,
noting that 
$\partial_{j}\val{D}^{-2}\rot$ is a singular integral and $\rot^{2}v$ belongs to $L^{3/2}$, thanks to Lemma \ref{lem.reg1},
\begin{equation*}
\partial_{j}\omega
=\partial_{j}{{\mathbf C} v}
=\partial_{j}{\vert{D}\vert}^{-2}{\mathbf C}^{2}{\mathbf C}v
=\partial_{j}{\vert{D}\vert}^{-2}{\mathbf C}{\mathbf C}^{2}v,
\quad\text{thus}\quad\nabla\omega\in L^{3/2}.
\end{equation*}
Moreover,  $\omega=\rot v$ and $\nabla v$  are both in $L^{2}$, thanks to Lemma \ref{lem.reg001}.
Since $\frac23+\frac16=\frac56$, we get thus that 
$$
\rot^{2}\omega\in \omega\cdot \nabla v+L^{6/5}.
$$
Nonetheless we have $\omega$ and $\nabla v$ in $L^{2}$, thus $\omega\cdot \nabla v\in L^{1}$,
but $\omega$ and $\nabla v$ are both in the Wiener algebra $\mathcal W$ (cf. Proposition \ref{pro.gfd987}), so that the product
$\omega\cdot \nabla v$ belongs as well to $\mathcal W$, which is included in $L^{\io}$.
As a consequence, we have $\omega\cdot \nabla v\in L^{1}\cap L^{\io}$, which is included in $L^{6/5}$, yielding that
$$
\rot ^{2}\omega\in L^{6/5}, \
\text{i.e. } \rot^{3}v\in L^{6/5},\
\text{which implies }
\rot^{3}v_{[1]}=\rot^{3}\beta_{1}(D) v=\beta_{1}(D)\rot^{3} v\in L^{6/5}.
$$
Since $v_{[1]}$ belongs to $L^{2}$ with a null divergence, we get with $w_{1}$ given in \eqref{fdc129},
$$
\p_{j}w_{1}=\p_{j}\rot^{2} v_{[1]}=\p_{j}\rot^{4} \val{D}^{-2}v_{[1]}=\underbrace{\p_{j}\rot \val D^{-2}}_{\text{sing. int.}}\rot^{3}v_{[1]}
\quad\text{ which is thus in  $L^{6/5}$,}
$$
and this implies that 
\begin{equation}\label{429429}
\nabla w_{1}\in L^{6/5}.
\end{equation}
As a consequence, Formula \eqref{dsq897} reads
\begin{multline}\label{cvd547}
\kappa_{\theta}(x)=
\int a(y)b\bigl(x+\theta(y-x)\bigr)\cdot 
\gamma(y-x)dy,
\\ \text{with  } \gamma\in \mathscr S(\R^{3}), b\in L^{6/5}, a\in
L^{6}\cap L^{\io}. 
\end{multline}
\begin{claim}\label{cla.424242}
 $\kappa_{\theta}$ belongs to $L^{1}$ with a norm bounded above uniformly in $\theta$.
\end{claim}
\begin{proof}[Proof of the claim] 
 With $b\in L^{6/5}$, we check
 \begin{multline*}
\iint \val{a(x+t)}\val{ b\bigl(x+\theta t\bigr)} dx
\val{\gamma(t)}dt 
\\
\le \int \norm{a}_{L^{6}}\norm{b}_{L^{6/5}}
\val{\gamma(t)}dt = \norm{a}_{L^{6}}\norm{b}_{L^{6/5}}\norm{\gamma}_{L^{1}},
\end{multline*}
concluding the proof of the claim.
\end{proof}
As a consequence we get that $\rot^{2}v_{[1]}\cdot \alpha_{0}(6D)v_{[0]}$ belongs to $L^{1}$,
proving that $\Delta Q$ belongs to $L^{1}$.
Formulas \eqref{fff001}, \eqref{fff002}, \eqref{fff003} and Claim \ref{cla.424242} imply that 
\begin{equation}\label{426426}
\rot^{2}v\cdot v\in \rot^{2}v_{[0]}\cdot v_{[0]}+L^{1}=
\bigl(\rot^{2}\alpha_{0}(D)v\bigr)\cdot \bigl(\alpha_{0}(D)v\bigr)+L^{1}.
\end{equation}
The following result proves Theorem \ref{thm.chaenew++}.
\begin{theorem}\label{thm.chaenew}
Let us assume that the  assumptions of Conjecture \ref{coj001} are fulfilled. 
Moreover we assume that there exists $\alpha_{0}\in \mooc(\R^{3}_{\xi})$ whose support contains a neighborhood of $0$ such that
$$
 \alpha_{0}(D)\rot^{2}v\in L^{6/5}(\R^{3}).
$$
Then $v\equiv 0$.
\end{theorem}
\begin{proof}
Using \eqref{426426}, we get that 
$$
\rot^{2}v\cdot v\in \rot^{2}v_{[0]}\cdot v_{[0]}+L^{1}\subset L^{\frac65}\cdot L^{6}+L^{1}\subset L^{1},
$$
and since \eqref{3215+} implies 
$$
-\Delta Q=\rot^{2}v\cdot v-\val{\rot v}^{2},
$$
the fact that $\rot v\in L^{2}$ implies $\Delta Q\in L^{1}$, which yields $v\equiv 0$,
thanks to Proposition 
\ref{pro.417aez}.
\end{proof}
\begin{remark}\rm
 We may note that we have proven a slightly better result than the one formulated in the above theorem. Indeed, according to \eqref{426426}, it is enough to prove that $\rot^{2}v_{[0]}\cdot v_{[0]}$ belongs to $L^{1}$, which is a consequence of $\rot^{2}v_{[0]}$ belongs to $L^{6/5}$ since we know that $v, v_{[0]}$ belong to $L^{6}$.
 On the other hand, we have the following classical formula
 $$
 \rot^{2}v_{[0]}\cdot v_{[0]}=\frac12\val{D}^{2}\bigl(\val{v_{[0]}}^{2}\bigr)+\val{\nabla v_{[0]}}^{2},
 $$
 so that $\rot^{2}v_{[0]}\cdot v_{[0]}$ belongs to $L^{1}$ iff 
 $\val{D}^{2}\bigl(\val{v_{[0]}}^{2}\bigr)$ belongs to $L^{1}$,
 since $\nabla v_{[0]}=\alpha_{0}(D)\nabla v$
 belongs to 
 $L^{2}$ (we have $\nabla v\in L^{2})$.
\end{remark}
\subsection{ Proof of Theorem \ref{thm.54ez}}\label{sec.rsq547}
The proof of  Theorem \ref{thm.54ez} will follow the lines of Section \ref{sec.kja5},
tackling Theorem \ref{thm.chaenew++}. 
With $\alpha_{0}\in\moo(\R^{3}, [0,1])$ such  that $\alpha_{0}(\xi)=1$ for $\val \xi\le 1$, $\supp \alpha_{0}\subset 2\mathbb B^{3}$,
$\beta_{1}=1-\alpha_{0},$
$\supp \beta_{1}\subset\{\val \xi\ge 1\}, \beta_{1}(\xi)=1$ for $\val \xi\ge 2$,
we have 
\begin{equation}
v=v_{[0]}+v_{[1]},\quad  v_{[0]}=\alpha_{0}(D) v,\quad  v_{[1]}=\beta_{1}(D) v, 
\end{equation}
Our assumption is that 
\begin{equation}\label{fds546}
\alpha_{0}(D) \nabla Q\in L^{6/5}, \quad \alpha_{0}\in \mooc(\R^{3}), \alpha_{0}(0)\not=0.
\end{equation}
Let us recall a couple of formulas on the Laplacean of $Q$: we have the pointwise equalities, 
\begin{align}\label{rds452}
-\Delta Q=\rot^{2}v\cdot v-\val{\rot v}^{2}=-\nu^{-1}v\cdot \nabla Q-\val{\rot v}^{2}.
\end{align}
and we set
\begin{equation}\label{}
v_{[0]}=\alpha_{0}(D) v, \quad v_{[1]}=\beta_{1}(D) v, \quad\text{so that $v=v_{[0]}+v_{[1]}$.   }
\end{equation}
\emph{We want to verify that $\Delta Q$ belongs to $L^{1}(\R^{3})$} and since $\rot v$ belongs to $L^{2}$, the second equality in \eqref{rds452} shows that it is enough to check that $v\cdot \nabla Q$ belongs to $L^{1}$.
We have 
$$
v\cdot \nabla Q=v_{[0]}\cdot \nabla Q+\underbracket[0.1pt]{v_{[1]}}_{\in L^{2}}\cdot 
\underbracket[0.1pt]{\nabla Q}_{\in L^{2}},
$$
so that it is enough to check that $v_{[0]}\cdot \nabla Q$ belongs to $L^{1}$.
We have 
$$
\underbracket[0.1pt]{v_{[0]}}_{\substack{\in L^{6}\\
\text{since $v\in L^{6}$}}}\cdot 
\underbracket[0.1pt]{\alpha_{0}(D)(\nabla Q)}_{\substack{\in L^{6/5}\\
\text{\eqref{fds546}}}}+v_{[0]}\cdot \beta_{1}(D)(\nabla Q),
$$
so that it is enough to check that $v_{[0]}\cdot \beta_{1}(D)(\nabla Q)$ belongs to $L^{1}$.
We have for $\mu=6$,  
\begin{equation}\label{434oiu}
\beta_{1}(D)(\nabla Q)\cdot v_{[0]}=\beta_{1}(D)(\nabla Q)\cdot \beta_{1}(\mu D)v_{[0]}
+\beta_{1}(D)(\nabla Q)\cdot \alpha_{0}(\mu D)v_{[0]}.
\end{equation}
We note that the spectrum of $\beta_{1}(D)(\nabla Q)$ is included in $\{\val \xi\ge 1\}$
and that the spectrum of $\beta_{1}(\mu D)v_{[0]}$  is included in $\{\val \xi\ge 1/\mu\}$. Since $Q$ and  $v_{[0]}$ are both in the algebra $\mathcal A$,
this implies that $\beta_{1}(D)(\nabla Q)$ and $\beta_{1}(\mu D)v_{[0]}$ are both in $L^{2}$,
so that the first product in the right-hand-side of \eqref{434oiu} belongs to $L^{1}$: we need only to check that
\begin{equation}\label{swq111}
\beta_{1}(D)(\nabla Q)\cdot \alpha_{0}(\mu D)v_{[0]}\in L^{1}.
\end{equation}
The spectrum of $\alpha_{0}(\mu D)v_{[0]}$ is included in $\{\val \xi\le \frac2\mu=\frac13\}$, so that 
the spectrum  of $\beta_{1}(D)(\nabla Q)\cdot \alpha_{0}(\mu D)v_{[0]}$ is included in 
$$
\{\xi, \val{\xi}\ge 1-\frac2\mu =\frac23\}.
$$
Since the function $\xi\mapsto \beta_{1}(3\xi)$ is equal to 1 on $\{\val \xi\ge 2/3\}$, we obtain that
$$
\beta_{1}(D)(\nabla Q)\cdot \alpha_{0}(\mu D)v_{[0]}=
\beta(3D)\Bigl\{\beta_{1}(D)(\nabla Q)\cdot \alpha_{0}(\mu D)v_{[0]}\Bigr\},
$$
and setting 
\begin{equation}\label{874ggt}
w_{2}=\beta_{1}(D)(\nabla Q),
\end{equation}
we find that
\begin{equation}\label{}
\beta_{1}(D)(\nabla Q)\cdot \alpha_{0}(\mu D)v_{[0]}=
\beta(3D)\Bigl\{w_{2}\cdot \alpha_{0}(\mu D)v_{[0]}\Bigr\}.
\end{equation}
We note also that 
$\beta(3\xi) \alpha_{0}(\mu \xi)=0$ since for $\val \xi\le 2/\mu=1/3$,
we have $\val{3\xi}\le 1$ and thus $\beta(3\xi)=0$; as a consequence, we get
\begin{equation}\label{}
\beta_{1}(D)(\nabla Q)\cdot \alpha_{0}(\mu D)v_{[0]}=
\bigl[\beta(3D),w_{2}\bigr]\bigl(\alpha_{0}(6 D)v_{[0]}\bigr).
\end{equation}
Using now Formula \eqref{fda258}, we obtain 
\begin{multline}\label{tap586}
\bigl(\beta_{1}(D)(\nabla Q)\cdot \alpha_{0}(\mu D)v_{[0]}\bigr)(x)
\\=
3^{-2}
 \int_{0}^{1}\int u(y)(\nabla w_{2})\bigl(x+\theta(y-x)\bigr)\cdot 
\widehat{D \beta_{1}}(\frac13(y-x))dy d\theta(2\pi)^{\frac32},
\\
\text{with } u= \alpha_{0}(\mu D)v_{[0]},
\end{multline}
and since $D\beta_{1}$ is smooth compactly supported, we get  that
$\widehat{D \beta_{1}}$ belongs to the Schwartz class.
We define
\begin{multline}\label{dsq898}
\kappa_{\theta}(x)=
\int \bigl(\alpha_{0}(6D)v_{[0]}\bigr)(y)(\nabla w_{2})\bigl(x+\theta(y-x)\bigr)\cdot 
\gamma(y-x)dy d\theta,
\\\text{with\quad} \gamma(z)=3^{-2}\widehat{D\beta_{1}}(z/3), \gamma\in \mathscr S(\R^{3}).
\end{multline}
\begin{claim}\label{cla.987g}
 We have $\nabla w_{2}\in L^{6/5}$, where $w_{2}$ is defined in \eqref{874ggt}.
\end{claim}
\begin{proof}[Proof of the claim]
 Indeed, we have 
 \begin{equation}\label{}
 \p_{j}w_{2}=\underbracket[0.1pt]{-\p_{j}\beta_{1}(D)\nabla\val{D}^{-2}}_{\text{singular integral } \mathcal R_{0}}\Delta Q
 =\mathcal R_{0}\bigl(\val{\rot v}^{2}-\rot^{2}v\cdot v\bigr).
\end{equation}
 Since $\rot^{2}v\in L^{3/2}$ (cf. Lemma \ref{lem.reg1}) and $v\in L^{6}$, we have $\rot^{2}v\cdot v\in L^{6/5}$.
 Moreover we have $\val{\rot v}^{2}\in L^{1}$ and also $\rot v\in \mathcal W$ since $v$ belongs to $\mathcal A$, so that $\rot v\in \mathcal W\subset L^{\io}$ and thus
 $\val{\rot v}^{2}\in L^{1}\cap L^{\io}\subset L^{6/5}$; we have thus 
 $\val{\rot v}^{2}-\rot^{2}v\cdot v$ in $L^{6/5}$ and since $\mathcal R_{0}$ is a standard singular integral, thus a continuous endomorphism of $L^{6/5}$, this proves the claim.
\end{proof}
Eventually,  from \eqref{tap586}, \eqref{dsq898} and Claim \ref{cla.987g}
we find with $a\in L^{6}, b\in L^{6/5}$, $\gamma\in \mathscr S(\R^{3})$,
that 
\begin{align*}
\norm{\beta_{1}(D)(&\nabla Q)\cdot \alpha_{0}(\mu D)v_{[0]}}_{L^{1}}
\\
&\le \int_{0}^{1}\iint \val{a(y)} \val{b(x+\theta(y-x))}\val{\gamma(y-x)} dy dx d\theta
\\
&\hskip18pt = \int_{0}^{1}\iint \val{a(x+z)} \val{b(x+\theta z)}\val{\gamma(z)} dz dx d\theta
\\
&\hskip36pt\le\norm{a}_{L^{6}} \norm{b}_{L^{6/5}}\norm{\gamma}_{L^{1}}<+\io,
\end{align*}
proving \eqref{swq111} and thus $\Delta Q\in L^{1}$. We may then apply the already proven Corollary \ref{cor.54gf} and conclude that $v$ is identically vanishing.
The proof of Theorem \ref{thm.54ez} is complete.
\subsection{Final comments}
We could summarize our results as follows: writing the Stationary Navier-Stokes System for Incompressible Fluids as 
\begin{equation}\label{lkj558}
\begin{cases}
 \nu \rot^{2}v+\rot v\times v+\nabla Q=0,
\\
 \dive v=0,
 \end{cases}
\end{equation}
where $\rot$ stands for the $\curl$ operator and $Q=p+\frac12\val{v}^{2}$ for 
the Bernoulli head pressure,
using the pointwise cancellation identity 
$$
\proscald{(\rot v)(x)\times v(x)}{v(x)}=\det\bigl((\rot v)(x), v(x), v(x)\bigr)=0,
$$
assuming that $\rot v$ belongs to $L^{2}$, we prove that  
$v$ belongs to the Wiener algebra $\mathcal W$ as well as all its derivatives.
Moreover, we prove as well that 
$\nabla v$ as well as all its derivatives belong to $L^{2}$.
This enhanced regularity results allow us to multiply the (first) equation in \eqref{lkj558}
by $f(Q) v$, where $f$ is a Lipschitz-continuous function
supported in $(-\io, 0)$, taking advantage of the known result that $Q$ takes negative values (or is identically vanishing).
From 
the identity
$$
\poscal{\nu \rot^{2}v+\rot v\times v+\nabla Q}{f(Q) v}=0,
$$
we can extract enough information to obtain that $v=0$,
provided that $v$ goes to 0 at infinity {\bf and} that we have some other piece of information 
on the decay of $v$ at infinity.
We prove in fact the following lemma.
\begin{lemma} Let $v$ be a vector field satisfying the assumptions of Conjecture \ref{coj001} and let $Q$ be the Bernoulli head pressure. We have
with $v=v_{[0]}+v_{[1]},$  
$$
 \spec v_{[0]}=\supp \mathcal F(v_{[0]})=\text{compact neighbd of $0$}, \quad
 0\notin \spec v_{[1]}=\supp \mathcal F(v_{[1]}),
$$
\begin{equation}\label{}
Qv\in v_{[0]} S_{0}\bigl[ v_{[0]}\otimes v_{[0]}\bigr]+v_{[0]} L^{2}+v_{[1]}L^{6}, \quad v_{[1]}
\in L^{2}, \quad\text{$S_{0}$ singular integral,}
\end{equation}
which entails that
$
Qv-v_{[0]} S_{0}\bigl[ v_{[0]}\otimes v_{[0]}\bigr]\in L^{3/2}.
$
\end{lemma}
As a consequence, to obtain that $Qv$ belongs to $L^{3/2}$, it is enough to assume 
that 
$$ v_{[0]}\in L^{p}\quad\text{with }
\frac3 p=\frac23, \quad \text{i.e. } p=\frac92.
$$
On the other hand, we had another lemma.
\begin{lemma} Let $v$ be a vector field satisfying the assumptions of Conjecture \ref{coj001} and let $Q$ be the Bernoulli head pressure. We have
with $v=v_{[0]}+v_{[1]},$  
$$
 \spec v_{[0]}=\supp \mathcal F(v_{[0]})=\text{compact neighbd of $0$}, \quad
 0\notin \spec v_{[1]}=\supp \mathcal F(v_{[1]}),
$$
\begin{equation}\label{}
\rot^{2}v\cdot v\in v_{[1]}\cdot L^{2}+ v_{[0]}\cdot \rot^{2}v_{[1]}
+ v_{[0]}\cdot \rot^{2}v_{[0]}, \quad v_{[1]}\in L^{2},\quad v_{[0]}\in L^{6},
\end{equation}
which entails that
$
\rot^{2}v\cdot v\in  v_{[0]}\cdot \rot^{2}v_{[0]}+v_{[0]}\cdot \rot^{2}v_{[1]}+L^{1}.
$
A more technical result  based on a further spectral localization and a commutation argument
 allows us to prove  that $v_{[0]}\cdot \rot^{2}v_{[1]}$ belongs to $L^{1}$
 (the details were given on page \pageref{cvd547}).
\end{lemma}

As a consequence, to obtain that $\rot^{2}v\cdot v$ belongs to $L^{1}$, it is enough to assume that $v_{[0]}\cdot \rot^{2}v_{[0]}$ belongs to $L^{1}$ and for instance to assume that $\rot^{2}v$ belongs to $L^{6/5}$.
The previous summary suggest to formulate another conjecture,
apparently much easier than Conjecture \ref{coj001}, anyhow \emph{implied} by 
Conjecture \ref{coj001}.
\begin{conjecture}\label{coj002}
 Let $v$ be a vector field satisfying the assumptions of Conjecture \ref{coj001}.
 Let us assume moreover that $v$ is the restriction to $\R^{3}$ of an entire vector field
 of exponential type
 on $\C^{3}$. 
 Then we expect that $v$ vanishes identically.
\end{conjecture}
\begin{nb}
 Thanks to the Paley-Wiener Theorem,
 the above hypothesis means that we assume that the Fourier transform of $v$ is compactly supported.
\end{nb}
A couple of comments are in order:
obviously Conjecture \ref{coj001} implies Conjecture \ref{coj002}, which is thus weaker and it might be interesting to try our hand on the latter statement, a priori easier to tackle. We may note that the hypothesis of 
 Conjecture \ref{coj002} 
 is very strong in the sense that we assume nonetheless that the vector field $v$ is the restriction to $\R^{3}$ of an entire vector field $V$ on $\C^{3}$
 (thus $v$ is analytic on $\R^{3}$ with an entire extension to $\C^{3}$) but moreover that $V$ is an entire function of  exponential type.
 The Paley-Wiener-Schwartz Theorem (see e.g. Theorem 7.3.1 in \cite{MR1996773}) shows that an entire function of exponential type on $\C^{n}$ is the Fourier-Laplace transform
 of a compactly supported distribution. In our case, we already know that if $0$ does not belong to the support of $\hat v$, then, using for instance Theorem \ref{thm.galnew++},
 we obtain that $v$ must vanish identically.
 As a  consequence, we may indeed formulate the above conjecture and assume that $\supp \hat v$ is included in a compact neighborhood of $0$.
 \par
 Note that if the support of $\hat v$ is reduced to $\{0\}$,
  the vector field $v$ is a polynomial
  and since it belongs as well to 
$\mathcal L_{(0)}$ it must vanish identically, thanks to   Lemma \ref{cla.548uyt}.
Note also that, thanks to Lemma \ref{lem.jhg4}, the solution of \eqref{SNSI} given by $v=\nabla \psi$ with $\psi$ harmonic on $\R^{3}$ verifies \eqref{hyp002}. If \eqref{hyp001}
holds true for $v$,
the vector field $\nabla \psi$ is harmonic on $\R^{3}$ and bounded since $v=\nabla \psi$ is bounded,
so that the classical Liouville Theorem (see e.g. Theorem \ref{thm.001}) implies that $v$ is a constant,
which must be 0 to satisfy  \eqref{hyp001}:
$v$ vanishes identically.
We may thus reformulate  Conjecture \ref{coj002} as follows.
\vs\no
{\bf Reformulation of Conjecture \ref{coj002}.}
{\it Let $v$ be a vector field satisfying the assumptions of Conjecture \ref{coj001}.
 Let us assume moreover that $\widehat v$ is supported in a compact neighborhood of $0$ in $\R^{3}$. Then $v$ vanishes identically.
}
\section{\color{magenta}Appendix}
\subsection{ Fourier Transformation, Singular Integrals}\label{sec.fourier}
\subsubsection*{\underline{Fourier Transformation}}
The Fourier transformation is defined on $\mathscr S(\R^{3},\C^{3})$ by
\begin{equation}\label{app1}
(\mathcal F w)(\xi)=\int e^{-ix\cdot \xi} w(x) dx (2\pi)^{-3/2}\quad \text{ (noted also\ }\hat w(\xi)),
\end{equation}
and we have for $w\in \mathscr S(\R^{3},\C^{3})$,
\begin{equation}\label{app2}
 w(x)=\int e^{ix\cdot \xi} \hat w(\xi) d\xi (2\pi)^{-3/2}.
\end{equation}
Both formulas \eqref{app1}, \eqref{app2} can be extended to $L^{2}(\R^{3},\C^{3})$
(and in fact also to
$\mathscr S'(\R^{3},\C^{3})$)\footnote{To define 
$\mathcal F$ on 
 $\mathscr S'(\R^{3},\C^{3})$,
 we set
 \begin{equation}\label{}
 \poscal{\mathcal FT}{\phi}_{\mathscr S'(\R^{3},\C^{3}), \mathscr S(\R^{3},\C^{3})}
 =
 \poscal{T}{\mathcal F\phi}_{\mathscr S'(\R^{3},\C^{3}), \mathscr S(\R^{3},\C^{3})},
\end{equation}
and it is easy to prove  that, on $\mathscr S'(\R^{3},\C^{3})$ as well, we have 
\begin{equation}\label{}
\mathcal C_{0} \mathcal F=\mathcal F\mathcal C_{0}= \mathcal F^{*}, \quad \mathcal C_{0}\mathcal F^{2}=\Id.
\end{equation}
Also the constant vector field $a$ is such that $(\mathcal F a)(\xi)=\delta_{0}(\xi)(2\pi)^{3/2} a$.
}  
and $\mathcal F$ appears as a unitary transformation of $L^{2}(\R^{3},\C^{3})$ such that 
\begin{equation}\label{}
\mathcal F^{-1}=\mathcal C_{0} \mathcal F=\mathcal F\mathcal C_{0}= \mathcal F^{*},
\end{equation}
where $\mathcal C_{0}$ is the self-adjoint and unitary operator defined  by 
\begin{equation}\label{czero}
(\mathcal C_{0}w)(x)=w(-x).
\end{equation}
Moreover, for $w\in \mathscr S(\R^{3}, \C^{3}),$
we shall define the \emph{spectrum} of $w$ as
\begin{equation}\label{spectre}
\spec{w}=\supp \hat w.
\end{equation}
\begin{lemma}\label{lem.fourea}
A vector field $w$ in $\mathscr S'(\R^{3},\C^{3})$
 is real-valued if and only if
 $\re(\hat w)$ is even and $\im(\hat w)$ is odd.
\end{lemma}
\begin{proof}
 Let $w$ be in  $\mathscr S(\R^{3},\C^{3})$. We then have 
\begin{align*}
&(2\pi)^{3/2} 2i\im(w(x))=2i\im\Bigl(\int e^{ix\cdot \xi}\hat w(\xi) d\xi\Bigr)
 \\
 &=\int \bigl[e^{ix\cdot \xi}\hat w(\xi)-e^{-ix\cdot \xi}\overline{\hat w(\xi)}\bigr] d\xi
=\int e^{ix\cdot \xi}\bigl[\hat w(\xi)-\overline{\hat w(-\xi)}\bigr] d\xi.
\end{align*}
 As a consequence, $w$ is real-valued is equivalent to the identity
 $
 \hat w(\xi)=\overline{\hat w(-\xi)},
 $
 that is to the identity
 $$
 \re(\hat w(\xi))+i\im(\hat w(\xi))=\re(\hat w(-\xi))-i\im(\hat w(-\xi)),
 $$
 proving the lemma for $w\in \mathscr S(\R^{3},\C^{3})$.
 Let $w$ be in $\mathscr S'(\R^{3},\C^{3})$; we define $\overline w$ by 
 $$
 \poscal{\overline w}{\phi}_{\mathscr S'(\R^{3},\C^{3}), \mathscr S(\R^{3},\C^{3})}
 = \overline{\poscal{w}{\overline \phi}}_{\mathscr S'(\R^{3},\C^{3}), \mathscr S(\R^{3},\C^{3})},
 $$
and we have 
 $$
2i\im(w)=w-\overline w=\mathcal F^{-1}\bigl[\mathcal F w-\overline{ \mathcal C_{0}\mathcal Fw}\bigr],
 $$
 so that the reality of $w$ is equivalent to $\mathcal F w=\overline{ \mathcal C_{0}\mathcal Fw}$,
 which gives the same result as above and the lemma.
\end{proof}
\begin{nb}
 It is  important to stay with real-valued vector fields
 as solutions of the Navier-Stokes system, but these vector fields may have complex-valued Fourier transforms and the above lemma is providing the way to keep track of the reality of $w(x)$ via some algebraic properties of the real part and imaginary part of $\hat w$.
\end{nb}
\begin{lemma}\label{lem.kj98}
 Let $a,b,$ be vector fields in $\R^{3}$ such that $a\in L^{p}(\R^{3}, \R^{3}), b\in L^{q}(\R^{3}, \R^{3})$ with 
 $p,q\in [1,+\io]$ and 
 $\frac1p+\frac1q\le 1$. Then $a\times b$ belongs to $L^{r}(\R^{3}, \R^{3})$ with 
 $\frac1p+\frac1q=\frac1r$ and we have 
 $$
 \norm{a\times b}_{L^{r}}\le \norm{a}_{L^{p}}\norm{b}_{L^{q}}.
 $$
\end{lemma}
\begin{proof}
 This lemma is an obvious consequence of H\"older's inequality, except for the constant 1 in the right-hand-side of the inequality.
 Recalling Hadamard's inequality for $N$ vectors $C_{1}, \dots, C_{N}$ in $\R^{N}$ equipped with its canonical Euclidean norm, we have 
 \begin{equation}\label{had001}
\val{C_{1}\wedge\dots\wedge C_{N}}\le \prod_{1\le j\le N}\norm{C_{j}}_{\R^{N}}.
\end{equation}
We have thus
\begin{multline*}
\norm{a\times b}_{L^{r}(\R^{3})}=\sup_{\substack{   c\in L^{r'}(\R^{3})  \\ \norm{c}_{ L^{r'}} =1 }}
\Val{\int_{\R^{3}}\proscal3{a(x)\times b(x)}{c(x)} dx}
\\
=\sup_{\substack{   c\in L^{r'}(\R^{3})  \\ \norm{c}_{ L^{r'}} =1 }}
\Val{\int_{\R^{3}}\det\bigl({a(x), b(x), c(x)}\bigr) dx}
\\
\le 
\sup_{\substack{   c\in L^{r'}(\R^{3})  \\ \norm{c}_{ L^{r'}} =1 }}
{\int_{\R^{3}}\norm{a(x)}_{\R^{3}} \norm{b(x)}_{\R^{3}}\norm{c(x)}_{\R^{3}}dx}
\le \norm{a}_{L^{p}}\norm{b}_{L^{q}},
\end{multline*}
 concluding the proof.
 \end{proof}
\begin{lemma}
Let $a,b,$ be vector fields in $\R^{3}$ such that
 $a\in L^{p}(\R^{3}, \R^{3}), b\in L^{q}(\R^{3}, \R^{3})$ with 
 $p,q\in [1,+\io]$ and 
 $\frac1{p'}+\frac1{q'}\le 1.$
 We define
\begin{equation}\label{}
(a\star b)(x)=\int_{\R^{3}}a(x-y)\times b(y)dy(2\pi)^{-\frac 32},
\end{equation}
and we have $a\star b\in L^{r}$ with $\frac1{p'}+\frac1{q'}=\frac1{r'}$
as well as
 $$
 \norm{a\star b}_{L^{r}}\le \norm{a}_{L^{p}}\norm{b}_{L^{q}}(2\pi)^{-3/2}.
 $$
\end{lemma}
\begin{proof}
 This lemma is an obvious consequence of Young's inequality and we obtain readily the result by using again \eqref{had001}.
\end{proof}
\begin{remark}\rm Let $a,b$ be satisfying the assumptions of Lemma \ref{lem.kj98} with $p=q=2$.
 We have $\hat a\in L^{2}, \hat b\in L^{2}$ and $(\hat a\oast\hat b)\in \mathcal W\subset C^{0}_{(0)}\subset L^{\io}$ so that \begin{gather}\label{}
\mathcal F\bigl( a\times b\bigr)(\xi)=(\hat a\star\hat b)(\xi).\end{gather}
\end{remark}
\subsubsection*{\underline{Functions with limit $0$ at infinity}}
\begin{definition}\label{def.417ooo}
 Let $u$ be in $L^{1}_{\text{\rm loc}}(\R^{d})$. We shall say that $u$ goes to 0 at infinity if 
 \begin{equation}\label{fds555+}
\forall t>0, \quad \Val{\{x\in \R^{d}, \val{u(x)}>t\}}<+\io,
\end{equation}
where $\val{A}$ stands for the Lebesgue measure of $A$
on $\R^{d}$.
\end{definition}
\begin{lemma}\label{lem.jhgf}
 The set 
 \begin{equation}\label{l0inft}
\mathcal L_{(0)}(\R^{d})=\{u\in L^{1}_{\text{\rm loc}}(\R^{d}), \text{such that }
\eqref{fds555+} \text{ holds true}\},
\end{equation}
 is a vector space.
Moreover, for any $p\in [1,\io)$,
\begin{align}\label{}
\mathcal L^{[p]}_{(0)}(\R^{d})&=\bigl\{u\in L^{1}_{\text{\rm loc}}(\R^{d}), \text{$\exists R_{0}\ge 0$ so that } u\in L^{p}(\{\val x\ge  R_{0}\})\bigr\},
\label{lp0000}
\\
\mathcal L^{[\io]}_{(0)}(\R^{d})&=\bigl\{u\in L^{1}_{\text{\rm loc}}(\R^{d}), \text{$\exists R_{0}\ge 0$ so that  } u\in L^{\io}(\{\val x\ge  R_{0}\})
\label{linfi0}\\
&\text{\hskip155pt and } \lim_{R\rightarrow +\io}
\norm{u}_{L^{\io}(\{\val x\ge R\})}=0\bigr\},
\notag\end{align}
are both subspaces of $\mathcal L_{(0)}(\R^{d})$.
\end{lemma}
\begin{proof}
Indeed, if $u_{1},u_{2}$ belong to $\mathcal L_{(0)}(\R^{d})$, $\lambda_{1},\lambda_{2}$ are scalars, we have for $t>0$,
\begin{multline*}
\{x, \val{\lambda_{1}u_{1}(x)+\lambda_{2}u_{2}(x)}>t\}
\subset
\{x, \val{\lambda_{1}}\val{u_{1}(x)}>t/2\}\cup
\{x, \val{\lambda_{2}}\val{u_{2}(x)}>t/2\}
\\\text{\tiny (assuming both $\lambda_{j}\not=0$)}
=
\{x, \val{u_{1}(x)}>\frac{t}{2\val{\lambda_{1}}}\}
\cup
\{x, \val{u_{2}(x)}>\frac{t}{2\val{\lambda_{2}}}\},
\end{multline*}
where both sets have finite measure. If $\lambda_{1}=0$ (resp. $\lambda_{2}=0$), then 
$\{\val{\lambda_{1}}\val{u_{1}(x)}>t/2\}=\emptyset$
(resp.$\{\val{\lambda_{2}}\val{u_{2}(x)}>t/2\}=\emptyset$),
and we get the sought vector space property.
Moreover if $u$ belongs to the space \eqref{lp0000},
we have for $t>0$, with $K_{0}=R_{0}\mathbb B^{d}$,
$E=\{x\in \R^{d}, \val{u(x)}>t\}$ that $$
\val{E}=\underbracket[0.2pt]{\val{E\cap K_{0}}}_{\text{finite}}+\val{E\cap K_{0}^{c}},$$
and 
$$
\val{E\cap K_{0}^{c}}\le  \int_{E\cap K_{0}^{c}}\val{u(x)}^{p}t^{-p}dx\le t^{-p}\norm{u}_{L^{p}(E\cap K_{0}^{c})}^{p}<+\io,
$$
proving our claim for \eqref{lp0000}.
 If $u$ belongs to the space \eqref{linfi0},
we  define
$$
\varepsilon(R)=\norm{u}_{L^{\io}(\{\val x\ge R\})},
$$
for which we know that $\lim_{R\rightarrow+\io}\varepsilon(R)=0$ and $\varepsilon$ is decreasing.
 We have for $t>0$ given,
 \begin{multline*}
 \{\val{u(x)}>t\}\subset
(R_{0}+1)\mathbb B^{d}
 \cup_{k> [R_{0}]=k_{0}} \bigl\{\val{u(x)}>t, k\le \val x< k+1\bigr\}
 \\
 \subset (R_{0}+1)\mathbb B^{d}\cup_{k> k_{0}} \Bigl[\bigl\{t<\val{u(x)}\le \varepsilon(k), k\le \val x<k+1\bigr\}\cup N_{k}\Bigr],
\end{multline*}
where the $N_{k}$ have measure 0.
We note that there is an integer $k_{t}$ such that
$\varepsilon(k_{t})<t$ since $t>0$ and 
$\lim_{k\rightarrow+\io}\varepsilon(k)=0.$
Since $\varepsilon$ is decreasing, we find
$$
\Val{\{\val{u(x)}>t\}\backslash (R_{0}+1)\mathbb B^{d}}
\le  
\sum_{k_{0}\le k<k_{t}}\Val {\{t<\val{u(x)}\le \varepsilon(k), k\le \val x<k+1\}}<+\io,
$$
proving the sought property for \eqref{linfi0}.
\end{proof}
\begin{definition}\label{def.145gfd}
Let $p\in [1,+\io]$. We define the space 
$L^{p}_{(0)}$ as the set of $v\in L^{1}_{\text{loc}}(\R^{3},\R^{3})$ such that 
there exists $R_{0}\ge 0$ so that 
$v\in (L^{p}\cap L^{\io})(\{\val x\ge R_{0}\},\R^{3})$
and 
$$
\lim_{R\rightarrow+\io}\norm{v}_{L^{\io}(\{\val x\ge R \}, \R^{3})}=0.
$$
 \end{definition}
 Note that \eqref{hyp001plu} is obviously satisfied  for the functions in $L^{p}_{(0)}$ so that we have
 $$
 \cup_{1\le p\le +\io}L^{p}_{(0)}\subset \cup_{1\le p\le +\io}\mathcal L_{(0)}^{{[p]}}\subset \mathcal L_{(0)}.
 $$
 The following proposition is obvious.
\begin{proposition}Let $p_{1}, p_{2}\in [1,+\io]$ such that $p_{1}\le p_{2}$.
 Then we have 
 $
 L^{p_{1}}_{(0)}\subset L^{p_{2}}_{(0)}.
 $
 Moreover, if $v\in L^{p}_{(0)}$ for some $p\in [1,+\io]$, Property \eqref{hyp001plu} holds true, which implies that \eqref{hyp001} is satisfied, thanks to \eqref{541hgx}.
 \end{proposition}
\begin{lemma}\label{cla.548uyt}
 Let $P$ be a polynomial in $\mathcal L^{(0)}(\R^{d})$. Then $P$ is the zero polynomial.
\end{lemma}
\begin{proof}
{\sl Reductio as absurdum:} let us assume that $P$ is not the zero polynomial,
i.e. for some $m\in \N$,
$$
P(x)=\underbrace{\sum_{\substack{\alpha\in \N^{d}\\ \val \alpha=m}} c_{\alpha}x^{\alpha}}_{P_{m}(x)}+
\underbrace{\sum_{\substack{\alpha\in \N^{d}\\ \val \alpha\le m-1}} c_{\alpha}x^{\alpha}}_{P_{m-1}(x)},
\quad \text{$P_{m}$ is not the zero polynomial}.
$$
 Since $P_{m}$ is not the zero polynomial, there exists $\omega_{0}\in \mathbb S^{d-1}$ such that
 $P_{m}(\omega_{0})\not=0$, and thus 
 $$
 V_{0}=\{\omega\in \mathbb S^{d-1}, \val{P_{m}(\omega)}\ge \frac12\val{P_{m}(\omega_{0})}\}
 \text{ is a neighborhood of $\omega_{0}$ in $\mathbb S^{d-1}$.}
 $$
 Thus for $x\in \R^{d}, x\not=0, \frac x{\val x}\in V_{0}$, we have 
 \begin{equation*}
 \val{P(x)}\ge \frac12\val{P_{m}(\omega_{0})}\val x^{m}-\sigma_{0}\bigl(1+\val x\bigr)^{m-1},
\end{equation*}
 so that there exists $R_{0}>0, \gamma_{0}>0$ such that
 \begin{equation}\label{54kjf9}
 x\in \Gamma_{0}=\bigl\{x\in \R^{d}, \val x\ge R_{0}, \frac x{\val x}\in V_{0}\bigr\}
 \Longrightarrow\val{P(x)}> \gamma_{0}.
\end{equation}
We have 
$$
\val{\Gamma_{0}}_{d}=\int_{R_{0}}^{+\io} r^{d-1}\val{V_{0}}_{d-1} dr=+\io,
$$
violating Property  \eqref{fds555+},
 concluding the proof of the lemma.
 \end{proof}
\begin{claim}\label{app.pro.123}
 Let $u$ be a fonction on $\R^{d}$ such that
 $$
 u=P+u_{1}+u_{2}, \quad\text{where $P$ is a polynomial, $u_{j}\in L^{p_{j}}$, $1\le p_{j}<+\io$.}
 $$
 Let us assume that $u$ belongs to $\mathcal L_{(0)}(\R^{d})$.
Then $P$ is the zero polynomial.
\end{claim}
\begin{proof}
The functions  $u, u_{1}, u_{2}$ belong to $\mathcal L^{(0)}(\R^{d})$, which is a vector space,
thanks to Lemma \ref{lem.jhgf}, so that the polynomial $P$ belongs as well  to $\mathcal L^{(0)}(\R^{d})$,
and thus is the zero polynomial, thanks to Lemma \ref{cla.548uyt}.
\end{proof}
\begin{claim}\label{app.pro.124}
 Let $\mathfrak h$ be a harmonic fonction on $\R^{d}$ such that
\begin{equation}\label{sfd654}
\mathfrak h=u_{1}+u_{2}, \quad\text{where  $u_{j}\in L^{p_{j}}$, $1\le p_{j}<+\io$.}
\end{equation}
Then $\mathfrak h=0$.
\end{claim}
\begin{proof}
Since it follows from the equality \eqref{sfd654} that  $\mathfrak h$ is a temperate distribution,
 the argument of Theorem 1.1 gives that $\mathfrak h$ is a polynomial.
 That polynomial belongs to $\mathcal L^{(0)}(\R^{d})$, thanks to Lemma \ref{lem.jhgf}, and using Lemma \ref{cla.548uyt}, it entails that $\mathfrak h=0$,
concluding the proof.
\end{proof}
\begin{remark}\label{rem.4578}\rm
Definition \ref{def.417ooo} is quite convenient, in particular since the vector space 
$\mathcal L^{(0)}(\R^{d})$ contains the spaces 
 \eqref{lp0000} and \eqref{linfi0}. However, it is a bit disturbing to see that the space \eqref{l0inft}
 is \emph{not} included in $\mathscr S'(\R^{d})$: Take for instance $d=1$ and $x\in \R$,
 $$
 u(x)=\sum_{k\ge 1}\mathbf 1_{[k,k+k^{-2}]}(x) k^{2}e^{(k+1)^{2}}.
 $$
We find that
$
\{x, \val{u(x)}>0\}=\cup_{k\ge 1}[k,k+k^{-2}],
$
so that 
$$
\val{\{x, \val{u(x)}>0\}}=\sum_{k\ge 1}k^{-2}=\frac{\pi^{2}}6<+\infty,
$$
proving that $u\in \mathcal L_{(0)}(\R)$. However the function $u$ is \emph{not} a tempered distribution on $\R$:
indeed we have with a function $\chi_{0}$ defined by \eqref{fun001} ($d=1$), and $N\ge 1$, 
\begin{multline}\label{54m98d}
\int u(x) \chi_{0}(x/N) e^{-x^{2}} dx=\sum_{1\le k\le 2N}\int_{[k,k+\frac1{k^{2}}]}
k^{2} e^{(k+1)^{2}-x^{2}}\chi_{0}(x/N) dx
\\\ge 
\sum_{1\le k\le N}\int_{[k,k+\frac1{k^{2}}]}
k^{2} e^{(k+1)^{2}-x^{2}}dx\ge N.
\end{multline}
On the other hand we have, with $\phi_{N}(x)= \chi_{0}(x/N) e^{-x^{2}}$ 
($\phi_{N}\in \mathscr S(\R)$),
$\Gamma(x)=e^{-x^{2}}$,
$\alpha, \beta\in \N$,
\begin{align}\label{546uty}
p_{\alpha,\beta}(\phi_{N})
&=\sup_{x\in \R}\Bigl\vert{x^{\alpha}\sum_{\beta_{1}+\beta_{2}=\beta}}
\frac{\beta!}{\beta_{1}!\beta_{2}!}\chi_{0}^{(\beta_{1})}(x/N) N^{-\beta_{1}} \Gamma^{(\beta_{2})}(x)
\Bigr\vert
\\
&\le \sum_{\beta_{1}+\beta_{2}=\beta}\frac{\beta!}{\beta_{1}!\beta_{2}!}
\norm{\chi_{0}^{(\beta_{1})}}_{L^{\io}(\R)} p_{\alpha,\beta_{2}}(\Gamma)
\notag
\\
&\le 2^{\beta}\bigl(\max_{\beta_{1}\le \beta}\norm{\chi_{0}^{(\beta_{1})}}_{L^{\io}(\R)}\bigr)
\bigl(\max_{\beta_{2}\le \beta}p_{\alpha,\beta_{2}}(\Gamma)\bigr)
=c(\beta)\max_{\beta_{2}\le \beta}p_{\alpha,\beta_{2}}(\Gamma).
\notag
\end{align}
If $u$ were a tempered distribution, we could find $C_{0}, m_{0}$
such that  for all $N\ge 1$,
$$
\val{\poscal{u}{\phi_{N}}_{\mathscr D'(\R),\mathscr D(\R)}}\le 
C_{0}\max_{\val \alpha+\val \beta\le m_{0}} p_{\alpha,\beta}(\phi_{N}),
$$
and from \eqref{54m98d},\eqref{546uty}, this would imply 
$$
N\le \val{\poscal{u}{\phi_{N}}_{\mathscr D'(\R),\mathscr D(\R)}}\le 
C_{0}\max_{\val \alpha+\val \beta\le m_{0}} p_{\alpha,\beta}(\phi_{N})
\le C_{0}\max_{\substack{\val \alpha+\val \beta\le m_{0}
\\
\beta_{2}\le \beta}} c(\beta)p_{\alpha,\beta_{2}}(\Gamma),
$$
which is impossible, since the last term in the right-hand-side does not depend on $N$.
\end{remark}
\begin{remark}\rm
 We may note also that there exists $\psi\in \moo (\R,\R)$ such that for all $l\in\N$,
 $\psi^{(l)}\in \mathcal L_{(0)}(\R)$ and for all $l\in \N$ and all $K$ compact subset of $\R$,
 $\psi^{(l)}\notin L^{\io}(K^{c}):$
 Indeed, with $\chi$ defined by \eqref{fun001} with $d=1$,
 we set
 \begin{equation}\label{}
 \psi(x)=\sum_{k\ge 2} \chi\bigl((x-2^{k}) k^{2}\bigr) k,
\end{equation}
noting that the support of $x\mapsto\chi\bigl((x-2^{k}) k^{2}\bigr)$
is
$$
J_{k}=\{x\in \R, \val{x-2^{k}}\le 2k^{-2}\}=[2^{k}-\frac2{k^{2}}, 2^{k}+\frac2{k^{2}}],
$$
so that the $(J_{k})_{k\ge 2}$ are two by two disjoint\footnote{We need to check for $k\ge 2$,
$
2^{k}+\frac2{k^{2}}<2^{k+1}-\frac2{(k+1)^{2}}.
$
This is equivalent to 
$
2(k+1)^{2}+2k^{2}<2^{k}k^{2}(k+1)^{2}.
$
Since we assume that $k\ge 2$, it is enough to verify
$
4k^{2}+4k+2<2^{2}k^{2}(k+1)^{2},
$
which is equivalent to $ 4k+2<4k^{4}+8k^{3}$,
which is true for $k\ge 1$
(since for $k\ge 1$, we have $4k^{4}\ge 4k$ and $8k^{3}> 2$).
}; the function $\psi$ is thus $\moo$ and 
$$
\psi^{(l)}(x)=\sum_{k\ge 2} \chi^{(l)}\bigl((x-2^{k}) k^{2}\bigr) k^{1+2l}.
$$
\begin{claim}
 For any $l\in \N$, we have $\psi^{(l)}\in \mathcal L_{(0)}(\R)$.
\end{claim}
Let $l\in \N$ be given and let $\chi_{l}=\val{\chi^{(l)}}$ (which is not identically 0):
the  set 
\begin{equation}\label{fad555}
\chi_{l}([-2,2])
\text{ is a compact interval with positive measure  contained in $\R_{+}$:}
\end{equation}
it is a compact interval of $\R_{+} $containing $0$ since $\chi_{l}(\pm 2)=0$
and if it were reduced to $\{0\}$, $\chi_{l}$ would be 0 on $[-2,2]$ and since $\supp\chi=[-2,2]$, $\chi_{l}$ 
would be zero and $\chi$ would be a polynomial,
which is not possible for a non-zero compactly supported function.
Let $t>0$ be given.
We have 
\begin{align*}
\{x, \val{\psi^{(l)}(x)}>t\}&=\cup_{k\ge 2}
\{x, \val{\chi^{(l)}((x-2^{k})k^{2})}k^{1+2l}>t\}
\\&=\cup_{k\ge 2}
\{x, \val{\chi^{(l)}((x-2^{k})k^{2})}>tk^{-1-2l}\}
\\
&\subset\cup_{k\ge 2}
\{x, \val{\chi^{(l)}((x-2^{k})k^{2})}>0\}\subset
\cup_{k\ge 2}
\{x, (x-2^{k})k^{2}\in [-2,2]\}
\end{align*}
so that 
$$
\Val{\{x, \val{\psi^{(l)}(x)}>t\}}\le \sum_{k\ge 2}k^{-2}\Val{[-2,2]+2^{k}}=
4\bigl(\frac{\pi^{2}}6-1\bigr)<+\io,
$$
proving the claim.
\begin{claim}
  For all $l\in \N$ and all compact subsets $K$ of $\R$,
  we have  $\psi^{(l)}\notin L^{\io}(K^{c})$.
\end{claim}
\begin{proof}
Indeed we have for $k\ge 2$, 
$
\psi(2^{k})=\chi(0)k=k
$
and since $\psi$ is continuous,
$$
\{x\in \R, \psi(x)>k/2\}\text{ is a neighborhood of $2^{k}$},
$$
proving that $\psi$ is unbounded on the complement of any compact set.
Thanks to \eqref{fad555},
for any $l$, the set $\chi_{l}((-2,2))$ contains a positive point $t_{l}$,
so that there exists
$x_{l}\in (-2,2)$ such that $\chi_{l}(x_{l})=t_{l}>0$.
We get then for $k\ge 2$, 
$$
\val{\psi^{(l)}(2^{k}+k^{-2} x_{l})}=k^{1+2l}\chi_{l}(x_{l})=k^{1+2l}t_{l},
$$
and since $\psi^{(l)}$ is continuous,
$$
\{x\in \R, \psi^{(l)}(x)>k^{1+2l}t_{l}/2\}\text{ is a neighborhood of $2^{k}+k^{-2} x_{l}$},
$$
proving that $\psi^{(l)}$ is unbounded on the complement of any compact set,
since the sequence 
$(2^{k}+k^{-2}x_{l})_{k\ge 2}$ goes to $+\infty$ when $k$ tends to $+\infty$.
\end{proof}
\end{remark}
\begin{claim}\label{cla.54po} There exist bounded functions $f$  in $\mathcal L_{(0)}(\R^{d})\cap C^{\infty}(\R^{d})$ such that 
\begin{equation}\label{dam298}
\limsup_{R\rightarrow+\io}\norm{f}_{L^{\io}(\{\val x\ge R\})}>0.
\end{equation}
\end{claim}
\begin{proof}[Proof of the claim]
 Indeed, with $\chi$ defined by \eqref{fun001} with $d=1$,
 we set
 \begin{equation}\label{}
f(x)=\sum_{k\ge 2} \chi\bigl((x-2^{k}) k^{2}\bigr).
\end{equation}
We note  that the support of $x\mapsto\chi\bigl((x-2^{k}) k^{2}\bigr)$
is
$$
J_{k}=\{x\in \R, \val{x-2^{k}}\le 2k^{-2}\}=[2^{k}-\frac2{k^{2}}, 2^{k}+\frac2{k^{2}}],
$$
so that the $(J_{k})_{k\ge 2}$ are two by two disjoint and  $f$ is a smooth function, and  we have also that 
$\norm{f}_{L^{\io}(\R)}=1$. Moreover, we have 
$$
\{x\in \R, \val{f(x)}>0\}=\cup_{k\ge2} J_{k}\Longrightarrow
\val{\{x\in \R, \val{f(x)}>0\}}=\sum_{k\ge 2}4k^{-2}<+\io,
$$
proving that $f\in \mathcal L_{(0)}(\R)$. On the other hand, since $f=1$ on $$\cup_{k\ge 2}[2^{k}-\frac1{k^{2}}, 2^{k}+\frac1{k^{2}}],
$$
we have for any $R\ge 0$, 
$
\norm{f}_{L^{\io}(\{\val x\ge R\})}=1,
$
proving \eqref{dam298}.
\end{proof}
\begin{claim}
 There exists $f$ belonging to $\mathcal L_{(0)}(\R^{d})\cap\moo(\R^{d})$
 such that
 $$
 \val{\{x, \val{f(x)}>0\}}=+\io
 $$
\end{claim}
\begin{proof}[Proof of the claim]
Indeed, with $\chi$ defined by \eqref{fun001} with $d=1$,
 we set
 \begin{equation}\label{}
f(x)=\sum_{k\ge 2} \chi\bigl((x-2^{k}) k\bigr) k^{-1}.
\end{equation}
We note  that the support of $x\mapsto\chi\bigl((x-2^{k}) k\bigr)$
is
$$
J_{k,1}=\{x\in \R, \val{x-2^{k}}\le 2k^{-1}\}=\bigl[2^{k}-\frac2{k}, 2^{k}+\frac2{k}\bigr].
$$
The $(J_{k,1})_{k\ge 2}$ are two by two disjoint if for all ${k\ge 2}$
$$
2^{k}+\frac2{k}<2^{k+1}-\frac2{(k+1)}.
$$
That condition is equivalent  to 
$
\frac2{k}+\frac2{(k+1)}< 2^{k},
$
which is equivalent to
$
\bigl(\frac{k+1}{k}\bigr)+1<2^{k-1}(k+1)
$
that is to 
$$
\underbracket[0.1pt]{(1+\frac1k)+1}_{\text{decreasing of $k\ge 2$}}<\underbracket[0.1pt]{2^{k-1}(k+1)}_{\text{increasing of $k\ge 2$}}.
$$ 
It is then enough to check that
$
(1+\frac12)+1<2\times 3=6,
$
which is true.
As a consequence, the function $f$ is $\moo$
and we have 
$$
\{x, f(x)>0\}=\cup_{k\ge 2}\text{interior}(J_{k,1}), \quad\text{so that\ }
\val{\{x, f(x)>0\}}=\sum_{k\ge 2}\frac4k=+\io.
$$
We have also for $s>0$,
$$
\{x, f(x)>s\}\subset\cup_{\substack{k\ge 2\\k^{-1}\ge s}} J_{k,1},\text{ implying, }
\val{\{x, f(x)>s\}}\le \sum_{2\le k\le s^{-1}}\frac4k<+\io,
$$
proving the claim.
\end{proof}
\subsubsection*{\underline{Singular Integrals}}\label{sec.singular}
The basic results  on Fourier multipliers bounded on $L^{p}$ for $p\in (1,+\io)$
are given in Section 6.2.3 of L.~Grafakos' book,
\emph{Classical {F}ourier analysis}
 \cite{MR3243734}, grounded on several works of 
 A. Calder\'on \& A. Zygmund (\cite{MR84633},
\cite{MR87810}),
 L. H\"ormander (Theorem 7.9.5 in \cite{MR1996773}),
 S.G.~Mihlin \cite{MR80799},
 E.M. Stein \cite{MR1232192}.
 It will be enough for us to remind a particular case,
 following from Theorem 6.2.7 in  \cite{MR3243734}.
 \begin{theorem}\label{thm.singsing}
 Let $\R^{3}\backslash\{0\}\ni \xi\mapsto m(\xi)$, be a smooth map, where  $m(\xi)$ is a  $3\times 3$ complex-valued matrix such that
 \begin{equation}\label{}
 \forall \alpha\in \N^{3}, \val \alpha\le1+ [n/2],\quad 
\sup_{\xi\not=0}\val{\xi}^{\val\alpha}\val{\p_{\xi}^{\alpha}m(\xi)}<+\infty.
\end{equation}
Then for all $p\in ]1, +\io[$, the Fourier multiplier $m(D)$ is bounded on $L^{p}(\R^{3}, \R^{3})$.
\end{theorem}
\begin{remark}\rm
 We have $m(D)=\mathcal F^{-1} m(\xi)\mathcal F=\mathcal F^{*} m(\xi)\mathcal F$,
 so that the operator $m(D)$ is bounded self-adjoint in $L^{2}$ iff $m(\xi)$ is bounded and  
 $m(\xi)^{*}=m(\xi),$
 where 
 $
 m(\xi)^{*}=\overline{\tr{\ m(\xi)}}.
 $
\end{remark}
\begin{lemma}\label{lem.54oi}
 Let $m(D)$ be a Fourier multiplier satisfying the assumptions of Theorem \ref{thm.singsing}. Then if the real part of the matrix $m(\xi)$ is  symmetric and even with respect to $\xi$
and  if the imaginary part of the matrix $m(\xi)$ is antisymmetric and odd with respect to $\xi$, the vector
 $m(D) f$ is real-valued if $f$ is real-valued.
\end{lemma}
\begin{proof}
 Using the formula
\begin{equation}\label{1w58qr}
 \overline{\widehat{f}(\xi)}=\widehat{\overline f}(-\xi),
\end{equation}
we obtain that 
\begin{align*}
2\re\bigl((&m(D) f)(x)\bigr)=
(m(D) f)(x)+\overline{(m(D) f)(x)}
\\
&=\int e^{ix\cdot \xi} m(\xi)\widehat f(\xi)  d\xi (2\pi)^{-d/2}
+\int e^{-ix\cdot \xi} \overline{m(\xi)}\ \widehat{\overline f}(-\xi)  d\xi (2\pi)^{-d/2}
\\
&=\int e^{ix\cdot \xi} m(\xi)\widehat f(\xi)  d\xi (2\pi)^{-d/2}
+\int e^{ix\cdot \xi} \overline{m(-\xi)}\ \widehat{\overline f}(\xi)  d\xi (2\pi)^{-d/2}
\\
\text{\tiny ($f$ real-valued})&=\int e^{ix\cdot \xi} \bigl[m(\xi)+\overline{m(-\xi)}
\bigr]
\widehat f(\xi)  d\xi (2\pi)^{-d/2}.
\end{align*}
Now if $m$ is real-valued and even, we find that 
$
\re\bigl((m(D) f)(x)\bigr)=(m(D) f)(x);
$
if $m$ is purely imaginary and odd, we find the same identity,
since 
$$
\frac12\bigl(m(\xi)+\overline{m(-\xi)}\bigr)=\frac12\bigl(m(\xi)-\overline{m(\xi)}\bigr)= m(\xi).
$$
The proof of the lemma is complete.
\end{proof}
\subsubsection*{\underline{On the Leray projection}}\label{subsec.leray}
We prove in this section  that the range of the Leray projector is 
${L^{2}_{0}}(\R^{3}, \R^{3})$,
defined in Lemma \ref{lem.leray1}.
Let $w\in L^{2}(\R^{3}, \R^{3})$; for a scalar function $\varphi$
in the Schwartz space, we have,
defining $\mathscr S^{*}$ as the space of continuous anti-linear forms on $\mathscr S$\footnote{For $T\in \mathscr S'$, we define 
$
\poscal{T}{\phi}_{\mathscr S^{*}, \mathscr S}=\poscal{T}{\bar\phi}_{\mathscr S^{'}, \mathscr S}.
$
Using the formula
$
\overline{\hat \phi}=\mathcal F^{-1}(\bar \phi),
$
we get in particular
\begin{multline*}
\poscal{\mathcal FT}{\mathcal F \phi}_{\mathscr S^{*}, \mathscr S}
=\poscal{\mathcal FT}{\overline{\mathcal F \phi}}_{\mathscr S^{'}, \mathscr S}
=\poscal{T}{\mathcal F\bigl(\overline{\mathcal F \phi}\bigr)}_{\mathscr S^{'}, \mathscr S}
\\
=\poscal{T}{\mathcal F\bigl(\mathcal F^{-1}\bar \phi\bigr)}_{\mathscr S^{'}, \mathscr S}
=\poscal{T}{\bar \phi}_{\mathscr S^{'}, \mathscr S}
=\poscal{T}{\phi}_{\mathscr S^{*}, \mathscr S},
\end{multline*}
so that it is more convenient to use the complex duality for which we have 
$$
\poscal{\mathcal F^{*}\mathcal FT}{\phi}_{\mathscr S^{*}, \mathscr S}
=
\poscal{\mathcal F^{*}\mathcal FT}{\mathcal F\mathcal F^{-1}\phi}_{\mathscr S^{*}, \mathscr S}
=
\poscal{\mathcal FT}{\mathcal F^{2}\mathcal F^{-1}\phi}_{\mathscr S^{*}, \mathscr S}
=\poscal{\mathcal FT}{\mathcal F\phi}_{\mathscr S^{*}, \mathscr S}
=\poscal{T}{\phi}_{\mathscr S^{*}, \mathscr S},
$$
that is $\mathcal F^{*}\mathcal F=\Id_{\mathscr S^{*}}$ and also
$\mathcal F^{*}=\mathcal F^{-1}$.
}
\begin{align*}
\poscal{\dive\mathbb P w}{\varphi}_{\mathscr S^{*}, \mathscr S}&=-\int \bigl(I_{3}-\val{\xi}^{-2}(\xi\otimes \xi)\bigr)\hat w(\xi)\cdot \overline{ i\xi\hat \varphi(\xi)}d\xi
\\
&=i\int (\hat w(\xi)\cdot \xi)\overline{\hat \varphi(\xi)} d\xi
-i\int \val{\xi}^{-2}
(\hat w(\xi)\cdot \xi) (\xi\cdot \xi)\overline{\hat \varphi(\xi)} d\xi=0,
\end{align*}
proving that $\range \mathbb P\subset L^{2}_{0}(\R^{3}, \R^{3})$.
Conversely, if $w\in L^{2}_{0}(\R^{3}, \R^{3})$,
we have 
$
w=\mathbb P w+\widetilde{\mathbb P} w,
$
where $\widetilde{\mathbb P}$ is defined by \eqref{ortler},
and for a vector field $\Phi$ in the Schwartz space 
\begin{multline*}
\poscal{\widetilde{\mathbb P} w}{\Phi}_{\mathscr S^{*}, \mathscr S}
=\int \val{\xi}^{-2} (\xi\otimes \xi) \hat w(\xi)\cdot \overline{\hat \Phi(\xi)} d\xi
=\int \val{\xi}^{-2} \underbrace{(\xi\cdot\hat w(\xi))}_{=0} (\xi\cdot  \overline{\hat \Phi(\xi)}) d\xi=0,
\end{multline*}
so that $w=\mathbb Pw\in \range \mathbb P$.
\begin{nb}
 Although it is important to consider only \emph{real-valued} vector fields as solutions of the Navier-Stokes equation,
 once you start using the Fourier transformation, the vector field $\hat v$ is no longer real-valued
 (cf. Lemma \ref{lem.fourea}) in general so that it is simpler to use the terminology adapted to  complex vector spaces, in particular for duality and Fourier transformation.
It is better to use the space $\mathscr S^{*}(\R^{3})$ defined as the antidual of 
 $\mathscr S(\R^{3})$ than the more familiar $\mathscr S^{'}(\R^{3})$ to keep the standard formulas related to the Fourier transform such as
 $$
 \mathcal F^{*}\mathcal F=\Id=\mathcal F\mathcal F^{*},
 $$
 and to use the notation
 $
 m(D)= \mathcal F^{*}m(\xi)\mathcal F,
 $
 from which it is clear that real symmetric  even
  (resp. purely imaginary skew-symmetric odd)
  square matrix multipliers are formally self-adjoint.
\end{nb}
\subsection{ Isoperimetric inequality}\label{sec.gani}
\begin{lemma}\label{lem.21ml}
 Let $v$ be a divergence-free vector field such that
  $v\in \mathscr S'(\R^{3}, \R^{3})\cap \mathcal L_{(0)}(\R^{3}, \R^{3})$
and $\curl  v\in L^{2}(\R^{3}, \R^{3})$. Then $v$ belongs to $L^{6}(\R^{3}, \R^{3})$.
 \end{lemma}
 \begin{proof}
 We have seen already in \eqref{drs128} that 
 \begin{equation*}
\forall q_{1}\in (1,\frac65), \forall q_{2}\in (\frac65, 2], \quad
\hat v\in L^{q_{1}}+L^{q_{2}},\quad
v\in L^{q'_{1}}+L^{q'_{2}},
\end{equation*}
so that 
\begin{equation}\label{gcv159}
v=\rot^{2}\val{D}^{-2}v=\val{D}^{-1}\underbracket[0.1pt]{\rot\val{D}^{-1}}_{\substack{
\text{singular}\\\text{integral}
}}\underbrace{\rot v}_{\in L^{2}},
\end{equation}
and using the Hardy-Littlewood-Sobolev Inequality,
we get that 
$$
\val{D}^{-1}: L^{2}(\R^{3})\longrightarrow L^{6}(\R^{3}),
$$
since $\val{D}^{-1}$ is the convolution with $\val x^{1-3}=\val x^{-\frac{3}{3/2}}$ so that
$$
\norm{\val x^{-2}\ast  u}_{L^{6}(\R^{3})}\le \sigma\norm{u}_{L^{2}(\R^{3})}, \quad \text{ since }
\frac12+\frac23=1+\frac16.
$$
As a consequence of \eqref{gcv159}, since  $\rot\val{D}^{-1} \rot v\in L^{2}$, we get 
$v\in L^{6}$ and 
$$
\norm{v}_{L^{6}}=\norm{\val{D}^{-1}\rot\val{D}^{-1}\rot v}_{L^{6}}
\le \sigma\norm{\rot\val{D}^{-1}\rot v}_{L^{2}}\le \sigma\norm{\rot v}_{L^{2}},
$$
and the result of the lemma.
\end{proof}
The following theorem is classical.
\begin{theorem}[Gagliardo-Nirenberg Isoperimetric Inequality]\label{thm.gani01}
Let $f$ be a scalar function in $L^{1}(\R^{d})$, where $d$ is an integer larger than $1$. 
If   $\nabla f$  belongs to $L^{1}(\R^{d})$, then $f$ belongs to $L^{\frac{d}{d-1}}(\R^{d})$ 
and we have 
\begin{equation}\label{}
d\val{\mathbb B^{d}}^{\frac 1d} \norm{f}_{L^{\frac{d}{d-1}}(\R^{d})}\le \norm{\nabla f}_{L^{1}(\R^{d})}.
\end{equation}
We note that 
$
\val{\mathbb B^{d}}=\frac{\pi^{\frac d2}}{\Gamma(\frac d2+1)}.
$
 \end{theorem}
 A standard corollary of the above theorem is the following result.
 \begin{corollary}\label{cor.350}
 Let $d$ be an integer $\ge 2$ and let $1\le p<d$ be a real number. We define the $d$-dimensional conjugate exponent $p^{*}$ of $p$ as
 \begin{equation}\label{conexp}
 p^{*}=\frac{pd}{d-p},\quad\text{i.e.\quad}\frac{1}{p^{*}}=\frac{1}p-\frac 1d.
\end{equation}
Then, there exists a positive constant $c(p,d)$ such that
 for all $f\in L^{\frac{p^{*}(d-1)}{d}}(\R^{d})$ such that $\nabla f\in L^{p}(\R^{d})$, we have 
 $f\in L^{p^{*}}(\R^{d})$ and 
\begin{equation}\label{}
c_{p,d} \norm{f}_{L^{ p^{*}}(\R^{d})}\le \norm{\nabla f}_{L^{p}(\R^{d})}.
\end{equation}
 \end{corollary}
 \no Footnote \ref{foot.lieb} on page \pageref{foot.lieb} provides a reference for the expression of the $c_{p,d}$.
 \begin{remark}\rm
 If $\dis d=3,p=\frac32$, we get $p^{*}=3$ and from the above corollary,
\begin{align}\label{}
 &f\in L^{2}(\R^{3}), \nabla f \in L^{3/2}(\R^{3})\Longrightarrow f\in L^{3}(\R^{3}),
 \quad \text{as well as }
 \\
 &
 c_{\frac32,3} \norm{f}_{L^{3}(\R^{3})}\le \norm{\nabla f}_{L^{3/2}(\R^{3})}.
\end{align}
\end{remark}
Following the book \cite{MR2768550} of H.~Bahouri, J.-Y.~Chemin, R.~Danchin,
we recall a few definitions and results (see Section 1.3 of \cite{MR2768550}).
\begin{definition}
 Let $d$ be a positive integer and let  $s$ be a real number. We define the \emph{Homogeneous Sobolev Space} $\dot H^{s}(\R^{d})$ as
 \begin{equation}\label{}
\{u\in \mathscr S'(\R^{d}), \hat u\in L^{1}_{\text{\rm loc}}(\R^{d}), 
\norm{u}_{\dot H^{s}(\R^{d})}^{2}:=
\int_{\R^{d}}\val{\xi}^{2s}\val{\hat u(\xi)}^{2} d\xi<+\io.
\}
\end{equation}
\end{definition}
\begin{nb}
 Note that if a constant function $u$ belongs to some $\dot H^{s}(\R^{d})$, then the condition
 $\hat u \in L^{1}_{\text{\rm loc}}(\R^{d})$ implies $u=0$.
\end{nb}
\begin{proposition} Let $d$ be a positive integer and let  $s$ be a real number.
\begin{enumerate} 
\item If $s<d/2$, $\dot H^{s}(\R^{d})$ is a Hilbert space.
\item If $s<d/2$, the space $\mathscr S_{0}(\R^{d})=\{\phi \in \mathscr S(\R^{d}), 0\notin \supp \hat \phi\}$ is dense in $\dot H^{s}(\R^{d})$.
\item If $\val s<d/2$, the dual space of $\dot H^{s}(\R^{d})$ is $\dot H^{-s}(\R^{d})$.
\end{enumerate} 
\end{proposition}
\begin{nb}
 Note that the condition $0\notin \supp \hat \phi$ is equivalent to the vanishing of $\hat \phi$ in a neighborhood of $0$.
\end{nb}
\begin{theorem}
  Let $d$ be a positive integer and let  $s$ be a real number.
   If $0\le s<d/2$, the space $\dot H^{s}(\R^{d})$ is continuously embedded into 
   $L^{\frac{2d}{d-2s}}(\R^{d}).$
\end{theorem}
\begin{remark}\label{rem.314}\rm
 The above theorem for $d=3, s=1$ implies that $\dot H^{1}(\R^{3})$
  is continuously embedded into 
   $L^{6}(\R^{3}).$
\end{remark}
\begin{proposition}\label{app.pro.dq99}
 Let $X$ be a real vector field 
in $\R^{d}$ ($d\ge 3$) with $\moo$ coefficients in $L^{d}(\R^{d})$
 such that  that $(\dive X)_{+}$ belongs to $L^{\frac d2}(\R^{d})$ such that 
 \begin{equation}\label{hqw159}
 \norm{(\dive X)_{+}}_{L^{\frac d 2}(\R^{d})}<2\nu \sigma(d),
\end{equation}
where $\sigma(d)$ is a dimensional constant defined as  the largest constant such that,
for all $\phi$ smooth compactly supported,
 \begin{equation}\label{fsc698}
 \sigma(d)\norm{\phi}_{L^{\frac{2d}{d-2}}(\R^{d})}^{2}\le \norm{\nabla \phi}_{L^{2}(\R^{d})}^{2}.
\end{equation}
Let $\nu$ be a positive parameter and $f$ be a solution of
\begin{equation}\label{equ+01}
-\nu \Delta f+X f=0
\end{equation}
such that  \eqref{lio002}, \eqref{lio003} hold true. Then $f$ is identically 0.
\end{proposition}
\begin{proof} We follow the proof of Proposition \ref{pro.dq98}, incorporating the results of Footnote \ref{foot02}
on page \pageref{foot02}: we find 
 \begin{align*}
0&=\poscal{-\nu\Delta f+Xf}{\chi_{\lambda}f}\notag\\
&=\nu\int \chi_{\lambda}(x)\val{(\nabla f)(x)}^{2}dx
+\nu\poscal{\underbrace{\nabla f}_{L^{2}}}{\underbrace{(\nabla \chi_{\lambda})}_{L^{d}}\underbrace {f}_{L^{q}}}
-\frac12\poscal{\underbrace{X\cdot (\nabla\chi_{\lambda})}_{L^{p}\cdot L^{d}} }
{\underbrace{f^{2}}_{L^{q/2}}}
\\
&\hskip232pt - \frac12\int_{\R^{\mathtt  d}} f(x)^{2}(\dive X)(x)\chi_{\lambda}(x) dx,
\end{align*}
so that with the notations \eqref{hgf541}, \eqref{hgf54+}, we obtain
\begin{multline*}
\nu\int \chi_{\lambda}(x)\val{(\nabla f)(x)}^{2}dx+\nu \varepsilon_{1}(\lambda)-\frac12\varepsilon_{2}(\lambda)
 +\frac12\int_{\R^{\mathtt  d}} f(x)^{2}(\dive X)_{-}(x)\chi_{\lambda}(x) dx,
\\=
  \frac12\int_{\R^{\mathtt  d}} f(x)^{2}(\dive X)_{+}(x)\chi_{\lambda}(x) dx,
\end{multline*}
which implies
\begin{multline*}
\nu\int \chi_{\lambda}(x)\val{(\nabla f)(x)}^{2}dx+\nu\varepsilon_{1}(\lambda)-\frac12\varepsilon_{2}(\lambda)
\le 
  \frac12\int_{\R^{\mathtt  d}} f(x)^{2}(\dive X)_{+}(x)\chi_{\lambda}(x) dx
  \\
  \le \frac12\norm{f}_{L^{\frac{2d}{d-2}}(\R^{\mathtt d})}^{2}
  \norm{(\dive X)_{+}}_{L^{\frac{d}{2}}(\R^{\mathtt d})}.
\end{multline*}
Taking now the limit of the right-hand-side for $\lambda$ tending to $+\io$, we get
either $f=0$ or $\norm{f}_{L^{\frac{2d}{d-2}}(\R^{\mathtt d})}>0$, and 
\begin{multline*}
\nu\sigma(d)\norm{f}_{L^{\frac{2d}{d-2}}(\R^{\mathtt d})}^{2}\le  \nu\int\val{(\nabla f)(x)}^{2}dx
\\\le 
   \frac12\norm{f}_{L^{\frac{2d}{d-2}}(\R^{\mathtt d})}^{2}
  \norm{(\dive X)_{+}}_{L^{\frac{d}{2}}(\R^{\mathtt d})}<\nu \sigma(d)\norm{f}_{L^{\frac{2d}{d-2}}(\R^{\mathtt d})}^{2},
\end{multline*}
which is impossible, entailing $f=0$, concluding the proof.
\end{proof}
\begin{remark}\rm 
 Assume that our vector field $X=\sum_{1\le j\le d}a_{j}(x)\frac{\p}{\p x_{j}}$, with $[a_{j}]= \mathtt L$.
 We find then that with $\nu>0$ such that $[\nu]=\mathtt{L^{2}}$
 $$
 [\nu \Delta f]=\mathtt{ L^{2}L^{-2}}[f], \ [X f]=\mathtt{LL^{-1}}[f], \ [\dive X]=\mathtt{L^{-1}L},\
 [\norm{\dive X}_{L^{\frac d2}}^{\frac d2}]=\mathtt{(1)^{\frac d2}L^{d}}=\mathtt {L^{d}}, 
 $$
 so it is legitimate to compare the numbers with units $\mathtt{L^{2}}=[\norm{(\dive X)_{+}}_{L^{\frac d 2}(\R^{d})}]$
and $\nu \sigma(d)$ since $\sigma(d)$ has no dimension. Indeed, we may note that the inequality \eqref{fsc698} is units-consistent since
$$
\bigl[\norm{f}_{L^{\frac{2d}{d-2}}(\R^{d})}^{2}\bigr]=([f]^{\frac{2d}{d-2}}\mathtt{ L^{d})^{\frac{d-2}d}}=[f]^{2}\mathtt {L^{d-2}}
,\quad
 \bigl[\norm{\nabla f}_{L^{2}(\R^{d})}^{2}\bigr]=(\mathtt{ L^{-1}}[f])^{2}\mathtt {L^{d}}
 =[f]^{2}\mathtt {L^{d-2}}.
$$ 
Moreover, we have a \emph{scaling property} for that simple linear equation: if $f$ solves
\eqref{equ+01}, then  for $\lambda>0$, $f_{\lambda}$ defined by $f_{\lambda}(x)=\lambda^{\mu}f(\lambda x) $ solves the equation
$$
-\nu \Delta f_{\lambda}+X_{\lambda}f_{\lambda}=0, \quad X_{\lambda}(x)=\lambda X(\lambda x).
$$
Indeed, we have 
\begin{multline*}
-\nu (\Delta f_{\lambda})(x)+X_{\lambda}(x) \cdot (df_{\lambda})(x)=
-\nu
\lambda^{\mu+2}(\Delta f)(\lambda x)+\lambda X(\lambda x)\cdot \lambda^{\mu+1}(df)(\lambda x)
\\
\text{\footnotesize (with $y=\lambda x$)\quad}=\lambda^{\mu+2}\bigl[-\nu (\Delta f)(y)+X(y)\cdot (df)(y)\bigr]=0.
\end{multline*}
We note also that $\norm{X}_{L^{d}}=\norm{X_{\lambda}}_{L^{d}}$, since we have 
$$
\norm{X_{\lambda}}_{L^{p}}=\Bigl(\int \val{X(\lambda x)}^{p}\lambda^{p} dx\Bigr)^{\frac 1p}=
\Bigl(\int \val{X(y)}^{p}\lambda^{p-d} dy\Bigr)^{\frac 1p}=\lambda^{1-\frac dp}\norm{X}_{L^{p}}.
$$
As a consequence, it is natural to have a condition such as $X\in L^{d}(\R^{d})$,
which we call a condition \emph{at the scaling}: 
indeed, we have 
$$
\norm{X_{\lambda}}_{L^{\dd}(\R^{\dd})}=\Bigl(\int_{\R^{\dd}} \norm{\lambda X(\lambda x)}^{\dd}_{\R^{\dd}} dx\Bigr)^{1/\dd}
=\norm{X}_{L^{\dd}(\R^{\dd})}.
$$
We may note as well that 
\begin{multline*}
\norm{(\dive X_{\lambda})_{+}}_{L^{\frac d 2}(\R^{d})}^{\frac d2}=\int \val{(\dive X_{\lambda})_{+}(x)}^{\frac d2} dx=
\int \val{\lambda^{2}(\dive X)_{+}(\lambda x)}^{\frac d2} dx
\\
=\lambda^{d}\int \val{(\dive X)_{+}(y)}^{\frac d2} dy \lambda^{-d}=\norm{(\dive X)_{+}}_{L^{\frac d 2}(\R^{d})}^{\frac d2},
\end{multline*}
so that Condition \eqref{hqw159} is unchanged,
enhancing its relevance.
\end{remark}
\subsection{ Some entire solutions of the Navier-Stokes system}\label{sec.dg4+}
\begin{lemma}
 Let $\psi$ be a real-valued harmonic function on $\R^{3}$ and let $v=\nabla \psi$. Then with $Q=0$, the vector field $v$ is a solution  of the stationary Navier-Stokes system for incompressible fluids i.e. we have 
 \begin{align}
&\nu \rot^{2}v+\rot v\times v+\nabla Q=0, \quad \dive v=0,\quad \text{or equivalently,}
\label{tas569}\\
&-\nu \Delta v+\p_{j}(v_{j}v)+\nabla p=0, \quad p=-\frac{\val v^{2}}2, \quad \dive v=0.
\label{tps569}\end{align}
Moreover we have $\curl v=0$.
\end{lemma}
\begin{proof}
 This is proven in Remark \ref{rem.54kj}. Note that $\curl \nabla \psi=0$ implies the last property
 which implies as well that $\curl v\in L^{2}$. The proof is shorter if we use the equation \eqref{tas569}: indeed, since $\rot v=0$, choosing $Q=0$, we get 
 $$
 \nu \rot^{2}v+\rot v\times v+\nabla Q=-\nu \Delta \nabla \psi=-\nu \nabla \Delta \psi=0, \quad 
 \dive v=\Delta\psi=0,
 $$
 proving the lemma.
 \end{proof}
 \begin{remark}\rm
 Of course, there exist many  $\psi$ harmonic such that $v=\nabla \psi$ is not the zero function;
 taking $\psi$ as an harmonic polynomial with degree $\ge 2$ provides some examples. Also in dimension 3 , we may consider the harmonic function
 $$
\psi(x_{1}, x_{2}, x_{3})= e^{x_{1}}\sin x_{2}, \quad v=e^{x_{1}}\sin x_{2}\frac{\partial}{\p x_{1}}
+e^{x_{1}}\cos x_{2}\frac{\partial}{\p x_{2}},\quad p=-\frac{e^{2x_{1}}}2,
 $$
 so that $v$ solves the stationary Navier-Stokes system. The latter example is not a temperate distribution
 and does not belong either to $\mathcal L_{(0)}$. 
\end{remark}
\begin{lemma}\label{lem.jhg4}
  Let $\psi$ be a real-valued harmonic function on $\R^{3}$ and let $v=\nabla \psi$. If we assume that $v$ belongs to $\mathscr S'\cap \mathcal L_{(0)}$ (cf. \eqref{l0inft}), this implies that $v=0$.
\end{lemma}
\begin{proof}
 Indeed $v$ is an harmonic vector field belonging to $\mathscr S'$,
 thus its Fourier transform is supported in $\{0\}$, which implies that $v$ is a polynomial. Since $v$ belongs as well to $\mathcal L_{(0)}$, Lemma \ref{cla.548uyt} implies that $v=0$.
 \end{proof}
\subsection{ On some Identities and Inequalities}\label{sec.dge5}
\subsubsection*{\underline{Divergence of a vector product}}
\begin{lemma}\label{lem.kj44}
 Let $a\in L^{p}(\R^{3}, \R^{3})$, $b\in L^{q}(\R^{3}, \R^{3})$ such that
 $\curl a\in L^{r}(\R^{3}, \R^{3})$, \quad $\curl b\in L^{s}(\R^{3}, \R^{3})$,
 with $p, r\in (1,+\io)$ and 
 \begin{equation}\label{po55cx}
 \frac1p+\frac1q=\frac1{t_{1}}
\le 1, \quad \frac1s+\frac1p=\frac{1}{t_{2}}\le 1,
\quad \frac1r+\frac1q=\frac1{t_{3}}\le 1.
\end{equation}
Then $a\times b\in L^{t_{1}}(\R^{3}, \R^{3}))$, $(\curl a)\cdot b \in L^{t_{3}}(\R^{3}, \R^{3}))$,
$(\curl b)\cdot a\in L^{t_{2}}(\R^{3}, \R^{3})$
and we have 
\begin{equation}\label{aqw741}
\dive(a\times b)=(\curl a)\cdot b-(\curl b)\cdot a.
\end{equation}
\end{lemma}
\begin{nb}
 Note that, from Conditions \eqref{po55cx}, all the products involved in Formula \eqref{aqw741} make sense.
\end{nb}
\begin{proof}
Let $\phi\in \mooc(\R^{3}, \R)$. On the one hand,
we have, 
$$
\poscal{\dive(a\times b)}{\phi}_{\mathscr D', \mathscr D}=-\int_{\R^{3}}
\proscal3{a\times b}{\nabla \phi}dx=
\int_{\R^{3}} \det\bigl(b,a, \nabla \phi\bigr)            dx,
$$
and on the other hand
\begin{align*}
\poscal{(\curl a)&\cdot b}{\phi}_{\mathscr D', \mathscr D}
-\poscal{a\cdot(\curl b) }{\phi}_{\mathscr D', \mathscr D}
\\
&=\int_{\R^{3}}\proscal3{\curl a}{\phi b}dx-\int_{\R^{3}}\proscal3{\curl b}{\phi a}dx
\\
\text{\tiny (using \eqref{gcqf58} below)}&=\int_{\R^{3}}\proscal3{a}{\curl(\phi b)}dx-\int_{\R^{3}}\proscal3{\curl b}{\phi a}dx
\\
&=\int_{\R^{3}}\proscal3{a}{\phi \curl b}
+\proscal3{a}{\nabla\phi\times b}
dx-\int_{\R^{3}}\proscal3{\curl b}{\phi a}dx
\\
&=\int_{\R^{3}}\det\bigl(a, \nabla \phi, b\bigr) dx=\int_{\R^{3}}\det\bigl(b,a, \nabla \phi\bigr) dx,
\end{align*} 
yielding the sought result.
Note that we have used along our way that
\begin{equation}\label{gcqf58}
\poscal{\curl a}{\phi b}_{L^{r}, L^{r'}}=\poscal{a}{\curl(\phi b)}_{L^{p}, L^{p'}},
\end{equation}
which makes sense since  $\phi b\in L^{r'}$ using that  $\phi$ is smooth compactly supported
and $b$ in $L^{q}$ with $q\ge r'$ from the inequality
$$
\frac1r
+\frac1{q}\le 1=\frac1r+\frac1{r'}.$$
Also, we have that 
$\curl(\phi b)\in L^{p'}$
since $\phi$ is smooth compactly supported
and 
\begin{equation}\label{knxc11}
\curl(\phi b)=\phi \underbracket[0.1pt]{\curl b}_{L^{s}}+\nabla \phi\times\underbracket[0.1pt]{ b}_{L^{q}},
\end{equation}
with $s\ge p'$ from the inequality
$
\frac1s+\frac1p\le 1=\frac1{p'}+\frac1p,
$
as well as $q\ge p'$ from
$$
\frac1q+\frac1p\le 1=\frac1{p'}+\frac1p.
$$
To prove \eqref{knxc11} is quite straightforward and is true even for $b$ distribution vector field: we just have to check for a distribution $T$ that 
$
\p_{j}(\phi T)=T\p_{j}(\phi)+\phi\p_{j}T,
$
which is obvious and classical.
Let us prove \eqref{gcqf58}.
Using \eqref{knxc11} (now proven), we get that
\begin{equation}\label{dx54jn}
\poscal{a}{\curl(\phi b)}_{L^{p}, L^{p'}}=\poscal{a\phi}{\curl b}_{L^{p}, L^{p'}}+\int_{\R^{3}}\det\bigl(\nabla \phi, b, a\bigr) dx.
\end{equation}
Since $a\in L^{p}$, we can use a standard mollifier $\rho_{\varepsilon}$ and write that 
\begin{multline}\label{tsq459}
\poscal{a\phi}{\curl b}_{L^{p}, L^{p'}}=\lim_{\varepsilon\rightarrow 0}\poscal{(a\ast \rho_{\varepsilon})\phi}{\curl b}_{L^{p}, L^{p'}}
=\lim_{\varepsilon\rightarrow 0}\bigl\{\poscal{\curl b}{\phi(a\ast \rho_{\varepsilon})}_{\mathscr D', \mathscr D}\bigr\}
\\
=\lim_{\varepsilon\rightarrow 0}\bigl\{\poscal{b}{
\underbrace{\nabla \phi\times (a\ast \rho_{\varepsilon})}_{\substack{
\text{converges in $L^{p}$}\\\text{to $\nabla\phi\times a,$}\\\text{compactly supported}
\\\text{ in $\supp\nabla \phi$}
}}}_{\mathscr D', \mathscr D}
+\poscal{b}{\phi (\curl a\ast \rho_{\varepsilon})}_{\mathscr D', \mathscr D}
\bigr\}
\end{multline}
We can thus replace $b$ by $\chi b$ where $\chi$ is compactly supported,
equal to $1$ on a neigborhood of the support of $\phi$.
Here we have 
$$
\frac1p+\frac1q\le 1=\frac1p+\frac1{p'}\Longrightarrow q\ge p',
$$
so that 
\begin{multline}\label{98traq}
\lim_{\varepsilon\rightarrow 0}\poscal{b}{\nabla \phi\times (a\ast \rho_{\varepsilon})}_{\mathscr D',\mathscr D}
\\=\lim_{\varepsilon\rightarrow 0}\poscal{\chi b}{\nabla \phi\times (a\ast \rho_{\varepsilon})}_{\mathscr D',\mathscr D}
=\lim_{\varepsilon\rightarrow 0}\poscal{\chi b}{\nabla \phi\times (a\ast \rho_{\varepsilon})}_{L^{p'},L^{p}}
\\
=\poscal{\chi b}{\nabla \phi\times a}_{L^{p'},L^{p}} 
=\poscal{b}{\nabla \phi\times a}_{L^{p'},L^{p}}  =\int_{\R^{3}}\det(\nabla \phi, a, b)dx.
\end{multline}
We have also 
\begin{equation}\label{gf54bb}
\poscal{b}{\phi (\curl a\ast \rho_{\varepsilon})}_{\mathscr D', \mathscr D}
=
\poscal{\phi b}{\curl a\ast \rho_{\varepsilon}}_{\mathscr D', \mathscr D}=
\poscal{\phi b}{\curl a\ast \rho_{\varepsilon}}_{L^{r'}, L^{r}}
\end{equation}
since 
$\phi b\in L^{r'}$ using that $\phi$ is smooth compactly supported
and $b$ in $L^{q}$ with $q\ge r'$ from the inequality
$$
\frac1r
+\frac1{q}\le 1=\frac1r+\frac1{r'}.$$
As a consequence, we get from \eqref{tsq459}, \eqref{98traq}, \eqref{gf54bb} that 
\begin{equation}\label{cxw159}
\poscal{a\phi}{\curl b}_{L^{p}, L^{p'}}=\int_{\R^{3}}\det(\nabla \phi, a, b)dx
+\poscal{\curl a}{\phi b}_{L^{r}, L^{r'}},
\end{equation}
and 
we find that 
$$
\poscal{a}{\curl(\phi b)}_{L^{p}, L^{p'}}\underbracket[0pt]{=}_{\eqref{dx54jn}}\int_{\R^{3}}\det\bigl(\nabla \phi, b, a\bigr) dx
+\poscal{a\phi}{\curl b}_{L^{p}, L^{p'}}\underbracket[0pt]{=}_{\eqref{cxw159}}
\poscal{\curl a}{\phi b}_{L^{r}, L^{r'}},
$$
which proves \eqref{gcqf58}.
\end{proof}
\begin{remark}\rm
 Let $v$ be a vector field in $L^{6}(\R^{3}, \R^{3})$, such that $\curl v, \curl^{2}v\in L^{2}(\R^{3}, \R^{3})$.
 Taking $a=\curl v, b=v$, 
 $p=2, q=6$ we have $\curl a=\curl^{2}v\in L^{2}$, $r=2$, $\curl b=a$, $s=2$.
 Now the conditions \eqref{po55cx} are satisfied. Lemma \ref{lem.kj44} implies 
 \begin{equation}\label{4edsw6}
\dive (\curl v\times v)=\proscal3{\curl^{2}v}{v}-\norm{\curl v}_{\R^{3}}^{2}.
\end{equation}
\end{remark}
\subsubsection*{\underline{Proof of Remark \ref{rem.kj87}}}
Let $s\in \N$. We claim that 
\begin{equation}\label{ytr147}
\mathcal V^{s,\omega}=\{v\in \mathcal W, \forall \alpha\in \N^{d} \text{ with }\val{\alpha}\le s, D_{x}^{\alpha}v\in \mathcal W\}.
\end{equation}
Since for $\val \alpha\le s$, we have  
$
\int_{\R^{d}}\val{\xi^{\alpha}}\val{\hat v(\xi)}d\xi\le \int_{\R^{d}}\valjp{\xi}^{s}\val{\hat v(\xi)}d\xi,
$
we get that $\mathcal V^{s,\omega}$ is included in the right-hand-side of \eqref{ytr147}. Conversely, since we have
$$
\R^{d}=\cup_{1\le j\le d}\{\xi\in \R^{d}, {\xi_{j}}^{2}\ge \frac{\val \xi^{2}}{d}\},
$$
it is enough to consider for $s\in \N^{*}$,
\begin{multline}\label{874jjh}
\int_{\{\xi,\ \xi_{1}^{2}\ge \frac{\val \xi^{2}}{d}\}}\hskip-8pt \valjp{\xi}^{s}\val{\hat v(\xi)}d\xi
\le \hskip-5pt \int_{\R^{d}}\bigl(1+d\xi_{1}^{2}\bigr)^{s/2}\val{\hat v(\xi)}d\xi
\le 
\int_{\R^{d}}\bigl(1+d^{\frac12}\val{\xi_{1}}\bigr)^{s}\val{\hat v(\xi)}d\xi
\\\underbracket[0pt]{\le}_{\substack{
\text{cf. \eqref{gwp197}}\\\text{below}
}}
\int_{\R^{d}}\bigl(1+d^{\frac s2}\val{\xi_{1}}^{s}\bigr)2^{s-1}\val{\hat v(\xi)}d\xi
= 2^{s-1}\norm{v}_{\mathcal W}
+2^{s-1}d^{\frac s2} \norm{D_{1}^{s}v}_{\mathcal W},
\end{multline}
proving \eqref{ytr147}.
We have left to the reader the proof that 
\begin{equation}\label{gwp197}
\text{For $s\ge 1, a\ge 0$, we have, }
\quad
(1+a)^{s}\le 2^{s-1}(1+a^{s}),
\end{equation}
useful for the third inequality in \eqref{874jjh}.
\subsubsection*{\underline{On Peetre's Inequality}}
\begin{claim}\label{cla.214jhg}
We consider the set $\mathcal T$ of nonnegative real numbers $\tau$ such that for all $\xi_{1},\xi_{2}\in \R^{d}$, we have
\begin{equation}\label{}
\tau+\val{\xi_{1}+\xi_{2}}^{2}\le\bigl(\tau+\val{\xi_{1}}^{2}\bigr)
\bigl(\tau+\val{\xi_{2}}^{2}\bigr).
\end{equation}
Then we have 
$
\mathcal T=\bigl[\frac43, +\io\bigr).
$
\end{claim}
\begin{proof}[Proof of the Claim.]
We have for $a,b\ge 0$,
\begin{multline*}
\bigl(\frac43+a^{2}\bigr)\bigl(\frac43+b^{2}\bigr)-\frac43-(a+b)^{2}
=\frac49+a^{2}b^{2}+\frac{a^{2}+b^{2}}3-2ab
\\=
\frac49+a^{2}b^{2}+\frac{a^{2}+b^{2}-2ab}3-\frac43ab
=a^{2}b^{2}-\frac43ab+\frac49+\frac{(a-b)^{2}}3
\\
=\Bigl(ab-\frac23\Bigr)^{2}+\frac{(a-b)^{2}}3\ge 0.
\end{multline*}
This implies that for all $\xi_{1},\xi_{2}\in \R^{d}$,
\begin{equation}\label{rea489}
\frac43+\val{\xi_{1}+\xi_{2}}^{2}\le
\frac43+\bigl(\val{\xi_{1}}+\val{\xi_{2}}\bigr)^{2}\le\bigl(\frac43+\val{\xi_{1}}^{2}\bigr)
\bigl(\frac43+\val{\xi_{2}}^{2}\bigr),
\end{equation}
so that $4/3\in \mathcal T$.
If $\tau\ge 4/3$,
we find for all $\xi_{1},\xi_{2}\in \R^{d}$,
\begin{multline*}\label{}
\tau+\val{\xi_{1}+\xi_{2}}^{2}\le
\tau-\frac43+\frac43+\bigl(\val{\xi_{1}}+\val{\xi_{2}}\bigr)^{2}
\underbracket[0pt]
{\le}_{\substack  {  \text{using} \\\eqref{rea489}    }           }
\bigl(\frac43+\val{\xi_{1}}^{2}\bigr)
\bigl(\frac43+\val{\xi_{2}}^{2}\bigr)+\tau-\frac43
\\
\le\bigl(\frac43+\val{\xi_{1}}^{2}\bigr)
\bigl(\frac43+\val{\xi_{2}}^{2}\bigr)+(\tau-\frac43)(\frac43+\val{\xi_{1}}^{2})
=\bigl(\frac43+\val{\xi_{1}}^{2}\bigr)
\bigl(\tau+\val{\xi_{2}}^{2}\bigr)
\\
\le \bigl(\tau+\val{\xi_{1}}^{2}\bigr)
\bigl(\tau+\val{\xi_{2}}^{2}\bigr),
\end{multline*}
proving that $\mathcal T\supset\bigl[\frac43, +\io\bigr).$
Moreover, for $\tau\in \mathcal T$,
we find, choosing $\xi_{1}=\xi_{2}, \val{\xi_{1}}=\sqrt{\frac23}$, that
$
\tau+\frac83\le \bigl(\tau+\frac23\bigr)^{2},
$ 
i.e.
$$
0\le P(\tau)=\tau^{2}+\frac{\tau}3-\frac{20}9=(\tau-\frac43)(\tau+\frac53).
$$
Since we have $\tau\ge0$, this implies $\tau\ge 4/3$,
completing the proof of the claim.\end{proof}
\begin{nb}
It follows from the previous calculations that the best choice for the definition of $\valjp{\xi}$ in \eqref{233233} would be 
 $
 \valjp{\xi}=\bigl(\frac43+\val\xi^{2}\bigr)^{1/2}.
 $
 However, it does not seem useful to work with the best constant here,
although we \emph{do need} the Banach algebra property
 $
 \norm{v_{1}v_{2}}_{\mathcal V^{s,\omega}}\le  \norm{v_{1}}_{\mathcal V^{s,\omega}}\norm{v_{2}}_{\mathcal V^{s,\omega}}.
 $
\end{nb}
\def\cprime{$'$}

\end{document}